\documentclass{amsart}
\usepackage{mathtools,amsmath,amssymb,amsthm,array,etoolbox}
\usepackage{color}
\usepackage{cases}

\numberwithin{equation}{section}

\newtheorem{theorem}{Theorem}[section]
\newtheorem{lemma}[theorem]{Lemma}
\newtheorem{corollary}[theorem]{Corollary}

\theoremstyle{definition}

\newtheorem{proposition}[theorem]{Proposition}
\theoremstyle{remark}
\newtheorem{remark}[theorem]{Remark}

\newcommand{\R}{\mathbb{R}}
\newcommand{\E}{\mathbf{E}}

\newcommand{\BH}{\mathbb{H}}

\newcommand{\B}{\mathbf{B}}

\newcommand{\p}{\partial}
\newcommand{\ov}{\overline}
\newcommand{\be}{\begin{equation}}
	\newcommand{\ee}{\end{equation}}
\newcommand{\bse}{\begin{subequations}}
	\newcommand{\ese}{\end{subequations}}
\def\bea{\begin{eqnarray}}
	\def\eea{\end{eqnarray}}

\begin{document}

\title[Optimal Lower Bound for the Blow-up Rate]{Optimal Lower Bound for the Blow-up Rate of the Magnetic Zakharov System without the Skin Effect}
 \author{Zaihui Gan$^{1}$ , Yuchen Wang$^{2}$ , Yue Wang$^{1}$, Jialing Yu$^{1,3}$ }
\address{${}^1$ Center for Applied Mathematics, Tianjin University, \\ Tianjin 300072, China.}
\email{ganzaihui2008cn@tju.edu.cn.}
\email{wangy2017@tju.edu.cn.}
\address{${}^2$ School of Mathematics and Statistics, Central China Normal University\\ Wuhan 430079, Hubei, China}
\email{wangyuchen@mail.nankai.edu.cn}
\thanks{Zaihui Gan$^{1}$ is partially supported by the National Science Foundation of China No.11571254, \\
 Yuchen Wang$^{2}$ is partially supported by the National Science Foundation of China No.11831009 and Hubei Province Science and Technology Innovational Funding. }
 \address{${}^3$ Mathematisch Instituut-Universiteit Leiden, P.O. Box 9512, \\2300 RA Leiden, The Netherlands.}
\email{j.yu@math.leidenuniv.nl}
\date{}

\maketitle

\begin{abstract}
%\noindent
 We focus on the Cauchy problem of the magnetic Zakharov system in two-dimensional space:
$$
\left\{
\begin{array}{ll}
&\ i E_{1t}+\Delta E_1-n E_1+\eta E_2\left(E_1\overline{E_2}-\overline{E_1}E_2\right)=0,
\\[0.3cm]
&\ i E_{2t}+\Delta E_2-n E_2+\eta E_1\left(\overline{E_1}E_2-E_1\overline{E_2}\right)=0,
\\[0.3cm]
&\ n_t+\nabla\cdot\textbf{v}=0,
\\[0.3cm]
&\ \textbf{v}_t+\nabla n+\nabla\left(|E_1|^2+|E_2|^2\right)=0,
\\[0.3cm]
&\ (E_1,E_2,n,\textbf{v})(0,x)=(E_{10},E_{20},n_{0},\textbf{v}_{0})(x),
\end{array}
\right.\eqno(G-Z)
 $$
 which describes the spontaneous generation of a magnetic field without the skin effect in a cold plasma, where
$\eta>0$ is a physical constant coefficient. The two nonlinear terms~$E_2\left(E_1\overline{E_2}-\overline{E_1}E_2\right)$
 and $E_1\left(\overline{E_1}
E_2-E_1\overline{E_2}\right)$ generated by the cold magnetic field
 bring in a different difficulty from that for the classical Zakharov system.
Assuming the initial mass satisfies the following estimates:
\begin{gather*}
\frac{||Q||_{L^2(\mathbb{R}^2)}^2}{1+\eta}
<||E_{10}||_{L^2(\mathbb{R}^2)}^2+||E_{20}||_{L^2(\mathbb{R}^2)}^2
<\frac{||Q||_{L^2(\mathbb{R}^2)}^2}{\eta},
\end{gather*}
where $Q$ is the unique radially positive solution of the equation
$
-\Delta V+V=V^3 $, we prove that
there is a constant~$c>0$~ depending only on the initial data such that for~$t$ near $T$ (the blow-up time),
\begin{gather*}
\left\|\left(E_1,E_2,n,\textbf{v}\right)\right\|_{H^1(\mathbb{R}^2)\times H^1(\mathbb{R}^2)\times L^2(\mathbb{R}^2)\times
L^2(\mathbb{R}^2)}\geqslant\frac{c}{ T-t }.
\end{gather*}
As the magnetic coefficient $\eta$ tends to $0$, the blow-up rate recovers the result for the classical 2-D Zakharov system due to Merle \cite{25Frank}. For any size positive $\eta$,
under the current assumption on the initial mass, we give a mathematically rigorous justification for the fact that the presence of magnetic effects without the skin effect in the cold plasma does not change the optimal lower bound for the blow-up rates.\\

\end{abstract}

{\bf Statements and Declarations:} No conflict of interest exists in the submission of this manuscript. No data was used for the research described in this manuscript.

\section{Introduction and main results}

%Langmuir waves is an important physical phenomenon taking place in the non-magnetized or weakly magnetized plasma. Consider formal expansion of long -wavelength and small-amplitude Langmuir oscillations. Neglecting higher order terms in the formal expansion, we obtain the governing equation.

The magnetic Zakharov system
$$
 \left\{
	\begin{array}{ll}		&i\textbf{E}_{t}+\nabla \nabla\cdot\textbf{E}- n \E- \alpha\nabla\times(\nabla\times\textbf{E})+i(\textbf{E}\land \textbf{B})=0,
\\[0.3cm]
		&\p_t n =  - \nabla \cdot  \mathbf{v},
\\[0.3cm]
		&\p_t  \mathbf{v }= - \nabla n - \nabla |\E|^2,
\\[0.3cm]
		&\Delta \textbf{B}-i\eta\nabla\times\nabla\times\left(\textbf{E}\land\overline{\textbf{E}}\right)+\mathbf{A}=0,
	\end{array}
\right.\eqno(I)
$$

 is proposed to describe the spontaneous generation of a magnetic field in a cold plasma \cite{9Kono}.
%interaction of electronic and magnetic field in the plasma physics. As we denote by
Here, $\E=(E_1,E_2,E_3) \in \mathbb{C}^3$ denotes the slowly varying complex amplitudes of the high-frequency electric field, $n$ the fluctuation of the electron density from its equilibrium, $\B$ the self-generated magnetic field in a cold plasma, $\mathbf{A} = \delta \B$, $\delta\leq 0$, $\eta>0$ and $\alpha\geq 1$ are physical constants.
%It is written down as:

%where $\mathbf{A}$ could be determined by
%\be
%, \text{ or } \mathbf{A}= - \gamma  \int \p_t \B(y,t) \frac{1}{|x-y|^2} dy,
%\ee
%with some parameters $\delta,\gamma$. They corresponds to the cases in cold plasma and hot plasma respectively.
It is worthing to point out that the Zakharov system (I) is a
Schr\"{o}dinger-wave coupled system with different scalings, % on each component  hence some space-time invariances would not keep.
%On the other hand,
and it keeps two conservation laws including the {\it total mass}
$$
\|\E\|_{L^2(\mathbb{R}^2)}^2  = \|E_1\|_{L^2(\mathbb{R}^2)}^2 + \|E_2\|_{L^2(\mathbb{R}^2)}^2 + \|E_3\|_{L^2(\mathbb{R}^2)}^2,\eqno(II)
$$
as well as the {\it total energy} %of the system %conservation
$$
\left.
\begin{array}{ll}	\mathcal{H}:=&\displaystyle||\nabla\times\textbf{E}||_{L^2(\mathbb{R}^2)}^2
+||\nabla\cdot\textbf{E}||_{L^2(\mathbb{R}^2)}^2
\\[0.3cm]
&\displaystyle+
\frac{1}{2}||n||_{L^2(\mathbb{R}^2)}^2+\frac{1}{2}||\textbf{v}||_{L^2(\mathbb{R}^2)}^2   +\int_{\mathbb{R}^2}^{}n|\textbf{E}|^2 dx
\\[0.3cm]
&\displaystyle +\frac{\eta}{2}\int_{\mathbb{R}^2} \frac{1}{|\xi|^2-\delta}\left(\left|\xi\cdot\mathcal{F}(\textbf{E}
\times\overline{\textbf{E}})\right|^2-|\xi|^2\left|\mathcal{F}(\textbf{E}
\times\overline{\textbf{E}})\right|^2\right)d\xi.
\end{array}
\right.\eqno(III)
   $$
%Throughout this paper, we let $\mathbf{A}=0$ and $\alpha=0$ for brevity and leave others for a future work. \\
In the cold plasma, the term $\delta \B$ corresponds to the classical (collisionless) skin effect \cite{9Kono}. From a physical viewpoint, $\alpha$ is relative to the velocity of electrons and the plasma frequency. In the present paper, we focus on the case for $\delta=0$ and $\alpha=1$, that is, the skin effect would not be involved and the velocity of electrons increases synchronously with the frequency of plasma.

From a physical point of view, the two-dimensional case for $\textbf{E}$ is essential and of great importance. Let
$$
\E(t,x)=\left(E_1(t,x),E_2(t,x),0\right), \;\; x \in \R^2.
$$
Through standard computations, one obtains
$$
\B(t,x)=\left(0,0,B_3(t,x)\right) = \left(0,0, -i \eta\left(E_1 \ov{E_2} - \ov{E_1}E_2\right)\right)
$$
by the fact that $\nabla \cdot\left( \E \wedge \ov{\E}\right) = 0$ and  the vectorial identity $\Delta \E = \nabla (\nabla \cdot \E) - \nabla \times \nabla \times \E$.
Hence, the interaction term involving the electronic and magnetic fields is given by
$$
i \E \wedge \B = \eta \E \wedge \E \wedge \overline{\E} = \eta \left(E_2\left(E_1\overline{E_2}-\overline{E_1}E_2\right),E_1\left(\overline{E_1}E_2-
E_1\overline{E_2}\right),0\right),
$$
and (I) can be rewritten by the following two-dimensional magnetic Zakharov system:
$$
\left\{
\begin{array}{lll}
&\displaystyle i \p_t E_{1}+\Delta E_1-n E_1+ \eta E_2\left(E_1\overline{E}_2-\overline{E}_1E_2\right)=0,&\qquad (1.1-1)
\\[0.3cm]
		&\displaystyle i \p_t E_{2}+\Delta E_2-n E_2+\eta E_1\left(\overline{E}_1E_2-E_1\overline{E}_2\right)=0,&\qquad (1.1-2)
\\[0.3cm]
	&\displaystyle	\p_t n+\nabla\cdot \mathbf{v} =0,&\qquad (1.1-3)
\\[0.3cm]
	&\displaystyle	\p_t \mathbf{v} +\nabla n+\nabla\left(|E_1|^2+|E_2|^2\right)=0,
&\qquad (1.1-4)
	\end{array}
\right.\eqno(1.1)
   $$
 which describes the spontaneous generation of a magnetic field without the skin effect in a cold plasma \cite{9Kono}.
Here,
$\eta>0$ is a physical constant coefficient,~~$ E_1(t,x),E_2(t,x):\mathbb{R}^+\times\mathbb{R}^2\to\mathbb{C}, ~n(t,x):\mathbb{R}^+\times\mathbb{R}^2\to\mathbb{R}, ~\textbf{v}(t,x):\mathbb{R}^+\times\mathbb{R}^2\to\mathbb{R}^2$.
The initial data for (1.1) is given by
$$
\left\{
\begin{array}{lll}
&E_{1}(0,x)=E_{10}(x),~~&E_{2}(0,x)=E_{20}(x),
\\[0.3cm]
&n(0,x)=n_{0}(x),
~~&\textbf{v}(0,x)=\textbf{v}_{0}(x).
\end{array}
\right.\eqno(1.2)
   $$
Due to the identity (III), the Hamiltonian for (1.1) can be written as
$$
\displaystyle
\left.
\begin{array}{ll}
&\mathcal{H}(E_1,E_2,n,\mathbf{v})
 \\[0.3cm]
 &\displaystyle\quad = \| \nabla E_1\|_{L^2}^2 + \| \nabla E_2\|_{L^2}^2 + \frac{1}{2} \|n\|_{L^2}^2 + \frac{1}{2}\|\mathbf{v}\|_{L^2}^2
 \\[0.3cm]
 &\displaystyle\quad\quad + \int_{\R^2} n\left(|E_1|^2+|E_2|^2\right)dx - \frac{\eta}{2} \int_{\R^2} \left|E_1 \ov{E}_2 - E_2 \ov{E}_1\right|^2dx.
\end{array}
\right.\eqno(1.3)
$$
Clearly it is a well-defined functional on the energy space
 $
 \BH_1:=H^1(\mathbb{R}^2)\times H^1(\mathbb{R}^2) \times  L^{2}(\mathbb{R}^2)\times L^{2}(\mathbb{R}^2).
$

The blow-up dynamics of the two dimensional classical Zakharov system have been studied in detail by several authors, in particular, Glangetas and Merle \cite{22Frank,23Frank} and Merle \cite{24Frank,25Frank}.
For the Zakharov systems with magnetic field effect (I),  Laurey in \cite{20Laurey} proved the global existence of weak solutions for small initial data as well as local existence and uniqueness of smooth solutions in both two-dimensional and three-dimensional spaces. Based on Laurey's work \cite{20Laurey}, over the last decade, finite time blow-up dynamics for (I) were considered. Gan, Guo and Huang in \cite{28Gan} constructed a family of blow-up solutions in two-dimensional space and proved the existence of self-similar blow-up solutions. The instability and the concentration property of a class of periodic solution were also obtained. In \cite{27Gan}, the authors studied the Virial type blow-up solutions of the Cauchy problem for (I). Later, the authors in \cite{29Gan} established the space-time integral estimate of the blow-up rate for the finite time blow-up solutions to (I) in the three dimensional space. Note that (I) without the classical (collisionless) skin effect $(\delta=0)$ and $\alpha=1$ in two dimensional space reduces to system (1.1), these results for the system (I) are naturally true for system (1.1).\\
\indent Although for the Zakharov systems have an additional magnetic field, there have been some results on the well-posedness as well as some progress on the blowup dynamics. To our best knowledge, there is no estimate on the lower-bound of the finite time blow-up rate. The aim of this paper is to establish the (essentially optimal) lower bound of the blowup rate for the finite time blowup solutions to the system (1.1).\\
\indent Let us state a few preliminary results. Firstly, using
  the methods used in \cite{15Bourgain,17Bourgain,18Bourgain,16Bourgain}, one can obtain the local well-posedness of mild solutions in the energy space.\\

\begin{proposition} \label{1.1}
The two dimensional magentic Zakharov system (1.1) is locally well-posed in the energy space $H^1(\mathbb{R}^2)\times H^1(\mathbb{R}^2) \times  L^{2}(\mathbb{R}^2)\times L^{2}(\mathbb{R}^2) $.
That is, there exists a unique solution $\left(E_1,E_2,n,\mathbf{v}\right)$ satisfying
\begin{gather*}
		\left\|(E_1,E_2,n,\mathbf{v})(t)\right\|_{H^1(\mathbb{R}^2)\times H^1(\mathbb{R}^2) \times  L^{2}(\mathbb{R}^2)\times L^{2}(\mathbb{R}^2)}\le C, \;\; \forall t \in [0,T),
\end{gather*}
where $T$ is the maximal existence time of the solution and constant $C$ depends only on the initial data $\left( E_{10},E_{20},n_0,\mathbf{v}_0\right)$.
\end{proposition}

Next, following Gan, Guo, Han and Zhang \cite{27Gan}, we can establish easily the virial-type blow-up solution for the Cauchy problem (1.1)-(1.2).\\
\begin{proposition} \label{1.2}
Let $\eta>0$. Suppose that for all time, the solutions $\left(E_1,E_2,n,\mathbf{v}\right)(t)$ to the Cauchy problem (1.1)-(1.2) in $\mathbb{R}^2$ are radially symmetric functions and $\mathcal{H}\left(E_{10},E_{20},n_{0},\mathbf{v}_{0}\right) <0$. The following alternative holds:
\\[0.3cm]
\indent (i)~~$\left(E_1,E_2,n,\mathbf{v}\right)(t)$ blows up in finite time,\\

\indent (ii)~~$\left(E_1,E_2,n,\mathbf{v}\right)(t)$ blows up in the energy space $H^1(\mathbb{R}^2)\times H^1(\mathbb{R}^2) \times  L^{2}(\mathbb{R}^2)\times L^{2}(\mathbb{R}^2)$ at infinity. That is, $\left(E_1,E_2,n,\mathbf{v}\right)(t)$ is defined for all $t$ and
$$\lim\limits_{t\rightarrow +\infty}\left\|(E_1,E_2,n,\mathbf{v})\right\|_{H^1(\mathbb{R}^2)\times H^1(\mathbb{R}^2) \times  L^{2}(\mathbb{R}^2)\times L^{2}(\mathbb{R}^2)}=+\infty.$$ \end{proposition}

\noindent With these results, it is natural to consider more quantitative descriptions on the behaviour of
 $\left\|(E_1,E_2,n,\mathbf{v})(t)\right\|_{H^1(\mathbb{R}^2)\times H^1(\mathbb{R}^2) \times  L^{2}(\mathbb{R}^2)\times L^{2}(\mathbb{R}^2)}$ as $t$ near $T$, where $T<\infty$ is the blow-up time. Compared with the classical Zakharov system, (1.1) contains two extra cubic coupling terms.
The two extra terms~$E_2\left(E_1\overline{E_2}-\overline{E_1}E_2\right)$
 and $E_1\left(\overline{E_1}
E_2-E_1\overline{E_2}\right)$ which are physically generated by the cold magnetic field (without the skin effect), do
 bring a different difficulty from that for the classical Zakharov system.
 Assuming the initial mass satisfies the following estimates:

%we obtain the optimal lower bound of blow-up rate of the finite-time blow-up solutions to the Cauchy problem (G-Z)-(G-Z-I)~in the energy space~$\mathbb{H}_1=H^1(\mathbb{R}^2)\times H^1(\mathbb{R}^2)\times L^2(\mathbb{R}^2)\times
%L^2(\mathbb{R}^2)$ provided the initial data $\left(E_{10}(x),E_{20}(x),n_{0}(x),\mathbf{v}_{0}(x)\right)$ satisfy
\begin{gather*}
\frac{||Q||_{L^2(\mathbb{R}^2)}^2}{1+\eta}
<||E_{10}||_{L^2(\mathbb{R}^2)}^2+||E_{20}||_{L^2(\mathbb{R}^2)}^2
<\frac{||Q||_{L^2(\mathbb{R}^2)}^2}{\eta},
\end{gather*}
where $Q$ is the unique radially positive solution of the equation
$
-\Delta V+V=V^3,$ we prove in the present paper that
%Namely,
there is a constant~$c>0$~ depending only on the initial data such that for~$t$ near $T$ (the blow-up time),
\begin{gather*}
\left\|\left(E_1,E_2,n,\textbf{v}\right)\right\|_{H^1(\mathbb{R}^2)\times H^1(\mathbb{R}^2)\times L^2(\mathbb{R}^2)\times
L^2(\mathbb{R}^2)}\geqslant\frac{c}{ T-t }.
\end{gather*}
The main result of the paper can be described by the following:\\
 \begin{theorem} \label{1.3}
Let $\left(E_1,E_2,n,\mathbf{v}\right)(t)$ be the finite time blow-up solution of the Cauchy problem (1.1)-(1.2), and $T<\infty$ be the blow-up time. Suppose that the initial data $(E_{10},E_{20})$ satisfies
$$\frac{ \|Q\|_{L^2(\mathbb{R}^2)}^2}{1+\eta}
<\|E_{10}\|_{L^2(\mathbb{R}^2)}^2+\|E_{20}\|_{L^2(\mathbb{R}^2)}^2
<\frac{\|Q\|_{L^2(\mathbb{R}^2)}^2}{\eta},\eqno(1.4)$$
where $Q$ is the unique radial positive solution of
$$-\Delta V+V=V^3,\eqno(1.5)$$
then there exist constants $c>0,\tilde c>0$ depending only on initial data, such that as $t$ near $T$,\\
$(1)$
$$
\left\|(E_1,E_2,n,\textbf{v})(t)\right\|_{H^1(\mathbb{R}^2)\times H^1(\mathbb{R}^2)\times L^2(\mathbb{R}^2)\times L^2(\mathbb{R}^2)}\geqslant\frac{c}{T-t},\eqno(1.6)
$$
$$
\left(\left\|\nabla E_1(t)\right\|_{L^2(\mathbb{R}^2)}^2+\left\|\nabla E_2(t)\right\|_{L^2(\mathbb{R}^2)}^2\right)^{\frac{1}{2}}\geqslant\frac{\tilde c}{T-t},\eqno(1.7)
$$
$$
||n(t)||_{L^2(\mathbb{R}^2)}\geqslant\frac{\tilde c}{T-t}.\eqno(1.8)
$$
\noindent More precisely,
\\[0.2cm]
\noindent $(2)$
$$
\left(||\nabla E_1(t)||_{L^2(\mathbb{R}^2)}^2+||\nabla E_2(t)||_{L^2(\mathbb{R}^2)}^2\right)^{\frac{1}{2}}\qquad\qquad\qquad\qquad\qquad\qquad
$$
$$\geqslant
\frac{\tilde c}{\left(||E_{10}||_{L^2(\mathbb{R}^2)}^2+||E_{20}||_{L^2(\mathbb{R}^2)}^2
-\frac{||Q||_{L^2(\mathbb{R}^2)}^2}{1+\eta}\right)^{\frac{1}{2}}}\frac{1}{T-t},\eqno(1.9)
$$
\\[0.1cm]
$$
\|n(t)\|_{L^2(\mathbb{R}^2)}\geqslant\frac{ \tilde c}{\left(||E_{10}||_{L^2(\mathbb{R}^2)}^2+||E_{20}||_{L^2(\mathbb{R}^2)}^2
-\frac{||Q||_{L^2(\mathbb{R}^2)}^2}{1+\eta}\right)^{\frac{1}{2}}}\frac{1}{T-t}.\eqno(1.10)
$$
\end{theorem}

\indent The assumption on the lower bound for the initial mass is natural particularly when $\eta$ is small as it is basically the minimal mass needed for blow-up to occur. Indeed,
as the magnetic coefficient $\eta$ tends to $0$, the blow-up rate recovers the result for the classical 2-D Zakharov system due to Merle \cite{25Frank}. On the other hand, for any size positive $\eta$,
the assumption on the upper bound is not known to be optimal nor to be necessary, but when $\eta$ is sufficiently small this assumption is automatic for a giving finite mass initial data.
For large $\eta$ our result is in some sense more intruging. It provides a mass band in which one can get
more precise information for blow-ups. Whether or not such a statement can be generalized to more general (likely multiple bubble blow-ups) mass levels remains to be a fascinating open problem.
Under the current assumption on the initial mass, we give a mathematically rigorous justification for the fact that the presence of magnetic effects without the skin effect in the cold plasma does not change the optimal lower bound for the blow-up rates.
%It should be pointed out that as the magnetic coefficient $\eta$ tends to $0$, the blow-up rates recovers the result %for the classical 2-D Zakharov system duo to Merle\cite{25Frank}. Therefore, we actually give a mathematically %rigorous proof for the fact that the presence of magnetic effects without the skin effect in the cold plasma does %not change the optimal lower bound for the blow-up rates.
\begin{remark}\label{1.4}
In \cite{28Gan}, Gan, Guo and Huang constructed a family of blow-up solutions of the Cauchy problem (1.1)-(1.2):
$$
\left\{
\begin{array}{ll}
E_1(t,x)=\frac{\omega}{T-t}e^{i\left(\theta+\frac{|x|^2}{4(-T+t)}
-\frac{\omega^2}{ -T+t}\right)}\frac{\tilde
P(\frac{x\omega}{T-t})}{\sqrt{2}},
\\[0.3cm]
\ E_2(t,x)=-i E_1(t,x),\\[0.3cm]
\ n(t,x)=\frac{\omega^2}{(T-t)^2}\tilde N\left(\frac{x\omega}{T-t}\right),
\end{array}
\right.\eqno(1.11)
$$
where $\tilde P(x)=\tilde P(|x|),\tilde N(x)=\tilde N(|x|)$ are radial functions on $\mathbb{R}^2$, $\theta\in\mathbb{R}$ and $\omega>0$. Let $\tilde P=\frac{P}{\left(1+\eta \right)^{\frac{1}{2}}},~~\tilde N=\frac{N}{1+\eta}$, then $(P,N)$ satisfies
$$
\left\{
\begin{array}{ll}
\Delta P-P+\frac{\eta}{\eta+1}P^3=\frac{1}{\eta+1}NP,
\\[0.3cm]
\frac{1}{\omega^2} (r^2 N_{rr}+6rN_{r}+6N)-\Delta N=\Delta|P|^2.
\end{array}
\right.\eqno(1.12)
$$
%Here $(\tilde P(x),\tilde N(x),\tilde B(x))$ fits the following system:
%\begin{equation}\tag{1-10}
%\begin{cases}
%\ \Delta\tilde P-\tilde P+\tilde P\tilde B=\tilde N\tilde P,\\
%\ \lambda^2(r^2\tilde{N}_{rr}+6r\tilde{N_r}+6\tilde N)-\Delta\tilde N=\Delta|\tilde P|^2\\
%\ \Delta\tilde B+\delta c_0^2(T-t)^2\lambda^2\tilde B=\eta\Delta|\tilde P|^2\\
%\end{cases}
%\end{equation}\\
%where $r=|x|,\Delt
\\
Direct calculation yields
\\
$$
\left\{
\begin{array}{ll}
||\nabla E_1(t)||_{L^2(\mathbb{R}^2)}=\frac{\omega}{T-t}||\nabla\tilde P||_{L^2(\mathbb{R}^2)},
\\[0.3cm]
||\nabla E_2(t)||_{L^2(\mathbb{R}^2)}=\frac{\omega}{T-t}||\nabla\tilde P||_{L^2(\mathbb{R}^2)},
\\[0.3cm]
||n(t)||_{L^2(\mathbb{R}^2)}=\frac{\omega}{T-t}||N||_{L^2(\mathbb{R}^2)},
\\[0.3cm]
||v(t)||_{L^2(\mathbb{R}^2)}=\frac{\omega c(P,N)}{T-t}.
\end{array}
\right.\eqno(1.13)
$$
\\
These estimates imply that the lower bound estimates of blow-up rate (1.6), (1.7) and (1.8) in Theorem 1.3 are optimal. On the other hand, letting $\omega\rightarrow+\infty$ and $(P,N)\rightarrow (Q,-Q^2)$, it yields that (1.9) and (1.10) in Theorem 1.3 are also optimal. \hfill$\Box$ \\
\end{remark}

\indent We note that the Zakharov system (1.1) is a Hamiltonian system which leads to conservations of the total mass and total energy:
\\
$$\|E_{1}\|_{L^{2}(\mathbb{R}^{2})}^{2}+\|E_{2}\|_{L^{2}(\mathbb{R}^{2})}^{2}=
\|E_{10}\|_{L^{2}(\mathbb{R}^{2})}^{2}+\|E_{20}\|_{L^{2}(\mathbb{R}^{2})}^{2},\eqno(1.14)$$
\\
$$\mathcal{H}(E_{1},E_{2},n,\mathbf{v})
=\mathcal{H}(E_{10},E_{20},n_{0},\mathbf{v}_{0})=\mathcal{H}_{0},\eqno(1.15)$$
where
$ \mathcal{H}(E_1,E_2,n,\mathbf{v})$ is defined in (1.3).\\

\indent In contrast to the Zakharov system without the magnetic field effect, the presence of extra nonlinear terms $E_{2}\left(E_{1}\overline{E _{2}}-\overline{E_{1}}E_{2}\right)$ and $E_{1}\left(\overline{E _{1}}E_{2} -E_{1}\overline{E_{2}}\right)$ (generated by the magnetic field) in (1.1), make
the study of the lower-bound estimate for blow-up rates of finite time blow-up solution (to the Cauchy problem (1.1)-(1.2)) more difficult and complicated. Motivated by Merle's beautiful arguments given in \cite{25Frank}, in particular, its geometrical estimate, the non-vanishing estimate and the compactness argument, we need to establish  additional a priori estimates for the extra nonlinear terms. For this we also need some different techniques from those adopted in \cite{25Frank}. In particular, to obtain the optimal lower bound of blow-up rate, the initial mass is required to satisfy (1.4)
so that we can establish corresponding needed a priori estimates involving the higher order nonlinear terms. As we have mentioned earlier, the condition (1.4) is natural from both the physical and mathematical point of views. Indeed, by Lemma 2.2 in Section 2, it is standard to conclude that the mild (no blow up) solution of \eqref{1.1} is globally well-defined if the initial data $(E_{10} ,E_{20})$ satisfies
$$
||E_{10}||_{L^2(\mathbb{R}^2)}^2+||E_{20}||_{L^2(\mathbb{R}^2)}^2 < \frac{||Q||_{L^2(\mathbb{R}^2)}^2}{1+\eta}.
$$
\\
In \cite{28Gan} the authors also pointed out that there is no mass-concentration at a finite time provided
$$
||E_{10}||_{L^2(\mathbb{R}^2)}^2+||E_{20}||_{L^2(\mathbb{R}^2)}^2 = \frac{||Q||_{L^2(\mathbb{R}^2)}^2}{1+\eta}.
$$
We note that condition \eqref{1.4} is also consistent with the blow-up dynamics of the classical Zakharov system \cite{23Frank} when $\eta \to 0$. On the other hand, it is a fascinating issue to investigate some quantitative properties of the blow-up dynamics of (1.1)-(1.2) when $||E_{10}||_{L^2(\mathbb{R}^2)}^2+||E_{20}||_{L^2(\mathbb{R}^2)}^2 \geq \frac{||Q||_{L^2(\mathbb{R}^2)}^2}{\eta}$ .
For example, it may involve multiple bubbles blow-ups etc, but we have to leave it for a future study.
\\
 \indent The rest of paper is organized as follows. In Section 2, we collect preparatory materials, mainly on some lemmas and propositions which are crucial to the proof of Theorem 1.3. Section 3 is devoted to establishing some properties of the rescaled Zakharov system. In section 4 we prove the optimal lower-bound of finite time blow-up rate (Theorem 1.3).

\section{Preliminaries}

In this section, we give some notations and basic estimates. First we recall a lemma in \cite{2Brezis,6Evans}.\\
\begin{lemma} \label{2.1}
Let~$\Omega$~be a smooth bounded domain in~$\mathbb{R}^n$~with~$n\geqslant 2$ and $\frac{1}{p}+\frac{1}{q}=1$.~~If~~$f_k\rightharpoonup f$ in $L^p(\Omega)$, and $g_k\to g$ in $L^q(\Omega)$ as $k\to+\infty$, then
\begin{gather*}
\int_{\Omega}f_k g_kdx\to\int_{\Omega}fgdx~~as ~~k\to+\infty.
\end{gather*}
\end{lemma}
\qquad\hfill$\Box$\\

\begin{lemma} \label{2.2}(Weinstein \cite{30Weinstein})
\quad For all $u\in H^1(\mathbb{R}^2)$,
\begin{gather*}
\frac{1}{2}\|u\|_{L^4(\mathbb{R}^2)}^4
\leqslant\left(\frac{\|u\|_{L^2(\mathbb{R}^2)}^2}
{\|Q\|_{L^2(\mathbb{R}^2)}^2}\right)\|\nabla
u\|_{L^2(\mathbb{R}^2)}^2,
\end{gather*}
where~$Q$~is the unique positive solution of the following equation
\begin{gather*}
-\Delta V+V=V^3.
\end{gather*}
\end{lemma}
\qquad\hfill$\Box$\\
\indent Since (1.1-1)-(1.1-2) and (1.1-3)-(1.1-4) have the same scaling on spatial structure but they are of different space-time structures, it will conserve (1.1-1)-(1.1-2) and (1.1-3)-(1.1-4) respectively. Taking a suitable space-time scaling to the Zakharov system (1.1) can yield a re-scaled system.\\
\begin{proposition} \label{2.3}
 Let $(E_{1},E_{2},n,\mathbf{v})$ be the finite-time blow-up solutions to the Zakharov system (1.1) and T be its blow-up time. For any
 $t\in [0,T)$, let
$$
\left.
\begin{array}{ll}
&\tilde{E}_1(t,s,x)=\frac{1}{\lambda(t)}E_1\left(t+\frac{s}{\lambda(t)},
\frac{x}{\lambda(t)}\right),
\\[0.3cm]
&\tilde{E}_2(t,s,x)=\frac{1}{\lambda(t)}E_2\left(t+\frac{s}{\lambda(t)},
\frac{x}{\lambda(t)}\right),
\\[0.3cm]
&\tilde{n}(t,s,x)=\frac{1}{\lambda^2(t)}n\left(t+\frac{s}{\lambda(t)},
\frac{x}{\lambda(t)}\right),
\\[0.3cm]
&\tilde{\textbf{v}}(t,s,x)=\frac{1}{\lambda^2(t)}\textbf{v}\left(t+\frac{s}
{\lambda(t)},\frac{x}{\lambda(t)}\right),
\end{array}
\right.\eqno(2.1)
$$
where~$s\in\left[0,\lambda(t)(T-t)\right)$,\\
$$
\left.
\begin{array}{ll}
\lambda^2(t)&\displaystyle =\|(E_1,E_2,n,\textbf{v})\|^{2}_{H^{1}(\mathbb{R}^{2})\times H^{1}(\mathbb{R}^{2})
\times L^{2}(\mathbb{R}^{2})\times L^{2}(\mathbb{R}^{2})}
\\[0.3cm]
 &\displaystyle =\int_{\mathbb{R}^2}|\nabla E_1(t,x)|^2\,dx+\int_{\mathbb{R}^2}^{}|\nabla E_2(t,x)|^2\,dx
\\[0.4cm]
&\qquad\displaystyle+\frac{1}{2}\int_{\mathbb{R}^2}^{}|n(t,x)|^2\,dx+
\frac{1}{2}\int_{\mathbb{R}^2}^{}|\textbf{v}(t,x)|^2\,dx.
\end{array}
\right.\eqno(2.2)
$$
\\
Then
$(\tilde E_1,\tilde E_2,\tilde n,\tilde{\textbf{v}})(s)$~satisfies the following re-scaled Zakharov system $:$
\\
$$
\left\{
\begin{array}{ll}
\frac{1}{\lambda}i\tilde{E}_{1s}+\Delta\tilde E_1-\tilde
n\tilde E_1+\eta\tilde E_2\left(\tilde E_1\overline{\tilde E_2}-\overline{\tilde E_1}\tilde E_2\right)=0,\qquad\qquad&(2.3a)
\\[0.3cm]
\frac{1}{\lambda}i\tilde E_{2s}+\Delta\tilde E_2-\tilde
n\tilde E_2+\eta\tilde E_1\left(\overline{\tilde E_1}\tilde E_2-\tilde E_1\overline{\tilde E_2}\right)=0,&(2.3b)
\\[0.3cm]
\tilde n_s+\nabla\cdot\tilde{\textbf{v}}=0,&(2.3c)
\\[0.3cm]
\tilde{\textbf{v}}_s+\nabla\left(\tilde n+|\tilde E_1|^2+|\tilde E_2|^2\right)=0.&(2.3d)
 \end{array}
\right. \eqno(2.3)
$$
\\
In addition, there hold\\
\\
$(1)$
$$
\left.
\begin{array}{ll}
&\lim_{s\to\lambda(t)(T-t)}\left\|\left(\tilde E_1,\tilde E_2,\tilde n,\tilde{\textbf{v}}\right)(s)\right\|_{H^{1}(\mathbb{R}^{2})\times H^{1}(\mathbb{R}^{2})
\times L^{2}(\mathbb{R}^{2})\times L^{2}(\mathbb{R}^{2})}^2
\\[0.4cm]
 &\qquad=\lim_{s\to\lambda(t)(T-t)}\left(\left\|\nabla\tilde E_1(s)\right\|_{L^2(\mathbb{R}^2)}^2+\left\|\nabla\tilde E_2(s)\right\|_{L^2(\mathbb{R}^2)}^2\right.
 \\[0.4cm]
&\qquad\qquad\qquad\qquad\qquad\left.+\frac{1}{2}||\tilde
n(s)||_{L^2(\mathbb{R}^2)}^2+\frac{1}{2}
||\tilde{\textbf{v}}(s)||_{L^2(\mathbb{R}^2)}^2\right)
\\[0.4cm]
 &\qquad=+\infty,
 \end{array}
\right.\eqno(2.4)
$$
\\
$(2)$
\\
$$
\int_{\mathbb{R}^2} \left(\left|\nabla\tilde E_1(t,0,x)\right|^2+\left|\nabla\tilde E_2(t,0,x)\right|^2+\frac{1}{2}|\tilde
n(t,0,x)|^2+\frac{1}{2}|\tilde{\textbf{v}}(t,0,x)|^2\right)\,dx=1,\eqno(2.5)
$$
\\
$(3)$
$$
\left.
\begin{array}{ll}
&\left\|\tilde E_1(t,s,x)\right\|_{L^2(\mathbb{R}^2)}^2+\left\|\tilde E_2(t,s,x)\right\|_{L^2(\mathbb{R}^2)}^2
\\[0.4cm]
&\qquad=\left\|\tilde E_1(t,0,x)\right\|_{L^2(\mathbb{R}^2)}^2+\left\|\tilde E_2(t,0,x)\right\|_{L^2(\mathbb{R}^2)}^2
\\[0.4cm]
&\qquad=\left\|E_{10}\right\|_{L^2(\mathbb{R}^2)}^2+\left\|E_{20}\right\|_{L^2(\mathbb{R}^2)}^2,\end{array}
\right.\eqno(2.6)
$$
and\\
\\
$(4)$
$$
\left.
\begin{array}{ll}
&\mathcal{H}(\tilde E_1,\tilde E_2,\tilde n,\tilde{\textbf{v}})
\\[0.4cm]
&\quad=\left\|\nabla\tilde E_1\right\|_{L^2(\mathbb{R}^2)}^2+\left\|\nabla\tilde E_2\right\|_{L^2(\mathbb{R}^2)}^2+\frac{1}{2}||\tilde n||_{L^2(\mathbb{R}^2)}^2+\frac{1}{2}||\tilde{\textbf{v}}||
_{L^2(\mathbb{R}^2)}^2
\\[0.4cm]
&\displaystyle\quad\quad+\int_{\mathbb{R}^2}^{}\tilde{n}\left(|\tilde E_1|^2+|\tilde E_2|^2\right)\,dx-\eta\int_{\mathbb{R}^2}^{}|\tilde E_1|^2|\tilde E_2|^2\,dx
\\[0.4cm]
&\displaystyle\quad\quad+\frac{\eta}{2}\int_{\mathbb{R}^2}^{}\left(\left(\tilde E_1\right)^2\left(\overline{\tilde E_2}\right)^2+\left(\overline{\tilde E_1}\right)^2\left(\tilde E_2\right)^2\right)\,dx
\\[0.4cm]
&\displaystyle\quad=\frac{1}{\lambda^2(t)}\mathcal{H}(E_1,E_2,n,\textbf{v})
\\[0.4cm]
&\displaystyle\quad=\frac{1}{\lambda^2(t)}\mathcal{H}(E_{10},E_{20},n_0,\textbf{v}_0).
\end{array}
\right.\eqno(2.7)
$$
\end{proposition}

{\bf Proof.} According to (2.1),~direct calculation yields
$$
\left\{
\begin{array}{ll}
\tilde E_{1s}=\frac{1}{\lambda^2(t)}E_{1t},\quad\Delta\tilde E_1=\frac{1}{\lambda^3(t)}\Delta E_1,\quad \nabla\left|\tilde E_1\right|^2=\frac{1}{\lambda^3(t)}\nabla|E_1|^2,
 \\[0.3cm]
\tilde E_{2s}=\frac{1}{\lambda^2(t)}E_{2t},\quad\Delta\tilde E_2=\frac{1}{\lambda^3(t)}\Delta E_2,\quad \nabla\left|\tilde E_2\right|^2=\frac{1}{\lambda^3(t)}\nabla|E_2|^2,
 \\[0.3cm]
\tilde n_s=\frac{1}{\lambda^3(t)}n_t,\quad\nabla\tilde n=\frac{1}{\lambda^3(t)}\nabla n,\quad\nabla\cdot\tilde{\textbf{v}}=\frac{1}{\lambda^3(t)}\nabla\cdot\textbf{v},\quad\tilde{\textbf{v}}_s=\frac{1}{\lambda^3(t)}\textbf{v}_t.
\end{array}
\right.\eqno(2.8)
$$
Taking (2.8) into the Zakharov system (1.1) yields the re-scaled Zakharov system (2.3). Similarly, using (2.1) one obtains
$$
\left.
\begin{array}{ll}
&\displaystyle\left\|\left(\tilde E_1,\tilde E_2,\tilde n,\tilde{\textbf{v}}\right)(t,s,x)\right\|_{H^{1}(\mathbb{R}^{2})\times H^{1}(\mathbb{R}^{2})
\times L^{2}(\mathbb{R}^{2})\times L^{2}(\mathbb{R}^{2})}^2
 \\[0.4cm]
 &\displaystyle\qquad=\left\|\nabla\tilde E_1(t,s,x)\right\|_{L^2(\mathbb{R}^2)}^2+\left\|\nabla\tilde E_2(t,s,x)\right\|_{L^2(\mathbb{R}^2)}^2
 \\[0.4cm]
 &\displaystyle\qquad\quad
 +\frac{1}{2}\left\|\tilde
n(t,s,x)\right\|_{L^2(\mathbb{R}^2)}^2+\frac{1}{2}\left\|\tilde{\textbf{v}}(t,s,x)\right\|
_{L^2(\mathbb{R}^2)}^2
 \\[0.4cm]
 &\displaystyle\qquad=\left\|\frac{1}{\lambda^2(t)}\nabla E_1\right\|_{L^2(\mathbb{R}^2)}^2+\left\|\frac{1}{\lambda^2(t)}\nabla
E_2\right\|_{L^2(\mathbb{R}^2)}^2
\\[0.4cm]
 &\displaystyle\qquad\quad
+\frac{1}{2}\left\|\frac{1}{\lambda^2(t)}n\right\|_
{L^2(\mathbb{R}^2)}^2+\frac{1}{2}\left\|\frac{1}{\lambda^2(t)}\textbf{v}\right\|
_{L^2(\mathbb{R}^2)}^2
 \\[0.4cm]
&\displaystyle\qquad=\frac{1}{\lambda^2(t)}\int_{\mathbb{R}^2}^{}\left(|\nabla E_1|^2+|\nabla
E_2|^2+\frac{1}{2}|n|^2+\frac{1}{2}|\textbf{v}|^2\right)
 \\[0.4cm]
&\displaystyle\qquad\qquad\qquad\qquad
\left(t+\frac{s}
{\lambda(t)},\frac{x}{\lambda(t)}\right)\,d\left(\frac{x}{\lambda(t)}\right).
\end{array}
\right.\eqno(2.9)
$$
\\
Noting that $t+\frac{s}{\lambda(t)}\to T$ as $s\to\lambda(t)(T-t)$,~one has
$$
\left.
\begin{array}{ll}
&\displaystyle\lim_{s\to\lambda(t)(T-t)}\int_{\mathbb{R}^2}^{}\left(|\nabla E_1|^2+|\nabla
E_2|^2+\frac{1}{2}|n|^2+\frac{1}{2}|\textbf{v}|^2\right)\qquad\qquad\qquad\qquad
\\[0.4cm]
&\displaystyle\qquad\qquad\qquad\qquad\qquad\left(t+\frac{s}
{\lambda(t)},\frac{x}{\lambda(t)}\right)\,d\left(\frac{x}{\lambda(t)}\right)=+\infty.
\end{array}
\right.\eqno(2.10)
$$
That is,
 $$
\lim_{s\to\lambda(t)(T-t)}\left\|\left(\tilde E_1,\tilde E_2,\tilde n,\tilde {\mathbf{v}}\right)(t,s,x)\right\|_{H^{1}(\mathbb{R}^{2})\times H^{1}(\mathbb{R}^{2})
\times L^{2}(\mathbb{R}^{2})\times L^{2}(\mathbb{R}^{2})}^2=+\infty, \eqno(2.11)
$$
this is the estimate (2.4).\\
\indent Taking the inner product of (2.3a) with $\overline{\tilde{E_1}}$ and of (2.3b) with $\overline{\tilde{E_2}}$, integrating with respect to spatial variable $x$, then taking the imaginary part of the resulting equations yield
 $$
\left.
\begin{array}{ll}
&\displaystyle Im\int_{\mathbb{R}^2}\left[\frac{i}{\lambda}\tilde E_{1s}\cdot\overline{\tilde E_1}+\Delta\tilde E_1\cdot\overline{\tilde E_1}-\tilde n\tilde E_1\cdot\overline{\tilde E_1}+\eta\overline{\tilde E_1}\tilde E_2\left(\tilde E_1\overline{\tilde E_2}-\overline{\tilde E_1}\tilde E_2\right)\right]dx
\\[0.4cm]
 &\displaystyle\qquad=Re\int_{\mathbb{R}^2}\frac{1}{ \lambda} \tilde E_{1s}\cdot\overline{\tilde E_1} \,dx-Im\int_{\mathbb{R}^2}\eta\left(\overline{\tilde E_1}\right)^2\left(\tilde E_2\right)^2\,dx
 \\[0.4cm]
 &\displaystyle\qquad=\frac{1}{2\lambda}\frac{d}{ds}\int_{\mathbb{R}^2}\left|\tilde E_1\right|^2\,dx-Im\int_{\mathbb{R}^2}\eta\left(\overline{\tilde E_1}\right)^2\left(\tilde E_2\right)^2\,dx
 \\[0.4cm]
 &\qquad=0.
\end{array}
\right.\eqno(2.12)
$$
That is,
$$
\frac{1}{2\lambda}\frac{d}{ds}\int_{\mathbb{R}^2}\left|\tilde E_1\right|^2\,dx-Im\int_{\mathbb{R}^2}\eta\left(\overline{\tilde E_1}\right)^2\left(\tilde E_2\right)^2\,dx=0.\eqno(2.13)
$$
Similar discussion gives
$$
\frac{1}{2\lambda}\frac{d}{ds}\int_{\mathbb{R}^2}\left|\tilde E_2\right|^2\,dx-Im\int_{\mathbb{R}^2}\eta\left({\tilde E_1}\right)^2\left(\overline{\tilde E_2}\right)^2\,dx=0.\eqno(2.14)
$$
(2.13) and (2.14) yield

$$
\frac{d}{ds}\int_{\mathbb{R}^2}\left(\left|\tilde{E_1}\right|^2+\left|\tilde{E_2}\right|^2\right)\,dx=0.\eqno(2.15)
$$
Here we use the conclusion:  $\forall (f,g)\in\mathbb{C}^2$,~~$Im\left(f^2\overline{g}^2+\overline{f}^2 g^2\right)=0$.~Hence one gets
$$
\left.
\begin{array}{ll}
&\left\|\tilde E_1(t,s,x)\right\|_{L^2(\mathbb{R}^2)}^2+\left\|\tilde E_2(t,s,x)\right\|_{L^2(\mathbb{R}^2)}^2
\\[0.4cm]
&\qquad=\left\|\tilde E_1(t,0,x)\right\|_{L^2(\mathbb{R}^2)}^2+\left\|\tilde E_2(t,0,x)\right\|_{L^2(\mathbb{R}^2)}^2.
\end{array}
\right.\eqno(2.16)
$$
In addition, from the mass conservation (1.14), it follows that
$$
\left.
\begin{array}{ll}
&\left\|\tilde E_1(t,0,x)\right\|_{L^2(\mathbb{R}^2)}^2+\left\|\tilde E_2(t,0,x)\right\|_{L^2(\mathbb{R}^2)}^2
\\[0.4cm]
 &\displaystyle\qquad=\int_{\mathbb{R}^2}\left(\left|\tilde E_1(t,0,x)\right|^2+\left|\tilde E_2(t,0,x)\right|^2\right)\,dx
 \\[0.4cm]
 &\displaystyle\qquad=\int_{\mathbb{R}^2}^{}\frac{1}
 {\lambda^2(t)}\left(|E_1|^2+|E_2|^2\right)
\left(t,\frac{x}{\lambda(t)}\right)\,dx
\\[0.4cm]
&\displaystyle\qquad=\int_{\mathbb{R}^2}\left(|E_1|^2+|E_2|^2\right)
\left(t,\frac{x}{\lambda(t)}\right)d\left(\frac{x}{\lambda(t)}\right)
\\[0.4cm]
&\displaystyle\qquad=\|E_{10}\|_{L^2(\mathbb{R}^2)}^2+\|E_{20}\|_{L^2(\mathbb{R}^2)}^2.
\end{array}
\right.\eqno(2.17)
$$
Noting that
$$
\lambda^2(t)=\int_{\mathbb{R}^2}^{}|\nabla E_1|^2\,dx+\int_{\mathbb{R}^2}^{}|\nabla
E_2|^2\,dx+\frac{1}{2}\int_{\mathbb{R}^2}^{}|n|^2\,dx+\frac{1}{2}
\int_{\mathbb{R}^2}^{}|\textbf{v}|^2\,dx,\eqno(2.18)
$$
one gets
$$
\left.
\begin{array}{ll}
&\displaystyle\int_{\mathbb{R}^2}^{}\left(|\nabla\tilde E_1(t,0,x)|^2+|\nabla\tilde E_2(t,0,x)|^2+\frac{1}{2}|\tilde
n(t,0,x)|^2+\frac{1}{2}|\tilde{\textbf{v}}(t,0,x)|^2\right)\,dx
\\[0.4cm]
 &\displaystyle=\frac{1}{\lambda^2(t)}\int_{\mathbb{R}^2}^{}\left(|\nabla E_1(t,x)|^2+|\nabla
E_2(t,x)|^2+\frac{1}{2}|n(t,x)|^2+\frac{1}{2}|\textbf{v}(t,x)|^2\right)\,dx
\\[0.4cm]
&=1.
\end{array}
\right.\eqno(2.19)
$$
The above arguments imply (2.5) and (2.6).\\
\\
\indent Finally, taking the inner product of (2.3a) with $\overline{\tilde E_{1s}}$ and of (2.3b) with $\overline{\tilde E_{2s}}$,~integrating with respect to spatial variable ~$x$,~then taking the real part of the result equations,~we have
$$\displaystyle Re\int_{\mathbb{R}^2}\left[\frac{i}{\lambda}\tilde E_{1s}\cdot\overline{\tilde E_{1s}}+\Delta\tilde E_1\cdot\overline{\tilde E_{1s}}-\tilde n\tilde E_1\cdot\overline{\tilde E_{1s}}
 +\eta\overline{\tilde E_{1s}}\tilde E_2\left(\tilde E_1\overline{\tilde E_2}-\overline{\tilde E_1}\tilde E_2\right)\right]dx=0,
 $$
 that is,
$$
\left.
\begin{array}{ll}
 &\displaystyle Re\int_{\mathbb{R}^2}\frac{i}{\lambda}\left|\tilde E_{1s}\right|^2dx-\frac{1}{2}\frac{d}{ds}\int_{\mathbb{R}^2}
 \left|\nabla\tilde E_1\right|^2dx-\int_{\mathbb{R}^2}\tilde n\frac{1}{2}\frac{d}{ds}\left|\tilde E_1\right|^2dx
 \\[0.4cm]
&\displaystyle\qquad+\frac{\eta}{2}\int_{\mathbb{R}^2}\left|\tilde E_2\right|^2\cdot\frac{d}{ds}\left|\tilde E_1\right|^2 dx-\frac{\eta}{2}Re\int_{\mathbb{R}^2}\left(\tilde E_2\right)^2\cdot\frac{d}{ds}\left(\overline{\tilde E_1}\right)^2dx
\\[0.4cm]
&\displaystyle=-\frac{1}{2}\frac{d}{ds}\int_{\mathbb{R}^2}\left|\nabla\tilde E_1\right|^2dx-\frac{1}{2}\int_{\mathbb{R}^2}\tilde n\cdot\frac{d}{ds}\left|\tilde E_1\right|^2dx
\\[0.4cm]
&\displaystyle\qquad+\frac{\eta}{2}\int_{\mathbb{R}^2}\left|\tilde E_2\right|^2\frac{d}{ds}\left|\tilde E_1\right|^2dx-\frac{\eta}{2}Re\int_{\mathbb{R}^2}\left(\tilde E_2\right)^2\frac{d}{ds}\left(\overline{\tilde E_1}\right)^2dx
\\[0.4cm]
&=0.
\end{array}
\right.\eqno(2.20)
$$
Similarly, one attains
$$ \displaystyle Re\int_{\mathbb{R}^2}\left[\frac{i}{\lambda}\tilde E_{2s}\cdot\overline{\tilde E_{2s}}+\Delta\tilde E_2\cdot\overline{\tilde E_{2s}}-\tilde n\tilde E_2\cdot\overline{\tilde E_{2s}} +\eta\overline{\tilde E_{2s}}\tilde E_1\left(\overline{\tilde E_1}\tilde E_2-\tilde E_1\overline{\tilde E_2}\right)\right]dx
=0,$$
namely,
$$
\left.
\begin{array}{ll}
&\displaystyle-\frac{1}{2}\frac{d}{ds}\int_{\mathbb{R}^2}|\nabla\tilde E_2|^2dx-\frac{1}{2}\int_{\mathbb{R}^2}\tilde n\cdot\frac{d}{ds}|\tilde E_2|^2dx
\\[0.4cm]
&\displaystyle\quad+\frac{\eta}{2}\int_{\mathbb{R}^2}|\tilde E_1|^2\frac{d}{ds}|\tilde E_2|^2dx-\frac{\eta}{2}Re\int_{\mathbb{R}^2}(\tilde E_1)^2\frac{d}{ds}\left(\overline{\tilde E_2}\right)^2dx
\\[0.4cm]
&=0.\end{array}
\right.\eqno(2.21)
$$
Next, taking the inner product of (2.3c) with $\tilde n$~yields
$$
 \displaystyle\int_{\mathbb{R}^2}\left(\tilde n_s  \tilde n+ \left(\nabla\cdot\tilde {\mathbf{v}} \right)  \tilde n\right)dx
 =\frac{1}{2}\frac{d}{ds}\int_{\mathbb{R}^2}|\tilde n|^2dx-\int_{\mathbb{R}^2}\tilde {\mathbf{v}}\cdot\nabla\tilde ndx
 =0, \eqno(2.22)
$$
On the other hand, taking the inner product of (2.3d) with $\tilde v$~implies
$$
\left.
\begin{array}{ll}
\displaystyle\int_{\mathbb{R}^2}\left[\tilde v_s\cdot\tilde {\mathbf{v}}+\nabla\left(\tilde n+|\tilde E_1|^2+|\tilde E_2|^2\right)\cdot \tilde {\mathbf{v}}\right]dx&=\frac{1}{2}\frac{d}{ds}\int_{\mathbb{R}^2}|\tilde {\mathbf{v}} |^2dx-\int_{\mathbb{R}^2}\tilde n\nabla\cdot\tilde {\mathbf{v}}dx
\\[0.3cm]
&\displaystyle +\int_{\mathbb{R}^2}\tilde n_s\left(|\tilde E_1|^2+|\tilde E_2|^2\right)dx
=0.
\end{array}
\right.\eqno(2.23)
$$
Combining (2.21) with (2.22) and (2.23) gives
$$
\left.
\begin{array}{ll}
&\displaystyle \frac{d}{ds}\int_{\mathbb{R}^2}\left(\left|\nabla\tilde E_1\right|^2+\left|\nabla\tilde E_2\right|^2+\frac{1}{2}\left|\tilde n\right|^2+\frac{1}{2}\left|\tilde {\mathbf{v}}\right|^2+\tilde n\left(\left|\tilde E_1\right|^2+\left|\tilde E_2\right|^2\right)\right) dx
\\[0.4cm]
 &\displaystyle\qquad-\frac{d}{ds}\int_{\mathbb{R}^2}\left[\eta\left|\tilde E_1\right|^2\left|\tilde E_2\right|^2+\frac{\eta}{2}\left(\left(\tilde E_1\right)^2\left(\overline{\tilde E_2}\right)^2+\left(\overline{\tilde E_1}\right)^2\left(\tilde E_2\right)^2\right)\right]dx=0,
\end{array}
\right.
$$
which implies
$$
\mathcal{H}\left(\tilde E_1,\tilde E_2,\tilde n,\tilde {\mathbf{v}}\right)(t,s,x)=\mathcal{H}\left(\tilde E_1,\tilde E_2,\tilde n,  \tilde {\mathbf{v}}\right)(t,0,x).\eqno(2.24)
$$
According to conservation of Hamiltonian (1.15), we have
$$
\left.
\begin{array}{ll}
&\displaystyle\mathcal{H}\left(\tilde E_1,\tilde E_2,\tilde n,\tilde{\textbf{v}}\right)(t,0,x)
\\[0.4cm]
 &\displaystyle\qquad=\left\|\nabla\tilde E_1(t,0,x)\right\|_{L^2(\mathbb{R}^2)}^2+\left\|\nabla\tilde E_2(t,0,x)\right\|_{L^2(\mathbb{R}^2)}^2
 \\[0.4cm]
 &\displaystyle\qquad\quad+\frac{1}{2}\left\|\tilde n(t,0,x)\right\|_{L^2(\mathbb{R}^2)}^2
+\frac{1}{2}\left\|\tilde{\textbf{v}}(t,0,x)\right\|_{L^2(\mathbb{R}^2)}^2
\\[0.3cm]
&\displaystyle\qquad\quad+\int_{\mathbb{R}^2}^{}\tilde
n(t,0,x)\left(\left|\tilde E_1(t,0,x)\right|^2+\left|\tilde E_2(t,0,x)\right|^2\right)\,dx
\\[0.4cm]
&\displaystyle\qquad\quad-\eta\int_{\mathbb{R}^2}^{}\left|\tilde E_1(t,0,x)\right|^2\left|\tilde E_2(t,0,x)\right|^2\,dx
\\[0.4cm]
&\displaystyle\qquad\quad+\frac{\eta}{2}
\int_{\mathbb{R}^2}^{}\left(\left(\tilde{E_1}\left(t,0,x\right)\right)^2
\left(\overline{\tilde E_2}\left(t,0,x\right)\right)^2\right.
\\[0.4cm]
&\displaystyle\qquad\qquad\qquad\qquad\left.+\left(\overline{\tilde E_1}\left(t,0,x\right)\right)^2\left(\tilde E_2\left(t,0,x\right)\right)^2\right)\,dx
\\[0.4cm]
 &\displaystyle\qquad=\left\|\frac{1}{\lambda^2(t)}\nabla{E_1}(t)\right\|_{L^2(\mathbb{R}^2)}^2
+\left\|\frac{1}{\lambda^2(t)}\nabla{E_2}(t)\right\|_{L^2(\mathbb{R}^2)}^2
\\[0.4cm]
&\displaystyle\qquad\quad+\frac{1}{2}\left\|\frac{1}{\lambda^2(t)}
n(t)\right\|_{L^2(\mathbb{R}^2)}^2
+\frac{1}{2}\left\|\frac{1}{\lambda^2(t)}\textbf{v}(t)\right\|_{L^2(\mathbb{R}^2)}^2
\\[0.4cm]
 &\displaystyle\qquad\quad+\int_{\mathbb{R}^2}^{}\frac{1}{\lambda^4(t)}n(t)
 \left( \left|E_1(t)\right|^2+ \left|E_2(t)\right|^2\right)\,dx
 \\[0.4cm]
  &\displaystyle\qquad\quad-\eta\int_{\mathbb{R}^2}^{}
 \frac{1}{\lambda^4(t)} \left|E_1(t)\right|^2 \left|E_2(t)\right|^2\,dx
 \\[0.4cm]

 &\displaystyle\qquad\quad+\frac{\eta}{2}\int_{\mathbb{R}^2}^{}\frac{1}{\lambda^4(t)}
 \left(\left(E_1\left(t\right)\right)^2\left(\overline{E_2}\left(t\right)\right)^2
 +\left(\overline{E_1}\left(t\right)\right)^2\left(E_2\left(t\right)\right)^2
 \right)\,dx
 \\[0.4cm]
 &\displaystyle\qquad=\frac{1}{\lambda^2(t)}\mathcal{H}
 \left(E_1,E_2,n,\textbf{v}\right)(t)
 \\[0.4cm]
 &\displaystyle\qquad=\frac{1}{\lambda^2(t)}\mathcal{H}
 \left(E_{10},E_{20},n_0,\textbf{v}_0\right),
\end{array}
\right.\eqno(2.25)
$$
which is just the estimate (2.7). This finishes the proof of Proposition 2.3.
$$\eqno{\Box}$$

\section{Estimates for the rescaled Zakharov system (2.3)}

In this section, we first establish some a priori estimates for the rescaled Zakharov system (2.3) in order to gain the optimal lower bound for the blow-up rate of the finite time blow-up solution to the Zakharov system (1.1). For simplicity, we denote $\left(\tilde E_1,\tilde E_2,\tilde n,\tilde{\textbf{v}}\right)(t,s,x)$ by $\left(\tilde E_1,\tilde E_2,\tilde n,\tilde{\textbf{v}}\right)(s)$.\\
\indent  We now claim the following conclusion concerning some a priori estimates for the solution $\left(\tilde E_1,\tilde E_2,\tilde n,\tilde{\textbf{v}}\right)(s)$ to the re-scaled Zakharov system (2.3).\\

\begin{theorem}\label{3.1} (A priori estimates on $\left(\tilde E_1,\tilde E_2,\tilde n,\tilde{\textbf{v}}\right)(s)$)
\\
  Let $(E_1,E_2,n,\mathbf{v})(t)$ be a solution of the Zakharov system (1.1), $\left(\tilde E_1,\tilde E_2,\tilde n,\tilde{\textbf{v}}\right)(s) $ be a solution of the rescaled Zakharov system (2.3) and $T$ be the blow-up time. Suppose that the initial mass $(E_{10},E_{20})$ satisfies
$$
\frac{\|Q\|_{L^2(\mathbb{R}^2)}^2}{1+\eta}<\left\|E_{10}\right\|_{L^2(\mathbb{R}^2)}^2
+\left\|E_{20}\right\|_{L^2(\mathbb{R}^2)}^2
<\frac{\|Q\|_{L^2(\mathbb{R}^2)}^2}{\eta},
\eqno(3.1)
$$
where~$Q$ is the unique radial positive solution to the equation
$$
-\Delta V+V=V^3,\eqno(3.2)
$$
 then there are constants $\theta_{0}>0$,~and~$A>0$ depending only on the initial data such that
$$
\forall s\in[0,\theta_{0}),~~\left\|\left(\tilde E_1,\tilde E_2,\tilde n,\tilde{\textbf{v}}\right)(s)\right\|_{H_1}\leqslant A~~\mbox{for} ~~t~~\mbox{near}~~ T.\eqno(3.3)
$$
Furthermore, we can choose ~ $\theta_0=\tilde{c}\left(\left\|E_{10}\right\|_{L^2(\mathbb{R}^2)}^2
+\left\|E_{20}\right\|_{L^2(\mathbb{R}^2)}^2
-\frac{\|Q\|_{L^2(\mathbb{R}^2)}^2}{1+\eta}\right)^{-\frac{1}{2}}$.
\end{theorem}
\begin{remark}\label{3.2}
Theorem 3.1 is a crucial ingredient to show the main result (Theorem 1.3), whereas condition (3.1) being a key point for gaining the optimal lower bound on the  blow-up rate.  \hfill$\Box$
\end{remark}
\indent Ahead of proving Theorem 3.1, we first establish the geometrical estimates on the solution $\left(\tilde E_1,\tilde E_2,\tilde n,\tilde{\textbf{v}}\right)$ to the re-scaled Zakharov system (2.3). These estimates concern Sobolev type estimates for $\left(\tilde E_1,\tilde E_2,\tilde n,\tilde{\textbf{v}}\right)(0)$, nonvanishing properties of $\left(\tilde E_1,\tilde E_2,\tilde n\right)(0)$ and compactness properties of $\left(\tilde E_1,\tilde E_2,\tilde n\right)(0)$. We shall consider them in four portions:\\
\indent $\diamondsuit$ 3.1\quad  Sobolev Estimates on $\left(\tilde E_1(0),\tilde E_2(0),\tilde n(0),\tilde{\textbf{v}}(0)\right)$ for $t$ near $T$;\\
\indent $\diamondsuit$ 3.2 \quad Non-vanishing properties of the solutions to the re-scaled Zakharov system as $t$ near $T$;\\
 \indent $\diamondsuit$ 3.3 \quad Compactness of the solution to the re-scaled Zakharov system (2.3);\\
 \indent $\diamondsuit$ 3.4 \quad Proof of Theorem 3.1.\\
\subsection{\Large Sobolev Estimates on $\left(\tilde E_1(0),\tilde E_2(0),\tilde n(0),\tilde{\textbf{v}}(0)\right)$ for $t$ near $T$}
\qquad\\
\\
\indent The Sobolev estimates on $\left(\tilde E_1(0),\tilde E_2(0),\tilde n(0),\tilde{\textbf{v}}(0)\right)$ is given as follow.
\begin{proposition}\label{3.3}
Let $(E_1(t),E_2(t),n(t),\textbf{v}(t))$ be the finite time blow-up solution to the Cauchy problem (1.1)-(1.2) on $  t\in[0,T)$, and $T$ be the blow-up time. Suppose that the initial mass satisfies
$$
\frac{1}{1+\eta}\|Q\|_{L^2(\mathbb{R}^2)}^2<\|E_{10}\|_{L^2(\mathbb{R}^2)}^2
+\|E_{20}\|_{L^2(\mathbb{R}^2)}^2<\frac{1}{\eta}\|Q\|_{L^2(\mathbb{R}^2)}^2,\eqno(3.4)
$$
then there are constants
 $\delta_1>0,~c_1>0,~c_2>0$ and $c_{3}>0$ depending only on $(E_{10},E_{20},n_{0},\textbf{v}_{0})$, such that for
 $  t\in[T-\delta_1,T)$, the solution $\left(\tilde{E}_{1},\tilde{E}_{2},\tilde{n},\tilde{\textbf{v}}\right)(s)$ to the re-scaled Zakharov system (2.3) admits

$$ 0<c_1\leqslant\left(\left\|\nabla\tilde E_1(0)\right\|^2_{L^2(\mathbb{R}^2)}+\left\|\nabla\tilde E_2(0)\right\|^2_{L^2(\mathbb{R}^2)}\right)^{\frac{1}{2}}\leqslant c_2,\eqno(3.5)$$
\\[0.2cm]
$$ 0<c_1\leqslant\left\|\tilde n(0)\right\|_{L^2(\mathbb{R}^2)}\leqslant c_3,\eqno(3.6)$$
\\[0.2cm]
$$ ~0\leqslant\left\|\tilde{\textbf{v}}(0)\right\|_{L^2(\mathbb{R}^2)}\leqslant c_3.\eqno(3.7)$$
 \end{proposition}
\qquad \\
\begin{proof}
From (2.5) it follows that
$$
\left\{
\begin{array}{ll}
&\displaystyle\left(\left\|\nabla\tilde E_1(0)\right\|_{L^2(\mathbb{R}^2)}^2+\left\|\nabla\tilde E_2(0)\right\|_{L^2(\mathbb{R}^2)}^2\right)^{\frac{1}{2}}\leqslant 1,
\\[0.4cm]
  &\displaystyle\left\|\tilde n(0)\right\|_{L^2(\mathbb{R}^2)}\leqslant \sqrt{2},~~\left\|\tilde{\textbf{v}}(0)\right\|_{L^2(\mathbb{R}^2)}\leqslant \sqrt{2}.
  \end{array}
\right.\eqno(3.8)
$$
Note that $\lambda(t)\to+\infty$ as $t\to T$, by (2.7), there exists $\delta_1>0$ such that for any $
t\in[T-\delta_1,T)$,
$$\left|\mathcal{H}\left(\tilde E_1(0),\tilde E_2(0),\tilde n(0),\tilde{\textbf{v}}(0)\right)\right|=\left|\frac{\mathcal{H}_0}{\lambda^2(t)}\right|\leqslant \frac{1}{64}.\eqno(3.9)$$
Hence (2.5) and (2.7) yield
$$
\left.
\begin{array}{ll}
  1&\displaystyle= \int_{\mathbb{R}^2}|\nabla   \tilde E_1(0)|^2dx+\int_{\mathbb{R}^2}|\nabla \tilde E_2(0)|^2dx
\\[0.4cm]
&\displaystyle \quad+\frac{1}{2}\int_{\mathbb{R}^2}|\tilde n(0)|^2dx+\frac{1}{2}\int_{\mathbb{R}^2}|\tilde{\textbf{v}}(0)|^2 dx
\\[0.4cm]
&\displaystyle \leqslant \frac{1}{64} -\int_{\mathbb{R}^2}\tilde n(0)\left(|\tilde E_1(0)|^2+|\tilde E_2(0)|^2\right)dx
\\[0.4cm]
&\displaystyle \quad+ \frac{\eta}{2}\int_{\mathbb{R}^2}\left|\overline{\tilde E_1}(0) \tilde E_2(0) - \tilde E_1(0) \overline{\tilde E_2}(0) \right|^2dx.
\end{array}
\right.\eqno(3.10)
$$
Since $\overline{\tilde E_1}\tilde E_2 $ and $\tilde E_1\overline{\tilde E_2}$ are conjugate complex-valued functions, we claim the following estimate for the quadric term $\displaystyle \int_{\mathbb{R}^2}\left|\overline{\tilde E_1}(0)\tilde E_2(0)-\tilde E_1(0)\overline{\tilde E_2}(0) \right|^2dx$:
$$
\left.
\begin{array}{ll}
&\displaystyle
\frac{1}{2}\int_{\mathbb{R}^2}\left|\overline{\tilde E_1}(0) \tilde E_2(0) -   \tilde E_1(0) \overline{\tilde E_2}(0)\right |^2dx
  \\[0.4cm]
 &\displaystyle \qquad\leqslant 2 \int_{\mathbb{R}^2} |\tilde E_1(0)|^2| \tilde E_2(0)|^2 dx
 \\[0.4cm]
&\displaystyle \qquad  \leqslant 2 \int_{\mathbb{R}^2}  \left(\frac{|\tilde E_1(0)|^2 +| \tilde E_2(0)|^2}{2}\right)^2dx
\\[0.4cm]
&\displaystyle \qquad  =  \frac{1}{2} \int_{\mathbb{R}^2} ( |\tilde E_1(0)|^2 +| \tilde E_2(0)|^2  )^2dx.
 \end{array}
\right.\eqno(3.11)
$$
On the other hand, it follows from the H\"{o}lder inequality that
$$
\left.
\begin{array}{ll}
&\displaystyle
 \int_{\mathbb{R}^2} -\tilde n(0)\left ( |\tilde E_1(0)|^2 + |\tilde E_2(0)|^2\right)dx
 \\[0.5cm]
 &\displaystyle\qquad \leqslant  \left( b^2 \int_{\mathbb{R}^2} |\tilde n(0)|^2 dx\right)^{\frac{1}{2}}
  \left(\frac{1}{b^2} \int_{\mathbb{R}^2} \left( |\tilde E_1(0)|^2 + |\tilde E_2(0)|^2\right)^2 dx \right)^{\frac{1}{2}}
  \\[0.5cm]
  &\displaystyle\qquad \leqslant  \frac{b^2}{2} \int_{\mathbb{R}^2} |\tilde n(0)|^2dx + \frac{1}{2b^2}\int_{\mathbb{R}^2} \left( |\tilde E_1(0)|^2 + |\tilde E_2(0)|^2 \right)^2dx.
  \end{array}
\right.\eqno(3.12)
$$
Let $b^2 = \frac{1}{2}$. Combining (3.10) with (3.11) and (3.12) yields
$$\frac{3}{4} \leqslant  \frac{1}{4} \int_{\mathbb{R}^2} |\tilde n(0)|^2 dx+ \left(\frac{\eta}{2} + 1\right) \int_{\mathbb{R}^2} \left(|\tilde E_1(0)|^2 + |\tilde E_2(0)|^2\right)^2dx.\eqno(3.13)$$
Due to $\|\tilde n(0)\|_{L^{2}(\mathbb{R}^2)}^2 \leqslant 2$, (3.13) implies
$$\int_{\mathbb{R}^2} \left(|\tilde E_1(0)|^2 + |\tilde E_2(0)|^2\right)^2dx > \frac{1}{4 + 2 \eta}.\eqno(3.14)$$
Using the Gagliardo-Nirenberg inequality(Lemma 2.2), one gets the estimate for  $\displaystyle\int_{\mathbb{R}^2} \left(|\tilde E_1(0)|^2 + |\tilde E_2(0)|^2\right)^2dx$:
\\[0.2cm]
$$
\left.
\begin{array}{ll}
&\displaystyle \int_{\mathbb{R}^2} \left(|\tilde E_1(0)|^2 + |\tilde E_2(0)|^2\right)^2dx
\\[0.4cm]
&\displaystyle= \int_{\mathbb{R}^2} \left(|\tilde E_1(0)|^4 + |\tilde E_2(0)|^4\right)dx + 2 \int_{\mathbb{R}^2} |\tilde E_1(0)|^2 |\tilde E_2(0)|^2dx
  \\[0.4cm]
&\displaystyle \leqslant \frac{2 \|\tilde E_1(0)\|_{L^{2}(\mathbb{R}^2)}^2}{\|Q\|_{L^{2}(\mathbb{R}^2)}^2} \int_{\mathbb{R}^2} |\nabla \tilde E_1(0)|^2dx + \frac{2 \|\tilde E_2(0)\|_{L^{2}(\mathbb{R}^2)}^2}{\|Q\|_{L^{2}(\mathbb{R}^2)}^2} \int_{\mathbb{R}^2} |\nabla \tilde E_2(0)|^2dx
\\[0.4cm]
&\displaystyle\quad+ 2\left(\int_{\mathbb{R}^2} |\tilde E_1(0)|^4dx\right)^{\frac{1}{2}} \left(\int_{\mathbb{R}^2}  |\tilde E_2(0)|^4 dx\right)^{\frac{1}{2}}
\\[0.4cm]
&\displaystyle \leqslant  \frac{2 \|\tilde E_1(0)\|_{L^{2}(\mathbb{R}^2)}^2}{\|Q\|_{L^{2}(\mathbb{R}^2)}^2} \int_{\mathbb{R}^2} |\nabla \tilde E_1(0)|^2dx + \frac{2 \|\tilde E_2(0)\|_{L^{2}(\mathbb{R}^2)}^2}{\|Q\|_{L^{2}(\mathbb{R}^2)}^2} \int_{\mathbb{R}^2} |\nabla \tilde E_2(0)|^2dx
\\[0.4cm]
&\displaystyle\quad + 4  \frac{ \|\tilde E_1(0)\|_{L^{2}(\mathbb{R}^2)} \|\tilde E_2(0)\|_{L^{2}(\mathbb{R}^2)} }{\|Q\|_{L^{2}(\mathbb{R}^2)}^2} \| \nabla \tilde E_1(0)\|_{L^{2}(\mathbb{R}^2)} \| \nabla \tilde E_2(0)\|_{L^{2}(\mathbb{R}^2)}
\\[0.4cm]
 &\displaystyle \leqslant \frac{2}{\|Q\|_{L^{2}(\mathbb{R}^2)}^2}
  \\[0.4cm]
 &\displaystyle\quad\cdot \left( \|\tilde E_1(0)\|_{L^{2}(\mathbb{R}^2)}^2  \| \nabla \tilde E_1(0)\|_{L^{2}(\mathbb{R}^2)}^2 + \|\tilde E_2(0)\|_{L^{2}(\mathbb{R}^2)}^2\|\nabla \tilde E_1(0)\|_{L^{2}(\mathbb{R}^2)}^2\right.
\\[0.4cm]
 &\displaystyle\quad  \left.+ \|\tilde E_1(0)\|_{L^{2}(\mathbb{R}^2)}^2 \|\nabla \tilde E_2(0)\|_{L^{2}(\mathbb{R}^2)}^2
 + \|\tilde E_2(0)\|_{L^{2}(\mathbb{R}^2)}^2 \|\nabla \tilde E_2(0)\|_{L^{2}(\mathbb{R}^2)}^2 \right)
 \\[0.4cm]
&\displaystyle =  \frac{2}{\|Q\|_{L^{2}(\mathbb{R}^2)}^2}\left (\| \tilde E_1(0)\|_{L^{2}(\mathbb{R}^2)}^2 + \| \tilde E_2(0)\|_{L^{2}(\mathbb{R}^2)}^2\right)
\\[0.4cm]
 &\displaystyle\qquad\qquad\qquad\cdot\left(\| \nabla \tilde E_1(0)\|_{L^{2}(\mathbb{R}^2)}^2 + \| \nabla \tilde E_2(0)\|_{L^{2}(\mathbb{R}^2)}^2\right).
\end{array}
\right.\eqno(3.15)
$$
This together with (3.14) yields
$$
\left.
\begin{array}{ll}
\displaystyle
\frac{1}{8+4 \eta} &\displaystyle\leqslant \frac{\left\|\tilde E_1(0)\right\|_{L^{2}(\mathbb{R}^2)}^2 + \left\|  \tilde E_2(0)\right\|_{L^{2}(\mathbb{R}^2)}^2}{\|Q\|_{L^{2}(\mathbb{R}^2)}^2}
\\[0.6cm]
&\quad\displaystyle\cdot \left(\left\| \nabla \tilde E_1(0)\right\|_{L^{2}(\mathbb{R}^2)}^2 + \left\| \nabla \tilde E_2(0)\right\|_{L^{2}(\mathbb{R}^2)}^2 \right),
\end{array}
\right.\eqno(3.16)
$$
and the conclusion (3.5) follows from (2.5) and the mass identity (2.6).\\
\indent In the following we prove the  conclusions (3.6) and (3.7). In view of condition (3.4), we can assume that there exists a sufficiently small $\delta_{0}$ with $0<\delta_{0}<\frac{1}{1+\eta}$ such that
$$\| E_1(0)\|_{L^{2}(\mathbb{R}^2)}^2+\| E_2(0)\|_{L^{2}(\mathbb{R}^2)}^2 < \frac{1-\delta_{0}}{\eta}  \|Q\|_{L^{2}(\mathbb{R}^2)}^2.\eqno(3.17)$$
Then from (3.8) it follows that
$$
\left.
\begin{array}{ll}
&\displaystyle\frac{\eta}{2}\int_{\mathbb{R}^2}\left|\overline{\tilde E_1}(0) \tilde E_2(0)  - \tilde E_1(0) \overline{\tilde E_2}(0) \right|^2dx
\\[0.4cm]
&\displaystyle\qquad\leqslant  2 \eta\int_{\mathbb{R}^2} |\tilde E_1(0)|^2| \tilde E_2(0)|^2 dx \\[0.4cm]
&\displaystyle \qquad\leqslant 2 \eta \left(\int_{\mathbb{R}^2} |\tilde E_1(0)|^4dx\right)^\frac{1}{2} \left(\int_{\mathbb{R}^2} |\tilde E_2(0)|^4dx\right)^\frac{1}{2}
\\[0.4cm]
&\displaystyle\qquad\leqslant 4 \eta \frac{\|\tilde E_1(0)\|_{L^{2}(\mathbb{R}^2)} \|\tilde E_2(0)\|_{L^{2}(\mathbb{R}^2)} \|\nabla \tilde E_1(0)\|_{L^{2}(\mathbb{R}^2)}\|\nabla \tilde E_2(0)\|_{L^{2}(\mathbb{R}^2)} } {\|Q\|_{L^{2}(\mathbb{R}^2)}^2}
\\[0.4cm]
&\displaystyle\qquad \leqslant \eta \bigg(\|\nabla \tilde E_1(0)\|_{L^{2}(\mathbb{R}^2)}^2 + \|\nabla\tilde E_2(0)\|_{L^{2}(\mathbb{R}^2)}^2 \bigg)
\\[0.4cm]
&\displaystyle \qquad\qquad\cdot\left(\frac{\| \tilde E_1(0)\|_{L^{2}(\mathbb{R}^2)}^2 + \|\tilde E_2(0)\|_{L^{2}(\mathbb{R}^2)}^2}{\|Q\|_{L^{2}(\mathbb{R}^2)}^2}\right)
\\[0.4cm]
&\displaystyle \qquad\leqslant 1 - \delta_0.
  \end{array}
\right.\eqno(3.18)
$$
Combining (3.18) with (2.2), (2.5) and (2.7) yields
$$
\left.
\begin{array}{ll}
\delta_0 & \displaystyle\leqslant - \int_{\mathbb{R}^2}\tilde n(0) \left(|\tilde E_1(0)|^2 + |\tilde E_2(0)|^2\right)dx
 \\[0.4cm]
& \displaystyle\leqslant \left(\int_{\mathbb{R}^2}\tilde n^2(0)dx\right)^{\frac{1}{2}}
\left( \int_{\mathbb{R}^2} \left( |\tilde E_1(0)|^2 + |\tilde E_2(0)|^2\right)^2dx \right)^{\frac{1}{2}}.
\end{array}
\right.\eqno(3.19)
$$
Recalling (3.15), (3.17) yields

$$
\left.
\begin{array}{ll}
&\displaystyle\int_{\mathbb{R}^2} \left( |\tilde E_1(0)|^2 + |\tilde E_2(0)|^2\right)^2dx
\\[0.3cm]
&\displaystyle\qquad \leqslant 2 \bigg(\|\nabla \tilde E_1(0)\|_{L^{2}(\mathbb{R}^2)}^2 + \| \nabla \tilde E_2(0)\|_{L^{2}(\mathbb{R}^2)}^2 \bigg)
\\[0.4cm]
&\displaystyle\qquad\qquad\cdot\left(\frac{\| \tilde E_1(0)\|_{L^{2}(\mathbb{R}^2)}^2 + \|\tilde E_2(0)\|_{L^{2}(\mathbb{R}^2)}^2}{\|Q\|_{L^{2}(\mathbb{R}^2)}^2} \right)
\\[0.5cm]
&\displaystyle\qquad\leq \frac{2-2 \delta_0}{\eta},
\end{array}
\right.\eqno(3.20)
$$
Note that (3.19), for any fixed small $\delta_0>0$, there exists a constant $c_1>0$ such that
$$c_1 \leqslant \frac{\eta}{2} \frac{\delta_0^2}{1-\delta_0} \leqslant \int_{\mathbb{R}^2}|\tilde n(0)|^2dx.\eqno(3.21)$$
Note that for $0<\delta<\frac{1}{1+\eta}$ and $\eta>0$, $\frac{\eta}{2} \frac{\delta_0^2}{1-\delta_0}<2$, the upper bounds for $\|\nabla \tilde E_1(0)\|_{L^{2}(\mathbb{R}^2)}^2+\|\nabla  \tilde E_2(0)\|_{L^{2}(\mathbb{R}^2)}^2$,~~$\|\tilde n(0)\|_{L^{2}(\mathbb{R}^2)}^2$~
and\\
$\|  \tilde {\textbf{v}}(0)\|_{L^{2}(\mathbb{R}^2)}^2$ follow from  (2.5). The proof of Proposition 3.3 is completed.

\end{proof}
\indent Due to the condition (2.5) for the re-scaled Zakharov system (2.3), using Proposition 3.3 and the following scaling properties:
$$
\left.
\begin{array}{ll}
&\displaystyle\left\|\nabla \tilde E_j(0)\right\|_{L^{2}(\mathbb{R}^2)}= \frac{1}{\lambda(t)} \left\|\nabla E_j(t)\right\|_{L^{2}(\mathbb{R}^2)}, ~~j=1,2,
\\[0.4cm]
&\displaystyle \|\tilde n(0)\|_{L^{2}(\mathbb{R}^2)} = \frac{1}{\lambda(t)} \| n(t)\|_{L^{2}(\mathbb{R}^2)},~~ \| \tilde{\mathbf{ v}}(0) \|_{L^{2}(\mathbb{R}^2)} = \frac{1}{\lambda(t)} \|  \mathbf{v} (t)\|_{L^{2}(\mathbb{R}^2)},
\end{array}
\right.\eqno(3.22)
$$
we claim the following Sobolev-type estimates for the solution $\left(E_{1}(t),E_{2}(t),n(t),\mathbf{v}(t)\right)$
  to the Zakharov system (1.1).\\
  \begin{corollary}\label{3.4}
  Under the assumptions in Proposition 3.3, there exists constants $\delta_{1}>0$, $c^{*}_{1}$,~$c^{*}_{2}$ depending only on initial data (1.2) such that for $t \in [T-\delta_1,T)$, there hold:
  $$c^{*}_1 \|n(t)\|_{L^{2}(\mathbb{R}^2)}\leqslant\left( \|\nabla E_1(t)\|_{L^{2}(\mathbb{R}^2)}^2 + \|\nabla E_2(t)\|_{L^{2}(\mathbb{R}^2)}^2\right)^{\frac{1}{2}} \leqslant \frac{1}{c^{*}_1} \|n(t)\|_{L^{2}(\mathbb{R}^2)},\eqno(3.23)
$$
  $$\|v(t)\|_{L^{2}(\mathbb{R}^2)}  \leqslant \frac{1}{c^{*}_1} \|n(t)\|_{L^{2}(\mathbb{R}^2)},\eqno(3.24)$$

  $$
\left.
\begin{array}{ll}
 \displaystyle c^{*}_2 \|n(t)\|_{L^{2}(\mathbb{R}^2)}& \leqslant \|(E_1,E_2,n,v)(t)\|_{H^{1}(\mathbb{R}^2)\times H^{1}(\mathbb{R}^2)
 \times L^{2}(\mathbb{R}^2)\times L^{2}(\mathbb{R}^2)}
  \\[0.3cm]
  \displaystyle &\leqslant \frac{1}{c^{*}_2} \|n(t)\|_{L^{2}(\mathbb{R}^2)}.
  \end{array}
\right.\eqno(3.25)
$$
 \end{corollary}
 \begin{proof}
 From Proposition 3.3 it follows that $c_{1}<c_{2}$. Taking $c^{*}_1=\frac{c_{1}}{c_{2}}$ and
$c^{*}_2=\sqrt{\frac{2c_{1}^{2}}{4c_{2}^{2}+c_{1}^{2}}}$
 yields the conclusion of Corollary 3.4.
 \end{proof}
 \qquad\\
 \subsection{Non-vanishing Properties of the Solutions to the Re-scaled Zakharov System as $t$ near $T$ }
 \qquad\\
 \\
 \indent We now consider the non-vanishing properties of the solution $\left(\tilde E_1(s),\tilde E_2(s), \tilde n(s)\right)$ of the rescaled magnetic Zakharov system (2.3) for $s=0$, i.e., the non-vanishing properties of $\left(\tilde E_1(0),\tilde E_2(0), \tilde n(0)\right)$.
\begin{proposition}\label{3.5}
For any $t\in[0,T)$,  suppose that $\left(E_1(t),E_2(t),n(t),\textbf{v}(t)\right)$ and\\ $\left(\tilde E_1(0),\tilde E_2(0), \tilde n(0),\tilde {\textbf{v}}(0)\right)$ is the finite time blow-up solution to the Cauchy problem (1.1)-(1.2), the initial data $\left(E_{10}(x),E_{20}(x)\right)$ satisfy condition (3.4) and $T$ be the blow-up time. Then we claim:\\
\\
(1)\quad There exist constants $R_1>0$ and $\beta_1>0$ depending only on ~$\left\|E_{10}\right\|_{L^2(\mathbb{R}^2)}$ and $\left\|E_{20}\right\|_{L^2(\mathbb{R}^2)}$
such that for a sequence $x(t)\in\mathbb{R}^2$ one has
$$\liminf_{t\to T}\left(\left\|\tilde E_1\left(0,x\right)\right\|_{L^2\left(|x-x\left(t\right)|\leqslant R_1\right)}^2+\left\|\tilde E_2\left(0,x\right)\right\|_{L^2\left(|x-x\left(t\right)|\leqslant R_1\right)}^2\right)^{\frac{1}{2}}\geqslant\beta_1,\eqno(3.26)$$

$$\liminf_{t\to T}\left\|\tilde n(0,x)\right\|_{L^2(|x-x(t)|\leqslant R_1)}\geqslant\beta_1.\eqno(3.27)
 $$
 (2)\quad Let $\left(\tilde{E}_{1n},\tilde{E}_{2n},\tilde{n}_n\right)(s)$ be a sequence satisfying the estimates as follows :
$$
\left\|\tilde{E}_{1n}(0)\right\|_{L^2(\mathbb{R}^2)}^2+\left\|\tilde{E}_{2n}(0)\right\|_{L^2(\mathbb{R}^2)}^2
\leqslant\left\|E_{10}\right\|_{L^2(\mathbb{R}^2)}^2+\left\|E_{20}\right\|_{L^2(\mathbb{R}^2)}^2,\eqno(3.28)
$$

$$
c_1\leqslant\int_{\mathbb{R}^2}^{}\left|\nabla \tilde{E}_{1n}(0)\right|^2dx+\int_{\mathbb{R}^2}^{}\left|\nabla \tilde{E}_{2n}(0)\right|^2dx\leqslant c_2,\eqno(3.29)
$$

$$
c_1\leqslant\int_{\mathbb{R}^2}^{}|\tilde{n}_n(0)|^2dx\leqslant c_2,\eqno(3.30)
$$

$$
\limsup_{t\to T}\mathcal{H}\left(\tilde{E}_{1n}(0),\tilde{E}_{2n}(0),\tilde{n}_n(0),0\right)\leqslant 0.\eqno(3.31)
$$
Then there exist  $\beta_1>0$ and $R_1>0$ depending only on $\|E_{10}\|_{L^2(\mathbb{R}^2)},~\|E_{20}\|_{L^2(\mathbb{R}^2)}$, $c_1>0$ and $c_2>0$
 such that for a sequence $x_n\in\mathbb{R}^2$,
 $$
\lim_{n\to+\infty}\left(\left\|\tilde{E}_{1n}\right\|_{L^2\left(|x-x_n|\leqslant R_1\right)}^2+\left\|\tilde{E}_{2n}\right\|_{L^2\left(|x-x_n|\leqslant R_1\right)}^2\right)^{\frac{1}{2}}\geqslant\beta_1>0,\eqno(3.32)$$

$$
\lim_{n\to+\infty}\left\|\tilde{n}_n\right\|_{L^2(|x-x_n|\leqslant R_1)}\geqslant\beta_1>0.\eqno(3.33)
$$
\end{proposition}
  The proof of this proposition is similar to Proposition 3.6 in \cite{25Frank}. However, it is much more complicated here since the higher-order magnetic field terms are essentially involved. Proposition 3.5 will be proven step by step later.\\
  \\
\indent We first claim:\\
\begin{proposition}\label{3.6}
Assume there is $m_k= m_k\left(\|E_{10}\|_{L^2(\mathbb{R}^2)},\|E_{20}\|_{L^2(\mathbb{R}^2)}\right) >0$ such that the sequences $\left(E_{1k},E_{2k},n_k,\textbf{v}_k\right)\in H^1(\mathbb{R}^2)\times H^1(\mathbb{R}^2)\times L^2(\mathbb{R}^2) \times
L^2(\mathbb{R}^2)$ satisfy
$$\|E_{1k}\|_{L^2(\mathbb{R}^2)}^2+\|E_{2k}\|_{L^2(\mathbb{R}^2)}^2
=\|E_{10}\|_{L^2(\mathbb{R}^2)}^2+\|E_{20}\|_{L^2(\mathbb{R}^2)}^2>0,\eqno(3.34)$$
\\
and there exist constants $R_0>0$ and $\delta'_0>0$ such that
\\
$$
\sup_{y\in\mathbb{R}^2}\int_{|x-y|<R_0}\left(|E_{1k}(x)|^2+|E_{2k}(x)|^2
\right)dx\leqslant\frac{\|Q\|_{L^2(\mathbb{R}^2)}^2}{1+\eta}-\delta'_0,\eqno(3.35)
$$
or
$$
\sup_{y\in\mathbb{R}^2}\int_{|x-y|<R_0}|n_k(x)|dx\leqslant m_k-\delta'_0.\eqno(3.36)
$$
\\
Then there exist constants $C_1>0,~~C_2>0$ such that
$$
\left.
\begin{array}{ll}
&\displaystyle
-C_1+ C_2\int_{\mathbb{R}^2}\left(|\nabla E_{1k}|^2+|\nabla E_{2k}|^2+\frac{1}{2}|n_k|^2+\frac{1}{2}|\textbf{v}_k|^2\right)dx
\\[0.4cm]
&\displaystyle\qquad\qquad\leqslant\mathcal{H}(E_{1k},E_{2k},n_k,\textbf{v}_k),\end{array}
\right.\eqno(3.37)
$$
where $\mathcal{H}$ is defined by (1.3).
\end{proposition}
\qquad\\
{\bf Proof of Proposition 3.6.}\\

\indent We first define two functionals as follows:
$$
\left.
\begin{array}{ll}
 \displaystyle\mathcal{E}(E_1,E_2) &\displaystyle \triangleq \| \nabla E_1 \|_{L^{2}(\mathbb{R}^2)}^2 + \| \nabla E_2 \|_{L^{2}(\mathbb{R}^2)}^2 - \frac{1}{2} \int_{\mathbb{R}^2} \left(|E_1|^2 + |E_2|^2\right)^2 dx
\\[0.3cm]
&\displaystyle\quad - \frac{\eta}{2} \int_{\mathbb{R}^2}\left |\overline{E_1}E_2 - E_1 \overline{E_2}\right|^2dx,\\[0.3cm]
\end{array}
\right.\eqno(3.38)
$$
\\
$$\mathcal{H}_1(E_1,E_2,n) \triangleq  \mathcal{E}(E_1,E_2)+ \frac{1}{2} \int_{\mathbb{R}^2} \left(n + |E_1|^2 + |E_2|^2\right)^2dx.\eqno(3.39)$$
 \\
Let
$$
\left\{
\begin{array}{ll}
&\displaystyle
\tilde E_{1k}(x)=\frac{1}{\lambda_k}E_{1k}\left(\frac{x}{\lambda_k}\right),
\\[0.4cm]
&\tilde E_{2k}(x)=\frac{1}{\lambda_k}E_{2k}\left(\frac{x}{\lambda_k}\right),
\\[0.3cm]
&\tilde n_{k}(x)=\frac{1}{\lambda_k^2}n_{k}\left(\frac{x}{\lambda_k}\right),
\end{array}
\right.\eqno(3.40)
$$
where
$$\lambda_k^2=\|\nabla E_{1k}\|^{2}_{L^{2}(\mathbb{R}^{2})}+\|\nabla E_{2k}\|^{2}_{L^{2}(\mathbb{R}^{2})}+\frac{1}{2}\|n_{k}\|^{2}_{L^{2}(\mathbb{R}^{2})}.
\eqno(3.41)$$
We continue the proof of Proposition 3.6 through four steps.\\
\\
{\bf Step 1. \quad A non-vanishing property of $\left(\tilde E_{1k},\tilde E_{2k},\tilde n_{k}\right)$}\\
\\
\begin{lemma}\label{3.7}
For the sequences $(E_{1k},E_{2k},n_{k})$ introduced in Proposition 3.6, assume there is a sequence $(\tilde E_{1k},\tilde E_{2k},\tilde n_{k})\in H^1(\mathbb{R}^2)\times H^1(\mathbb{R}^2)\times L^2(\mathbb{R}^2)$ such that as $k\to+\infty$,  the following estimates hold:
$$\mathcal{H}\left(\tilde E_{1k},\tilde E_{2k},\tilde n_{k},0\right)\leqslant 0,\eqno(3.42)$$
$$\int_{\mathbb{R}^2}\left(|\tilde E_{1k}|^2+|\tilde E_{2k}|^2\right)dx\to c_1>0,\eqno(3.43)$$
$$\int_{\mathbb{R}^2}\left(|\nabla \tilde E_{1k}|^2+|\nabla \tilde E_{2k}|^2\right)dx+\frac{1}{2}\int_{\mathbb{R}^2}|\tilde n_{k}|^2dx\to
c_2>0,\eqno(3.44)$$
$$
\left.
\begin{array}{ll}
&\displaystyle\int_{\mathbb{R}^2}\tilde n_{k}\left(|\tilde E_{1k}|^2+|\tilde E_{2k}|^2\right)dx
-\eta\int_{\mathbb{R}^2}|\tilde E_{1k}|^2|\tilde E_{2k}|^2dx
\\[0.3cm]
 &\displaystyle\quad+\frac{\eta}{2}\int_{\mathbb{R}^2}
\left(\left(\tilde E_{1k}\right)^2\left(\overline{\tilde E_{2k}}\right)^2
+\left(\overline{\tilde E_{1k}}\right)^2\left(\tilde E_{2k}\right)^2\right)dx\to -c_3<0.
\end{array}
\right.\eqno(3.45)
$$
Then there exist a constant $c_4=c_4(c_1,c_2,c_3)>0$ and a sequence $x_k\in\mathbb{R}^2$ such that
$$\int_{|x-x_k|<1}|\tilde n_{k}|dx>c_4.\eqno(3.46)$$
\end{lemma}
\begin{proof}
By (3.40), we claim that there exists a sequence $x_k\in\mathbb{R}^2$ such that
$$
\left.
\begin{array}{ll}
&\displaystyle
 \int_{C_k}-\tilde n_{k}\left(|\tilde E_{1k}|^2+|\tilde E_{2k}|^2\right)dx+\eta\int_{C_k}\left|\tilde E_{1k}\right|^2\left|\tilde E_{2k}\right|^2dx
 \\[0.4cm]
&\displaystyle \qquad-\frac{\eta}{2}\int_{C_k}\left(\left(\tilde E_{1k}\right)^2\left(\overline{\tilde E_{2k}}
\right)^2+\left(\overline{\tilde E_{1k}}\right)^2\left(\tilde E_{2k}\right)^2\right)dx
\\[0.4cm]
&\displaystyle \geqslant  q\cdot\int_{C_k}\left[\left(|\nabla \tilde E_{1k}|^2+|\nabla \tilde E_{12k}|^2\right)+\left(|\tilde E_{1k}|^2+|\tilde E_{2k}|^2\right)+\frac{1}{2}\left|\tilde n_{k}\right|^2\right]dx,
\end{array}
\right.\eqno(3.47)
$$
for $k$ large enough , where $C_k$ is the square of center $x_k$ and~$\displaystyle q=\frac{c_3}{c_0(c_1+c_2)}$ with $c_0>1$ is a fixed constant. Otherwise,  one would obtain
$$
\left.
\begin{array}{ll}
&\displaystyle\int_{C_k}-\tilde n_{k}\left(|\tilde E_{1k}|^2+|\tilde E_{2k}|^2\right)dx
+\eta\int_{C_k}\left|\tilde E_{1k}\right|^2\left|\tilde E_{2k}\right|^2dx
\\[0.4cm]
&\displaystyle\quad-\frac{\eta}{2}\int_{\mathbb{R}^2}\left(\left(\tilde E_{1k}\right)^2
\left(\overline{\tilde E_{2k}}\right)^2+\left(\overline{\tilde E_{1k}}\right)^2
\left(\tilde E_{2k}\right)^2\right)dx
\\[0.4cm]
&\displaystyle<  q\cdot\int_{C_k}\left[\left(|\nabla \tilde E_{1k}|^2+|\nabla \tilde E_{2k}|^2\right)+\left(|\tilde E_{1k}|^2+|\tilde E_{2k}|^2\right)+\frac{1}{2}|\tilde n_{k}|^2\right]dx.
\end{array}
\right.\eqno(3.48)
$$
Let $k\to+\infty$, (3.48) yields $\displaystyle c_3< q(c_1+c_2)=\frac{c_3}{c_0}~(c_0>1)$, which is a contradiction.
 \\
\indent We now claim the following conclusion.\\
\\
{\bf Conclusion I:} There exist constants
\\
$$c^*_1=\frac{2\sqrt{2}q^2}{1+q\eta}\|Q\|_{L^2(\mathbb{R}^2)}>0,~~ c^*_2=\frac{4q^3}{1+q\eta}\|Q\|_{L^2(\mathbb{R}^2)}>0,$$
$$ c^*_3=\varepsilon c^*_1>0~~\mbox{with}~~\varepsilon=\dfrac{\sqrt{2}q\|Q\|_{L^{2}(C_{k})}}{q\eta+1}$$  such that
$$\left[\int_{C_k}\left(\left|\tilde E_{1k}\right|^2+\left|\tilde E_{2k}\right|^2\right)^2dx\right]^{\frac{1}{2}}\geqslant c^*_1>0,\eqno(3.49)$$
$$
\left.
\begin{array}{ll}
&\displaystyle\int_{C_k}-\tilde n_{k}\left(\left|\tilde E_{1k}\right|^2+\left|\tilde E_{2k}\right|^2\right)dx+
\eta\int_{C_k}\left|\tilde E_{1k}\right|^2\left|\tilde E_{2k}\right|^2dx
\\[0.4cm]
&\displaystyle\qquad-\frac{\eta}{2}\int_{C_k}
\left(\left(\tilde E_{1k}\right)^2\left(\overline{\tilde E_{2k}}\right)^2+\left(\overline{\tilde E_{1k}}\right)^2\left(\tilde E_{2k}\right)^2\right)dx\geqslant c^*_2>0,
\end{array}
\right.\eqno(3.50)
$$
$$\int_{C_k}-\tilde n_{k}\left(\left|\tilde E_{1k}\right|^2+\left|\tilde E_{2k}\right|^2\right)dx\geqslant c^*_3>0.\eqno(3.51)$$
{\bf Proof of Conclusion I.}\\
\\
\indent Lemma 2.2, Cauchy-Schwartz inequality: $\displaystyle 2ab\leqslant \frac{(a +b)^{2}}{2} ~(a,b>0)$ and Young's inequality give

$$
\left.
\begin{array}{ll}
&\displaystyle\int_{C_k}\left(\left|\nabla \tilde E_{1k}\right|^2+\left|\nabla \tilde E_{2k}\right|^2\right)dx+\int_{C_k}\left(\left|\tilde E_{1k}\right|^2+\left|\tilde E_{2k}\right|^2\right)dx
\\[0.3cm]
&\displaystyle\qquad\geqslant\sqrt{2}\|Q\|_{L^2\left(\mathbb{R}^2\right)}
\left[\int_{C_k}\left(\left|\tilde E_{1k}\right|^2+\left|\tilde E_{2k}\right|^2\right)^2dx\right]^{\frac{1}{2}}.
\end{array}
\right.\eqno(3.52)
$$

$$
\left.
\begin{array}{ll}
&\displaystyle q\cdot\sqrt{2}\|Q\|_{L^2(C_k)}\left[\int_{C_k}\left(\left|\tilde E_{1k}\right|^2+\left|\tilde E_{2k}\right|^2\right)^2dx
\right]^{\frac{1}{2}}+\frac{q}{2}\left\|\tilde n_{k}\right\|^2_{L^2\left(C_k\right)}
\\[0.4cm]
&\displaystyle\quad\leqslant\int_{C_k}-\tilde n_{k}\left(\left|\tilde E_{1k}\right|^2+\left|\tilde E_{2k}\right|^2\right)dx
+\eta\int_{C_k}\left|\tilde E_{1k}\right|^2\left|\tilde E_{2k}\right|^2dx
\\[0.4cm]
&\displaystyle\qquad-\frac{\eta}{2}\int_{C_k}\left(\left(\tilde E_{1k}\right)^2
\left(\overline{\tilde E_{2k}}\right)^2+\left(\overline{\tilde E_{1k}}\right)^2
\left(\tilde E_{2k}\right)^2\right)dx
\\[0.4cm]
&\displaystyle\quad\leqslant\frac{q}{2}\int_{C_k}\left|\tilde n_{k}\right|^2dx+\frac{1}{2q}
\int_{C_k}\left(\left|\tilde E_{1k}\right|^2+\left|\tilde E_{2k}\right|^2\right)^2dx
\\[0.4cm]
&\displaystyle\qquad+\frac{\eta}{2}\int_{C_k}\left(\left|\tilde E_{1k}\right|^2+\left|\tilde E_{2k}\right|^2\right)^2dx.
\end{array}
\right.\eqno(3.53)
$$
That is,
$$\left(\frac{1+\eta q}{2q}\right)\left[\int_{C_k}\left(\left|\tilde E_{1k}\right|^2+\left|\tilde E_{2k}\right|^2\right)^2dx\right]
^{\frac{1}{2}}\geqslant\sqrt{2}q\|Q\|_{L^2\left(C_k\right)}.\eqno(3.54)$$
This yields (3.49).\\
\indent Similarly, (3.53) and (3.54) imply
$$
\left.
\begin{array}{ll}
&\displaystyle\int_{C_k}-\tilde n_{k}\left(\left|\tilde E_{1k}\right|^2+\left|\tilde E_{2k}\right|^2\right)dx+
\eta\int_{C_k}\left|\tilde E_{1k}\right|^2\left|\tilde E_{2k}\right|^2dx
\\[0.4cm]
&\displaystyle\qquad-\frac{\eta}{2}\int_{C_k}\left(\left(\tilde E_{1k}\right)^2
\left(\overline{\tilde E_{2k}}\right)^2+\left(\overline{\tilde E_{1k}}\right)^2\left(\tilde E_{2k}\right)^2
\right)dx
\\[0.4cm]
&\displaystyle\quad\geqslant  q\sqrt{2}\|Q\|_{L^2\left(C_k\right)}\left[\int_{C_k}
\left(\left|\tilde E_{1k}\right|^2+\left|\tilde E_{2k}\right|^2\right)^2dx\right]^{\frac{1}{2}}
\\[0.4cm]
&\displaystyle\quad\geqslant \frac{4q^3\|Q\|_{L^2\left(C_k\right)}^2}{1+q\eta}=c^*_2>0.
\end{array}
\right.\eqno(3.55)
$$
So (3.50) is true.\\
\indent In the following we prove (3.51). By (3.47), one obtains
$$
\left.
\begin{array}{ll}
&\displaystyle q \int_{C_k}\left[\left(\left|\nabla \tilde E_{1k}\right|^2+\left|\nabla \tilde E_{2k}\right|^2\right)
 +\left(\left| \tilde E_{1k}\right|^2+\left|\tilde E_{2k}\right|^2\right)\right]dx
  \\[0.4cm]
 &\displaystyle\qquad\qquad\qquad+\frac{q}{2}\int_{C_k}\left|\tilde n_{k}\right|^2dx
 \\[0.4cm]
 &\displaystyle\quad\leqslant \frac{\eta}{2}\int_{C_k}\left|\overline{\tilde E_{1k}}\tilde E_{2k}
 -\tilde E_{1k}\overline{\tilde E_{2k}}\right|^{2}dx+\int_{C_k}-\tilde n_{k}\left(\left|\tilde E_{1k}\right|^2+\left|\tilde E_{2k}\right|
 ^2\right)dx.
\end{array}
\right.\eqno(3.56)
$$
From Lemma 2.2, it follows
$$\int_{C_k}\left(\left|\nabla \tilde E_{1k}\right|^2+\left|\nabla \tilde E_{2k}\right|^2\right)dx\geqslant
\dfrac{\|Q\|^{2}_{L^{2}(C_k)}\int_{C_k}\left(\left| \tilde E_{1k}\right|^2+\left|  \tilde E_{2k}\right|^2\right)^2dx}
{2\int_{C_k}\left(\left|  \tilde E_{1k}\right|^2+\left| \tilde E_{2k}\right|^2\right)dx}.$$

Note that
$$
\left.
\begin{array}{ll}
&\displaystyle\int_{C_k}\left[\left(\left|\nabla \tilde E_{1k}\right|^2+\left|\nabla \tilde E_{2k}\right|^2\right)
 +\left(\left|\tilde E_{1k}\right|^2+\left|\tilde E_{2k}\right|^2\right)\right]dx\qquad\qquad
 \\[0.4cm]
 &\displaystyle\geqslant 2\left(\int_{C_k}\left( \left|\nabla \tilde E_{1k}\right|^2+\left|\nabla \tilde E_{2k}\right|^2\right)dx\right)^{\frac{1}{2}}\left(\int_{C_k}\left(\left|\tilde E_{1k}\right|^2+\left|\tilde E_{2k}\right|^2\right)dx\right)^{\frac{1}{2}},
\end{array}
\right.
$$
(3.53),(3.54),(3.55) and (3.56) give
$$
\left.
\begin{array}{ll}
&\displaystyle\dfrac{\sqrt{2}q\|Q\|_{L^{2}(C_k)}}{\frac{\eta}{2}+\frac{1}{2q}}
 \sqrt{2}q\|Q\|_{L^{2}(C_k)}+\frac{q}{2}\int_{C_k}\left|\tilde n_{k}\right|^2dx
   \\[0.4cm]
 &\displaystyle\qquad\qquad-\frac{\eta}{2}
 \int_{C_k}\left(\left| \tilde E_{1k}\right|^2 + \left| \tilde E_{2k}\right|^2\right)^{2}dx
  \\[0.4cm]
 &\displaystyle\qquad \leqslant \int_{C_k}-\tilde n_{k}\left(\left|\tilde E_{1k}\right|^2+\left|\tilde E_{2k}\right|
 ^2\right)dx.
\end{array}
\right.\eqno(3.57)
$$
Assume that there exists an $\varepsilon>0$ such that
\\
$$\dfrac{\sqrt{2}q }{\frac{\eta}{2}+\frac{1}{2q}}
 \|Q\|_{L^{2}(C_k)}\leqslant\left(\int_{C_k}\left( \left| \tilde E_{1k}\right|^2 + \left|  \tilde E_{2k}\right|^2\right)^{2}dx\right)^{\frac{1}{2}}\leqslant\dfrac{\sqrt{2}q \|Q\|_{L^{2}(C_k)}-\varepsilon }{\frac{\eta}{2}}.\eqno(3.58)$$
 \\
This yields
$$\displaystyle\eta\leqslant\dfrac{2\sqrt{2}q \|Q\|_{L^{2}(C_k)}-2\varepsilon}{\left(\int_{C_k}\left(|\tilde E_{1k}|^2 + |  \tilde E_{2k}|^2\right)^{2}dx\right)^{\frac{1}{2}}}\leqslant
\dfrac{2\sqrt{2}q \|Q\|_{L^{2}(C_k)}-2\varepsilon}{\sqrt{2}q \|Q\|_{L^{2}(C_k)}}\left(\frac{\eta}{2}+\frac{1}{2q}\right),\eqno(3.59)$$
that is,
$$
\left.
\begin{array}{ll}
 \varepsilon&\displaystyle\leqslant  \sqrt{2}q \|Q\|_{L^{2}(C_k)}\left(1-\dfrac{\eta}{\eta+\frac{1}{q}}\right)
 \\[0.4cm]
&\displaystyle= \sqrt{2}q \|Q\|_{L^{2}(C_k)}\dfrac{1}{q\eta+1}=\dfrac{ \sqrt{2}q \|Q\|_{L^{2}(C_k)}}{q\eta+1}.
\end{array}
\right.\eqno(3.60)
$$
Taking $\varepsilon=\dfrac{\sqrt{2}q \|Q\|_{L^{2}(C_k)}}{q\eta+1}$ in (3.58), one obtains
$$
\left.
\begin{array}{ll}
&\displaystyle\sqrt{2}q \|Q\|_{L^{2}(C_k)}-\frac{\eta}{2}\left(\int_{C_k}\left(|  \tilde E_{1k}|^2 + | \tilde E_{2k}|^2\right)^{2}dx\right)^{\frac{1}{2}}
 \\[0.4cm]
 &\displaystyle\qquad\geqslant\sqrt{2}q \|Q\|_{L^{2}(C_k)}-\dfrac{\sqrt{2}q \|Q\|_{L^{2}(C_k)}-\dfrac{ \sqrt{2}q \|Q\|_{L^{2}(C_k)}}{q\eta+1} }{\frac{\eta}{2}}\cdot \frac{\eta}{2}
 \\[0.4cm]
 &\displaystyle\qquad=\dfrac{ \sqrt{2}q \|Q\|_{L^{2}(C_k)}}{q\eta+1}.
\end{array}
\right.\eqno(3.61)
$$
Note that $\eta >0$, then there exists $\varepsilon>0$ small enough such that the following estimate holds:
$$\sqrt{2}q\|Q\|_{L^2(C_k )}-\frac{\eta}{2}
\left[\int_{C_k}(|\tilde E_{1k}|^2+|\tilde E_{2k}|^2)^2dx\right]^{\frac{1}{2}}\geqslant\varepsilon>0.
\eqno(3.62)$$
Combining (3.47), (3.52), (3.53) and (3.62) together yields

$$
\left.
\begin{array}{ll}
&\displaystyle\int_{C_k}-\tilde n_{k}\left(\left|\tilde E_{1k}\right|^2+\left|\tilde E_{2k}\right|^2\right)dx
\\[0.4cm]
&\displaystyle\qquad\geqslant\sqrt{2}q\|Q\|_{L^2\left(C_k\right)}
\left(\int_{C_k}\left(\left|\tilde E_{1k}\right|^2+\left|\tilde E_{2k}\right|^2\right)^2dx\right)^{\frac{1}{2}}
\\[0.4cm]
&\displaystyle\qquad\quad+\frac{q}{2}\left\|\tilde n_{k}\right\|_{L^2\left(C_k\right)}^2
-\frac{\eta}{2}\left(\int_{C_k}\left(\left|\tilde E_{1k}\right|^2+\left|\tilde E_{2k}\right|^2\right)^2dx\right)
\\[0.4cm]
&\displaystyle\qquad=\left[\sqrt{2}q\|Q\|_{L^2\left(C_k\right)}
-\frac{\eta}{2}\left(\int_{C_k}\left(\left|\tilde E_{1k}\right|^2+\left|\tilde E_{2k}\right|^2\right)^2dx\right)
^{\frac{1}{2}}\right]
\\[0.4cm]
&\displaystyle\qquad\qquad\times\left(\int_{C_k}\left(\left|\tilde E_{1k}\right|^2+\left|\tilde E_{2k}\right|^2\right)^2dx\right)
^{\frac{1}{2}}+\frac{q}{2}\left\|\tilde n_{k}\right\|_{L^2\left(C_k\right)}^2
\\[0.4cm]
&\displaystyle\qquad\geqslant \varepsilon\left(\int_{C_k}\left(\left |\tilde E_{1k}\right|^2+\left|\tilde E_{2k}\right|^2\right)^2dx\right)
^{\frac{1}{2}}+\frac{q}{2}\left\|\tilde n_{k}\right\|_{L^2\left(C_k\right)}^2
\\[0.4cm]
&\displaystyle\qquad\geqslant \varepsilon\dfrac{2\sqrt{2}q^{2} \|Q\|_{L^{2}(C_k)}}{q\eta+1}= \varepsilon c^*_1=c^{*}_3>0.
\end{array}
\right.\eqno(3.63)
$$
This is just the estimate (3.51). Hence Conclusion I follows from (3.54),(3.55) and (3.63).\hfill$\sharp$ \\
\\
\indent We next finish the proof of Lemma 3.7 according to Conclusion I by contradiction.\\
 \indent Assume by contradiction that there exists a subsequence (still denoted by $\tilde n_{k}$) such that as $k\to+\infty$,
$$\int_{C_k}|\tilde n_{k}|dx\to 0.\eqno(3.64)$$
Let $$\tilde n_{k}(x_k+\cdot)\rightharpoonup N'~~\mbox{in}~~L^2(\mathbb{R}^2),\eqno(3.65)$$
and
$$\left(\tilde E_{1k}(x_k+\cdot),\tilde E_{2k}(x_k+\cdot)\right)\rightharpoonup\left(E'_{1},E'_{2}\right)
~~\mbox{in}~~H^1(\mathbb{R}^2)\times H^1(\mathbb{R}^2).\eqno(3.66)$$
By Sobolev-type estimates, we have
$$\left(\tilde E_{1k}(x_k+\cdot),\tilde E_{2k}(x_k+\cdot)\right)\rightarrow \left(E'_{1},E'_{2}\right)
~~\mbox{in}~~L_{loc}^4(\mathbb{R}^2)\times L_{loc}^4(\mathbb{R}^2),\eqno(3.67)$$
and
$$\left(|\tilde E_{1k}(x_k+\cdot)|^{2},|\tilde E_{2k}(x_k+\cdot)|^{2}\right)\rightarrow \left(|E'_{1}|^{2},|E'_{2}|^{2}\right)
~~\mbox{in}~~L_{loc}^2(\mathbb{R}^2)\times L_{loc}^2(\mathbb{R}^2).\eqno(3.68)$$
\indent Indeed, for a bounded open domain $C_{k}$ in  $\mathbb{R}^2$, $H^1(C_{k})\subset\subset L^p(C_{k}),~~p\in[2,+\infty)$, there holds:
$$
\left.
\begin{array}{ll}
&\displaystyle\left\|~\left|\tilde E_{1k}\right|^2-\left|E'_{1}\right|^2\right\|_{L^2(C_{k})}
\\[0.4cm]
&\displaystyle\qquad\leqslant c\left(\int_{\mathbb{R}^2}\left|\tilde E_{1k}-E'_{1} \right|^4dx
\right)^{\frac{1}{4}}\cdot\left(\int_{\mathbb{R}^2}\left(\left|\tilde E_{1k}\right|^2+\left|E'_{1}\right|^2\right)^2
dx\right)^{\frac{1}{4}}
\\[0.4cm]
&\displaystyle\qquad\leqslant 2c\left(\int_{\mathbb{R}^2}\left|\tilde E_{1k}-E'_{1}\right|^4dx\right)
^{\frac{1}{4}}\cdot\left(\int_{\mathbb{R}^2}\left(|\tilde E_{1k}|^4+|E'_{1}|^4\right)dx\right)
^{\frac{1}{4}}.
\end{array}
\right.\eqno(3.69)
$$
From (3.64) it follows that as $k\rightarrow\infty$,
$$\tilde n_{k}(x_k+\cdot)\rightharpoonup 0~~\mbox{in}~~L^2(C_k).\eqno(3.70)$$
On the other hand, by Lemma 2.1 there holds as $k\to+\infty$,
$$
\left.
\begin{array}{ll}
&\displaystyle\int_{C_k}\tilde n_{k}\left(|\tilde E_{1k}|^2+|\tilde E_{2k}|^2\right)dx
\\[0.3cm]
 &\displaystyle\qquad=\int_{C_0}\tilde n_{k}\left(x_k+x\right)\left(|\tilde E_{1k}\left(x_k+x\right)|^2
 +|\tilde E_{2k}\left(x_k+x\right)|^2\right)dx
 \\[0.4cm]
 &\qquad\to   0,
 \end{array}
\right.\eqno(3.71)
$$
which is contradictory to (3.63). Hence the Lemma 3.7.

\end{proof}
\qquad\\
{\bf Step 2: An alternative form for (3.37) }\\
\\
\indent Implementing similar arguments to those in the previous section, we follow from (3.63) that there exists a constant $c^*_4>0$ such that
$$
\int_{C_0}\left(\left|\tilde E_{1k}(x_k+x)\right|^2+\left|\tilde E_{2k}(x_k+x)\right|^2\right)dx\geqslant c^*_4>0.
 \eqno(3.72)$$
 According to the definitions of $\mathcal{H}_{1}(E_{1},E_{2},n)$ (see (3.39)) and $\mathcal{H}(E_{1},E_{2},n,\mathbf{v})$(see (1.3)), for the sake of proving (3.37), it is sufficient to show that there exist constants $C_1>0$ and $C_2>0$ such that

$$-C_1+C_2\int_{\mathbb{R}^2}\left(|\nabla E_{1k}|^2+|\nabla E_{2k}|^2+\frac{1}{2}|n_k|^2\right)dx\leqslant\mathcal{H}_1
\left(E_{1k},E_{2k},n_k\right).\eqno(3.73)$$
We will verify (3.73) by contradiction. \\

Assume that (3.73) would not hold for a subsequence $(E_{1k},E_{2k},n_{k})$. That is, for all constants $C_{1}$, $C_{2}$, there holds
$$-C_1+C_2\int_{\mathbb{R}^2}\left(|\nabla E_{1k}|^2+|\nabla E_{2k}|^2+\frac{1}{2}|n_k|^2\right)dx>\mathcal{H}_1
\left(E_{1k},E_{2k},n_k\right).\eqno(3.74)$$
 Then the following conclusions would  be true provided $k\to+\infty$:
$$\lambda_k^2:=\int_{\mathbb{R}^2}\left(|\nabla E_{1k}|^2+|\nabla E_{2k}|^2\right)dx+\frac{1}{2}\int_{\mathbb{R}^2}|n_k|^2dx\to+\infty,\eqno(3.75)$$

$$\limsup_{k\to+\infty}\frac{\mathcal{H}_1(E_{1k},E_{2k},n_k)}{\lambda_k^2}\leqslant 0.\eqno(3.76)$$
\\[0.5cm]
Otherwise, \\
\\
\indent{\bf  (1)\quad If $\lambda_k\leqslant C$}, then
$$
\left.
\begin{array}{ll}
&\displaystyle\left|\mathcal{H}_1(E_{1k},E_{2k},n_k)\right|
\\[0.3cm]
 &\displaystyle\quad=  \left\|\nabla E_{1k}\right\|_{L^{2}(\mathbb{R}^2)}^2 + \left\|\nabla E_{2k} \right\|_{L^{2}(\mathbb{R}^2)}^2 + \int_{\mathbb{R}^2} |n_k|\left(|E_{1k}|^2 + |E_{2k}|^2\right)dx
 \\[0.4cm]
  &\displaystyle\qquad+ \frac{\eta}{2} \int_{\mathbb{R}^2} \left| \overline{{E}}_{1k} E_{2k} - E_{1k}  \overline{{E}}_{2k} \right|^2dx + \frac{1}{2} \|n_k\|_{L^{2}(\mathbb{R}^2)}^2
 \\[0.4cm]
 &\displaystyle\quad\leq\lambda_k^2 + \frac{1}{2}\|n_k\|_{L^{2}(\mathbb{R}^2)}^2 + \frac{1}{2} \int_{\mathbb{R}^2}\left( |E_{1k}|^2 + |E_{2k}|^2 \right)^{2}dx
  \\[0.4cm]
 &\displaystyle\qquad\qquad\qquad+\eta  \int_{\mathbb{R}^2}\left( |E_{1k}|^2 + |E_{2k}|^2\right)^{2}dx
 \\[0.4cm]
  &\displaystyle\quad \leqslant \lambda_k^2 + \frac{1}{2}\|n_k\|_{L^{2}(\mathbb{R}^2)}^2+ (1+2\eta) \frac{\left(\|E_{1k}|_{L^{2}(\mathbb{R}^2)}^2 + \|E_{2k}\|_{L^{2}(\mathbb{R}^2)}^2\right)}{\|Q\|_{L^{2}(\mathbb{R}^2)}^2}
   \\[0.5cm]
  &\displaystyle\qquad \qquad \qquad \qquad
    \cdot\left(\|\nabla E_{1k}\|_{L^{2}(\mathbb{R}^2)}^2 + \| \nabla E_{2k}\|_{L^{2}(\mathbb{R}^2)}^2\right)
 \\[0.4cm]
  &\displaystyle\quad\leqslant \lambda_k^2 + \frac{1}{2}\|n_k\|_{L^{2}(\mathbb{R}^2)}^2 + (1+2\eta)\frac{1}{\eta} \left(\|\nabla E_{1k}\|_{L^{2}(\mathbb{R}^2)}^2 + \|\nabla E_{2k} \|_{L^{2}(\mathbb{R}^2)}^2\right)
  \\[0.4cm]
  &\displaystyle\quad\leqslant \left(3 + \frac{1}{ \eta }\right) \lambda_k^2 \leqslant C.
 \end{array}
\right.\eqno(3.76*)
$$
This implies (3.73), which is contradictory to the assumption (3.74). So (3.75) holds true.\\
\\
\indent {\bf (2)\quad If $\displaystyle\lim\limits_{k\to+\infty}
\frac{\mathcal{H}_1(E_{1k},E_{2k},n_k)}{\lambda_k^2}=C>0$}, then for $k_0>0$ large enough, there holds
$$
\left.
\begin{array}{ll}
\mathcal{H}_1(E_{1k},E_{2k},n_k) &\displaystyle\geqslant \frac{C}{2} \lambda_k^2
\\[0.3cm]
 &\displaystyle= \frac{C}{2} \bigg( \|\nabla E_{1k}\|_{L^{2}(\mathbb{R}^2)}^2 + \| \nabla E_{2k}\|_{L^{2}(\mathbb{R}^2)}^2 + \frac{1}2 \| \nabla n_{k}\|_{L^{2}(\mathbb{R}^2)}^2\bigg),
\end{array}
\right.
\eqno(3.77)
$$
\\[0.2cm]
which is a contradiction since (3.73) will be satisfied with $C_{1}=0$ and $C_{2}=\frac{C}{2}$.
\\
\\
{\bf Step 3: Scaling discussion}\\
\\
\indent The proof continues as follow. Let
$$
\left\{
\begin{array}{ll}
&\tilde E_{1k}(x)=\frac{1}{\lambda_k}E_{1k}\left(\frac{x}{\lambda_k}\right),
\\[0.3cm]
&\tilde E_{2k}(x)=\frac{1}{\lambda_k}E_{2k}\left(\frac{x}{\lambda_k}\right),
\\[0.3cm]
&\tilde n_{k}(x)=\frac{1}{\lambda_k^2}n_{k}\left(\frac{x}{\lambda_k}\right).
\end{array}
\right.\eqno(3.78)
$$
Straightforward calculation gives
$$\int_{\mathbb{R}^2}\left(\left|\tilde E_{1k}(x)\right|^2+\left|\tilde E_{2k}(x)\right|^2\right)dx=\int_{\mathbb{R}^2}\left(|  E_{10}(x)|^2+| E_{20}(x)|^2\right)dx,\eqno(3.79)$$

$$\int_{\mathbb{R}^2}\left(\left|\nabla\tilde E_{1k}(x)\right|^2+\left|\nabla\tilde E_{2k}(x)\right|^2\right )dx+\frac{1}{2}\int_{\mathbb{R}^2}|\tilde n_k(x)|^2dx=1.\eqno(3.80)$$
In view of
$$
\left.
\begin{array}{ll}
&\displaystyle\limsup_{k \to \infty} \left(1 + \int_{\mathbb{R}^2} \tilde n_k\left(\left|\tilde E_{1k}\right|^2 + \left|\tilde E_{2k}\right|^2\right)dx\right.
\\[0.4cm]
&\displaystyle\qquad\qquad\qquad\left. - \frac{\eta}{2} \int_{\mathbb{R}^2} \left |\overline{\tilde E_{1k}} \tilde E_{2k} - \tilde E_{1k} \overline{\tilde E_{2k}}\right|^2  dx\right)
\\[0.4cm]
&\displaystyle\qquad= \limsup_{k \to + \infty} \mathcal{H}_1\left(\tilde E_{1k},\tilde E_{2k},\tilde n_k\right) \leqslant 0,
\end{array}
\right.\eqno(3.81)
$$
by H\"{o}lder inequality one has

$$
\left.
\begin{array}{ll}
&\displaystyle\left|\int_{\mathbb{R}^2}  \tilde n_k\left(|\tilde E_{1k}|^2 + |\tilde E_{2k}|^2\right)dx - \frac{\eta}{2} \int_{\mathbb{R}^2} \left|\overline{\tilde E_{1k}} \tilde E_{2k} - \tilde E_{1k}\overline{\tilde E_{2k}}\right|^2 dx\right|
\\[0.4cm]
 &\displaystyle\quad\leqslant \frac{1}{2} \|\tilde n_k\|_{L^{2}(\mathbb{R}^2)}^2 + \frac{1}{2}
\int_{\mathbb{R}^2}\left(|\tilde E_{1k}|^2 + |\tilde E_{2k}|^2\right)^2dx
\\[0.5cm]
&\displaystyle\qquad\qquad\qquad\qquad+ \frac{\eta}{2} \int_{\mathbb{R}^2} \left(|\tilde E_{1k}|^2 + |\tilde E_{2k}|^2\right)^2dx
\\[0.4cm]
 &\displaystyle\quad\leqslant\frac{1}{2} \|\tilde n_k\|_{L^{2}(\mathbb{R}^2)}^2 + (1 + \eta)
 \frac{ \left(\|\tilde E_{1k}\|_{L^{2}(\mathbb{R}^2)}^2 + \|\tilde E_{2k}\|_{L^{2}(\mathbb{R}^2)}^2\right)}{\|Q\|_{L^{2}(\mathbb{R}^2)}^2}
 \\[0.5cm]
&\displaystyle\qquad\qquad\qquad\qquad\qquad\cdot\left( \|\nabla\tilde E_{1k}\|_{L^{2}(\mathbb{R}^2)}^2 + \|\nabla \tilde E_{2k}\|_{L^{2}(\mathbb{R}^2)}^2\right)
\\[0.4cm]
 &\displaystyle\quad\leqslant C.
\end{array}
\right.\eqno(3.82)
$$
Hence we can assume by (3.81) and (3.82) that as $k \to +\infty$,

$$
\left.
\begin{array}{ll}
&\displaystyle\int_{\mathbb{R}^2} \tilde n_k\left(|\tilde E_{1k}|^2 + |\tilde E_{2k}|^2\right)dx - \frac{\eta}{2} \int_{\mathbb{R}^2} \left|\overline{\tilde E_{1k}}\tilde E_{2k} - \tilde E_{1k} \overline{\tilde E_{2k}}\right|^2dx
\\[0.4cm]
 &\displaystyle\qquad\qquad\qquad\to c \leqslant -1.
\end{array}
\right.\eqno(3.83)
$$
On the other hand, recalling (3.35) and (3.75), we have $\forall R>0$,
$$\liminf_{k\to+\infty}\left(\sup_{y}\int_{|x-y|<R}\left(|\tilde E_{1k}(x)|^2+|\tilde E_{2k}(x)|^2\right)dx\right)\leqslant\frac{\|Q\|
_{L^2(\mathbb{R}^2)}^2}{1+\eta}-\delta'_0,
\eqno(3.84)$$
or as $R\rightarrow 0$,
$$\liminf_{k\to+\infty}\left(\sup_{y}\int_{|x-y|<R} |\tilde n_{k}| dx\right)\rightarrow0.\eqno(3.85)$$
Note that Lemma 3.7, (3.85) does not hold. Therefore we need to concern the case (3.84) only. \\
\\
{\bf Step 4: Proof of Proposition 3.6}\\
\\
\indent Recalling the definitions of $\mathcal{E}$ (see 3.38) and $\mathcal{H}_{1}$ (see 3.39), there hold
$$\mathcal{H}_1\left(\tilde E_{1k},\tilde E_{2k},\tilde n_{k}\right) = \mathcal{E}\left(\tilde E_{1k},\tilde E_{2k}\right) + \frac{1}{2} \int_{\mathbb{R}^2} \left(\tilde n_{k}+|\tilde E_{1k}|^2+ |\tilde E_{2k}|^2 \right)^2dx,\eqno(3.86)$$
and
$$\limsup_{k \to + \infty} \mathcal{E}\left(\tilde E_{1k},\tilde E_{2k}\right) \leqslant \limsup_{k \to + \infty} \mathcal{H}_1\left(\tilde E_{1k},\tilde E_{2k},\tilde n_{k}\right) \leqslant 0.\eqno(3.87)$$
Then by (3.79), (3.80) and Sobolev type estimates, we conclude that there exist $c_1>0$ and $c_2>0$ such that
$$c_1\leqslant\int_{\mathbb{R}^2}\left(\left|\tilde E_{1k}(x)\right|^2+\left|\tilde E_{2k}(x)\right|^2\right)^2dx\leqslant c_2, \eqno(3.88)$$
\\
$$c_1\leqslant\int_{\mathbb{R}^2}\left(\left|\nabla\tilde E_{1k}(x)\right|^2+\left|\nabla\tilde E_{2k}(x)\right|^2+\left|\tilde E_{1k}(x)\right|^2+\left|\tilde E_{2k}(x)\right|^2\right)dx\leqslant c_2. \eqno(3.89)$$
Hence there exist a constant $\delta_1>0$ and a sequence $x_k^1\in\mathbb{R}^2$ such that
$$\int_{|x-x_k^1|<1}\left(\left|\tilde E_{1k}(x)\right|^2+\left|\tilde E_{2k}(x)\right|^2\right)dx\geqslant\delta_1. \eqno(3.90)$$
In view of Lemma 3.7 and its proof, we introduce the following dichotomy
$$
\left\{
\begin{array}{ll}
&\displaystyle\tilde E_{1k}(x)=\tilde E_{1k}^1(x)+\tilde E_{1k}^{1,R}(x),
\\[0.3cm]
 &\displaystyle\tilde E_{2k}(x)=\tilde E_{2k}^1(x)+\tilde E_{2k}^{1,R}(x).
\end{array}
\right.\eqno(3.91)
$$
Hence, for a sequence $x_k^1$,
$$\left(\tilde E_{1k}^1(x+x_k^1),\tilde E_{2k}^1(x+x_k^1)\right)\rightharpoonup(\psi_1,\psi_2) ~~\mbox{in} ~~H^1(\mathbb{R}^2)\times H^1(\mathbb{R}^2),\eqno(3.92)$$
and
$$\left(\int_{|x-x_k^1|<1}\left(|\tilde E_{1k}^1(x+x_k^1)|^2+|\tilde E_{2k}^1(x+x_k^1)|^2\right)^{2}dx\right)^{\frac{1}{4}}\geqslant c>0.\eqno(3.93)$$
By Sobolev estimates, there exists a $\delta_{1}>0$ depending only on
$\|E_{10}\|_{L^2(\mathbb{R}^2)}$ and $\|E_{20}\|_{L^2(\mathbb{R}^2)}$ such that
$$\left\|\tilde E_{1k}^1(x_k^1+\cdot)\right\|_{L^2(|x-x_k^1|<1)}^2+\left\|\tilde E_{2k}^{1}(x_k^1+\cdot)\right\|_{L^2(|x-x_k^1|<1)}^2\geq\delta_{1}>0.\eqno(3.94)$$
Recalling (3.84), we also obtain for $\forall~R>0$,
$$ \liminf_{k\to+\infty}\left(\|\tilde E_{1k}^1 (x_k^1+\cdot) \|_{L^2(B _{R})}^2+\|\tilde E_{2k}^1(x_k^1+\cdot)\|_{L^2(B_{R})}^2\right)\leqslant
\frac{\|Q\|_{L^2(\mathbb{R}^2)}^2}{1+\eta}-\delta'_0.
\eqno(3.95)$$
Furthermore, using concentration compactness method (Lions \cite{32Lions}), one gets for a suitable choice for $\left(\tilde E_{1k},\tilde E_{2k}\right)$,
$$
\left.
\begin{array}{ll}
&\displaystyle\left\|\tilde E_{1k}^1\right\|_{L^2(\mathbb{R}^2)}^2+\left\|\tilde E_{1k}^{1,R}\right\|_{L^2(\mathbb{R}^2)}^2+\left\|\tilde E_{2k}^1\right\|_{L^2(\mathbb{R}^2)}^2+\left\|\tilde E_{2k}^{1,R}\right\|_{L^2(\mathbb{R}^2)}^2
\\[0.3cm]
 &\displaystyle\qquad\qquad\qquad\to
 \|E_{10}\|_{L^2(\mathbb{R}^2)}^2+\|E_{20}\|_{L^2(\mathbb{R}^2)}^2,
\end{array}
\right.\eqno(3.96)
$$

$$\delta_1\leqslant\lim_{k\to+\infty}\left(\|\tilde E_{1k}^1(x)\|_{L^2(\mathbb{R}^2)}^2+\|\tilde E_{2k}^1(x)\|_{L^2(\mathbb{R}^2)}^2\right)
\leqslant\frac{\|Q\|_{L^2(\mathbb{R}^2)}^2}{1+\eta}-\delta'_0,
\eqno(3.97)
$$
and
$$
\left.
\begin{array}{ll}
\mathcal{E}(\psi_{1},\psi_{2})&\displaystyle\leqslant\limsup_{k\to+\infty}\mathcal{E}\left(\tilde E_{1k}^1,\tilde E_{2k}^1\right)+\limsup_{k\to+\infty}\mathcal{E}\left(\tilde E_{1k}^{1,R},\tilde E_{2k}^{1,R}\right)
\\[0.3cm]
 &\displaystyle\leqslant\limsup_{k\to+\infty}\mathcal{E}\left(\tilde E_{1k},\tilde E_{2k}\right)\leqslant 0.
\end{array}
\right.\eqno(3.98)
$$
Since

$$\delta_1\leqslant\|\psi_1\|_{L^2(\mathbb{R}^2)}^2+\|\psi_2\|
_{L^2(\mathbb{R}^2)}^2\leqslant\frac{\|Q\|_{L^2(\mathbb{R}^2)}^2}{1+\eta}-\delta'_0,
\eqno(3.99)$$
we have
$$\limsup\limits_{k\rightarrow +\infty}\mathcal{E}\left( \tilde E_{1k}^{1,R},\tilde E_{2k}^{1,R}\right)\leq -\mathcal{E}\left(\psi_1,\psi_2\right)<0.\eqno(3.100)$$
We now can extract a subsequence still denoted by $\left( \tilde E_{1k}^{1,R},\tilde E_{2k}^{1,R}\right)$ such that
$$\left\|\tilde E_{1k}^{1,R}(x)\right\|_{L^2(\mathbb{R}^2)}^2+\left\|\tilde E_{2k}^{1,R}(x)\right\|_{L^2(\mathbb{R}^2)}^2\rightarrow c_{1}<\frac{\|Q\|_{L^2(\mathbb{R}^2)}^2}{1+\eta}-\delta'_1,\eqno(3.101)$$
$$\limsup_{k\to+\infty}\mathcal{E}\left(\tilde E_{1k}^{1,R},\tilde E_{2k}^{1,R}\right)\leqslant  -\mathcal{E}(\psi_1,\psi_2)<0.\eqno(3.102)$$
Then there exists a constant $k_0>0$ such that $\forall k\geqslant k_0$,
$$\mathcal{E}\left(\tilde E_{1k}^{1,R},\tilde E_{2k}^{1,R}\right)\leqslant\frac{-\mathcal{E}(\psi_1,\psi_2)}{2}<0.\eqno(3.103)$$
Note that
$$\left\|\tilde E_{1k}^{1,R}(x)\right\|_{L^2(\mathbb{R}^2)}^2+\left\|\tilde E_{2k}^{1,R}(x)\right\|_{L^2(\mathbb{R}^2)}^2\leqslant\frac{\|Q\|
_{L^2(\mathbb{R}^2)}^2}{1+\eta},\eqno(3.104)$$
then
$$
\left.
\begin{array}{ll}
\mathcal{E}\left(\tilde E_{1k}^{1,R},\tilde E_{2k}^{1,R}\right) &\displaystyle=\int_{\mathbb{R}^2}\left(|\nabla \tilde E_{1k}^{1,R}|^2+|\nabla \tilde E_{2k}^{1,R}|^2\right)dx
 \\[0.4cm]
&\quad \displaystyle-\frac{1}{2}\int_{\mathbb{R}^2}\left(|\tilde E_{1k}^{1,R}|^2+|\tilde E_{2k}^{1,R}|^2\right)^2dx
  \\[0.4cm]
 &\quad \displaystyle-\frac{\eta}{2}\int_{\mathbb{R}^2}\left|\overline{\tilde E_{1k}^{1,R}}
\tilde E_{2k}^{1,R}-\tilde E_{1k}^{1,R}\overline{\tilde E_{2k}^{1,R}}\right|^2dx
 \\[0.4cm]
&\displaystyle\geqslant \int_{\mathbb{R}^2}\left(|\nabla\tilde E_{1k}^{1,R}|^2+|\nabla\tilde E_{1k}^{1,R}|^2\right)dx
 \\[0.4cm]
 &\qquad \displaystyle-\frac{(1+\eta)}{\|Q\|_{L^2(\mathbb{R}^2)}^2}\left(\|\tilde E_{1k}^{1,R}\|_{L^2(\mathbb{R}^2)}^2+\|\tilde E_{2k}^{1,R}\|_{L^2(\mathbb{R}^2)}^2\right)
 \\[0.4cm]
 &\qquad\qquad\qquad  \displaystyle\cdot\int_{\mathbb{R}^2}\left(|\nabla\tilde E_{1k}^{1,R}|^2+|\nabla\tilde E_{2k}^{1,R}|^2\right)dx
 \\[0.4cm]
 & \geqslant  0.
\end{array}
\right.\eqno(3.105)
$$
This is contradictory to (3.103). Hence (3.104) is not true. Therefore, we claim\\
\indent (2)~~If
$$\left\|\tilde E_{1k}^{1,R}(x)\right\|_{L^2(\mathbb{R}^2)}^2+\left\|\tilde E_{2k}^{1,R}(x)\right\|_{L^2(\mathbb{R}^2)}^2>\frac{\|Q\|
_{L^2(\mathbb{R}^2)}^2}{1+\eta}.
\eqno(3.106)$$
In view of (3.103), there exists a constant $c>0$ such that
$$\int_{\mathbb{R}^2}\left(|\tilde E_{1k}^{1,R}|^2+|\tilde E_{2k}^{1,R}|^2\right)^2dx>c.
\eqno(3.107)$$
{\bf Indeed, by $\mathcal{E}\left(\tilde E_{1k}^{1,R}, \tilde E_{2k}^{1,R}\right)<0$,
we have
$$
\left.
\begin{array}{ll}
&\displaystyle \int_{\mathbb{R}^2}\left(|\tilde E_{1k}^{1,R}|^2+|\tilde E_{2k}^{1,R}|^2\right)^2dx
 \\[0.4cm]
 &\qquad \displaystyle >2\int_{\mathbb{R}^2}\left(|\nabla\tilde E_{1k}^{1,R}|^2+|\nabla\tilde E_{2k}^{1,R}|^2\right)dx-\eta\int_{\mathbb{R}^2}\left|\overline{\tilde E_{1k}^{1,R}}\tilde E_{2k}^{1,R}-\tilde E_{1k}^{1,R}\overline{\tilde E_{2k}^{1,R}}\right|^2dx
 \\[0.4cm]
 &\qquad \displaystyle>2\int_{\mathbb{R}^2}\left(|\nabla\tilde E_{1k}^{1,R}|^2+|\nabla\tilde E_{2k}^{1,R}|^2\right)dx -\eta\int_{\mathbb{R}^2}\left(|\tilde E_{1k}^{1,R}|^2+|\tilde E_{2k}^{1,R}|^2\right)^2dx.
\end{array}
\right.
$$
This yields the estimate (3.107).\\
We then iterate the same procedure as above and define
$$
\left\{
\begin{array}{ll}
&\displaystyle\tilde E_{1k}^{1,R}=\tilde E_{1k}^{2}+\tilde E_{1k}^{2,R},
  \\[0.3cm]
 &\displaystyle\tilde E_{2k}^{1,R}=\tilde E_{2k}^{2}+\tilde E_{2k}^{2,R},
\end{array}
\right.\eqno(3.108)
$$
where $\tilde E_{1k}^{2}$ and $\tilde E_{2k}^{2}$ satisfy for a sequence $x^{2}_{k}$,
$$\left\|\tilde E_{1k}^{2}(x^{2}_{k}+\cdot)\right\|_{L^2(|x-x^{2}_{k}|<1)}^2+\left\|\tilde E_{2k}^{2}(x^{2}_{k}+\cdot)\right\|_{L^2(|x-x^{2}_{k}|<1)}^2\geq \delta_{1}.\eqno(3.109)$$
Defining $p$ such that
$$-p \delta_{1}+\left\| E_{10} \right\|_{L^2(\mathbb{R}^2)}^2+\left\| E_{20} \right\|_{L^2(\mathbb{R}^2)}^2
<\frac{\left\|Q \right\|_{L^2(\mathbb{R}^2)}^2}{1+\eta},\eqno(3.110)$$
applying the same procedure at most $p$ times, we find for an $i\leq p$ and $k$ large, there exists a function $\left(\tilde E_{1k}^{i,R},\tilde E_{2k}^{i,R}\right)$ such that
$$\left\|\tilde E_{1k}^{i,R}\right\|_{L^2(\mathbb{R}^2)}^2+\left\|\tilde E_{2k}^{i,R}\right\|_{L^2(\mathbb{R}^2)}^2\leqslant
\frac{\|Q\|_{L^2(\mathbb{R}^2)}^2}{1+\eta},
\eqno(3.111)$$
and
$$\mathcal{E}\left(\tilde E_{1k}^{i,R},\tilde E_{2k}^{i,R}\right)\leqslant\frac{-\mathcal{E} (\psi_1,\psi_2)}{2}<0.
\eqno(3.112)$$
Then by Lemma 2.2, (3.111) and (3.112) are contradictory. }\\
\indent In addition, from (3.4) one gets
$$
\left.
\begin{array}{ll}
&\displaystyle\left\|\nabla\tilde E_{1k}\right\|_{L^2(\mathbb{R}^2)}^2+\left\|\nabla\tilde E_{2k}\right\|_{L^2(\mathbb{R}^2)}^2
  \\[0.4cm]
 &\displaystyle\qquad\geqslant\frac{\int_{\mathbb{R}^2}\left(|\tilde E_{1k}|^2+|\tilde E_{2k}|^2\right)^2dx\cdot\|Q\|_{L^2(\mathbb{R}^2)}^2}{2\left(\|\tilde E_{1k}\|_{L^2(\mathbb{R}^2)}^2+\|\tilde E_{2k}\|_{L^2(\mathbb{R}^2)}^2\right)}
 \\[0.4cm]
 &\displaystyle\qquad>\frac{\eta}{2}\int_{\mathbb{R}^2}\left(|\tilde E_{1k}|^2+|\tilde E_{2k}|^2\right)^2dx
  \\[0.4cm]
  &\displaystyle\qquad\geqslant\frac{\eta}{2}\int_{\mathbb{R}^2}\left|\overline{\tilde E_{1k}}\tilde E_{2k}-\tilde E_{1k}\overline{\tilde E_{2k}}\right|^2dx.
  \end{array}
\right.\eqno(3.113)
$$
Since
$$\limsup_{k\to+\infty}\mathcal{H}_{1}\left(\tilde E_{1k},\tilde E_{1k},\tilde E_{2k},\tilde n_{k}\right)=\limsup_{k\to+\infty}\frac{\mathcal{H}_{1}\left(\tilde E_{1k},\tilde E_{1k},\tilde E_{2k},\tilde n_{k}\right)}{\lambda_{k}^{2}}\leq 0,\eqno(3.114)$$
then
$$\int_{\mathbb{R}^2}\tilde n_k\left(|\tilde E_{1k}|^2+|\tilde E_{2k}|^2\right)dx\to-C<0.\eqno(3.115)$$
It follows from Lemma 3.7 that there exist a constant $C'>0$ and a sequence $x_k$ such that
$$\int_{|x-x_k|<1}|\tilde n_k|dx>C'>0.\eqno(3.116)$$
Therefore, recalling (3.36) and the definition of $\tilde n_k$, we have as $R\to 0$,
$$\liminf_{k\to+\infty}\left(\sup_y\int_{|x-y|<R}|\tilde n_k|dx\right)\to 0.\eqno(3.117)$$
(3.116) and (3.117) are contradictory.\\
\indent This finishes the proof of Proposition 3.6.\hfill$\Box$\\
\\
\indent We now claim the following conclusions to prove the non-vanishing properties of $\left(\tilde E_{1}(0),\tilde E_{2}(0),\tilde n(0)\right)$.\\
\begin{proposition}\label{3.8}
Let $(E_1,E_2,n,\textbf{v})$ be the finite time blow-up solution of the Zakharov system (1.1) and $T$ be its blowup time. That is, as $t\rightarrow T$,
$$\|E_1 \|_{H^1(\mathbb{R}^2)}+\|E_2 \|
_{H^1(\mathbb{R}^2)}+\|n \|_{L^2(\mathbb{R}^2)}
+\|\textbf{v} \|_{L^2(\mathbb{R}^2)}\to+\infty.\eqno(3.118)$$
Assume that the initial data satisfy (3.4), then\\
\indent (1)\quad If $E_1,E_2,n$ are radially symmetric functions, then there exists a constant $m>0$ such that  for any $  R>0$,
$$ \liminf_{t\to T}\left(\|E_1(t,x)\|_{L^2\left(B\left(0,R\right)\right)}^2
+\|E_2(t,x)\|_{L^2\left(B\left(0,R\right)\right)}^2\right)>
\frac{1}{1+\eta}\|Q\|_{L^2(\mathbb{R}^2)}^2,\eqno(3.119)$$
$$ \liminf_{t\to T} \|n(t,x)\|_{L^1\left(B\left(0,R\right)\right)} \geqslant m.\eqno(3.120)$$
\indent (2)\quad If $E_1,E_2,n$ are non-radially symmetric functions, then there exists a sequence $x(t)\in\mathbb{R}^2$ and a constant $m>0$ depending only on initial data such that  for any $  R>0$,
$$\liminf_{t\to T}\left(\|E_1(t,x)\|_{L^2\left(B\left(x\left(t\right),R\right)\right)}^2
+\|E_2(t,x)\|_{L^2\left(B\left(x\left(t\right),R\right)\right)}^2\right)$$
$$>\frac{1}{1+\eta}\|Q\|_{L^2(\mathbb{R}^2)}^2,
 \eqno(3.121)$$
 \\
$$ \liminf_{t\to T}\|n(t,x)\|_{L^1\left(B\left(x\left(t\right),R\right)\right)}\geqslant m. \eqno(3.122)$$
\end{proposition}
{\bf Proof.} We first show the case (1):  ~$(E_1,E_2,n)\in H_r^1(\mathbb{R}^2)\times H_r^1(\mathbb{R}^2)\times L_r^2(\mathbb{R}^2)$. \\
\indent Define two spaces:\\
$$H_r^1(\mathbb{R}^2)=\left\{f\in H^1(\mathbb{R}^2), f(x)=f(|x|)\right\},$$
$$L_r^2(\mathbb{R}^2)=\left\{f\in L^2(\mathbb{R}^2), f(x)=f(|x|)\right\}.$$
It follows from (1.3),(3.38) and (3.39) that
$$\mathcal{E}(E_1,E_2) \leqslant \mathcal{H}_1(E_1,E_2,n,\textbf{v}) = \mathcal{H}(E_1,E_2,n,\textbf{v}) - \frac{1}{2} \|\textbf{v}\|_{L^{2}(\mathbb{R}^2)}^2.\eqno(3.123)$$
We proceed our proof by contradiction. \\

Assume that there exist constants $\delta_0>0$, $R_0>0$  and a sequence $t_k\to T~~(k\to+\infty)$ such that
$$\int_{|x|<R_0}\left(|E_1(t_k,x)|^2+|E_2(t_k,x)|^2\right)dx
\leqslant\frac{1}{(1+\eta)}\|Q\|_{L^2(\mathbb{R}^2)}^2-\delta_0, \eqno(3.124)$$
\\
or
$$\liminf_{k\to+\infty} \int_{|x|<R_0}|n\left(t_k,x\right)|dx =0. \eqno(3.125)$$
We then complete the proof of case (1) by scaling and compactness method.\\
\indent Let
$$
\left\{
\begin{array}{ll}
&\displaystyle E_{1k}(x)=\frac{1}{\lambda_k}E_1\left(t_k,\frac{x}{\lambda_k}\right),\quad E_{2k}(x)=\frac{1}{\lambda_k}E_2\left((t_k,\frac{x}{\lambda_k}\right),
  \\[0.4cm]
 &\displaystyle n_k(x)=\frac{1}{\lambda_k^2}n\left(t_k,\frac{x}{\lambda_k}\right),
\quad {v}_k(x)=\frac{1}{\lambda_k^2}\textbf{v}\left(t_k,\frac{x}{\lambda_k}\right),
\end{array}
\right.\eqno(3.126)
$$
where
$$\lambda_k^2=\left\|\nabla E_1(t_k,x)\right\|_{L^2(\mathbb{R}^2)}^2+\left\|\nabla E_2(t_k,x)\right\|_{L^2(\mathbb{R}^2)}^2.$$
Direct calculation gives
$$
\left\{
\begin{array}{ll}
&\displaystyle \int_{\mathbb{R}^2}|\nabla E_{1k}|^2dx+\int_{\mathbb{R}^2}|\nabla E_{2k}|^2dx=1,
 \\[0.4cm]
&\displaystyle\int_{\mathbb{R}^2}|E_{1k}|^2dx+\int_{\mathbb{R}^2}|E_{2k}|^2dx
=\int_{\mathbb{R}^2}|E_{10}|^2dx+\int_{\mathbb{R}^2}|E_{20}|^2dx,
 \\[0.4cm]
&\displaystyle\mathcal{E}\left(E_{1k},E_{2k}\right)=
\frac{1}{\lambda_k^2}\mathcal{E}\left(E_1\left(t_k,x\right),
E_2\left(t_k,x\right)\right),
 \\[0.4cm]
&\displaystyle\mathcal{H}_1\left(E_{1k},E_{2k},n_k\right)=\frac{1}{\lambda_k^2}\mathcal{H}_1
\left(E_1\left(t_k,x\right),E_2\left(t_k,x\right),n\left(t_k,x\right)\right),
 \\[0.4cm]
&\displaystyle\mathcal{H}(t_k)=\mathcal{H}(0).
\end{array}
\right.\eqno(3.127)
$$
Note that
$$
\left.
\begin{array}{ll}
 \mathcal{H}(t_k)=&\displaystyle \mathcal{E}\left(E_1(t_k,x),E_2(t_k,x)\right)
+\frac{1}{2}\int_{\mathbb{R}^2}|\textbf{v}(t_k)|^2dx
 \\[0.3cm]
 &\displaystyle\quad+\frac{1}{2}
\int_{\mathbb{R}^2}\left[n(t_k)+\left(|E_1(t_k)|^2+|E_2(t_k)|^2\right)\right]^2dx.
 \end{array}
\right.\eqno(3.128)
$$
 We then conclude
$$\mathcal{E}\left(E_1(t_k),E_2(t_k)\right)\leqslant\mathcal{H}_1\left(E_1(t_k),
E_2(t_k),n(t_k)\right)\leqslant\mathcal{H}(t_k)=\mathcal{H}(0),
\eqno(3.129)$$
$$\mathcal{E}\left(E_{1k},E_{2k}\right)\leqslant\mathcal{H}_1
\left(E_{1k},E_{2k},n_k,\textbf{v}_k\right)\leqslant\frac{1}
{\lambda_k^2}\mathcal{H}(0)\stackrel{k\to+\infty}{\longrightarrow}0.
\eqno(3.130)$$
Especially,
$$\limsup_{k\to+\infty}\mathcal{E}\left(E_{1k},E_{2k}\right)\leqslant 0,~~~
\limsup_{k\to+\infty}\mathcal{H}_1\left(E_{1k},E_{2k},n_{k}\right)\leqslant 0.\eqno(3.131)$$
Using Cauchy-Schwartz inequality: $\displaystyle ab\leq \frac{(a+b)^{2}}{4}~(a>0,~b>0)$, we obtain
$$
\left.
\begin{array}{ll}
  &\displaystyle\mathcal{E}(E_1,E_2)
  \\[0.3cm]
   &\displaystyle\qquad\geqslant\left\|\nabla E_1\right\|_{L^2(\mathbb{R}^2)}^2+\left\|\nabla E_2\right\|_{L^2(\mathbb{R}^2)}^2
  \\[0.4cm]
  &\displaystyle\qquad\quad-\frac{1}{2}\int_{\mathbb{R}^2}\left(|E_1|^2+|E_2|^2\right)^2dx
-2\eta\int_{\mathbb{R}^2}|E_1|^2|E_2|^2dx
  \\[0.4cm]
   &\displaystyle\qquad\geqslant\|\nabla E_1\|_{L^2(\mathbb{R}^2)}^2+\|\nabla E_2\|_{L^2(\mathbb{R}^2)}^2
-\frac{1+\eta}{2}\int_{\mathbb{R}^2}\left(|E_1|^2+|E_2|^2\right)^2dx.
    \end{array}
\right.\eqno(3.132)
$$
This together with (3.127) and (3.131) yields
$$
\left.
\begin{array}{ll}
  &\displaystyle\liminf_{k\to+\infty}\int_{\mathbb{R}^2}\left(|E_{1k}|^2+|E_{2k}|^2\right)^2dx
  \\[0.4cm]
   &\displaystyle\qquad\geqslant\frac{2}{1+\eta}\liminf_{k\to+\infty}
   \left(\int_{\mathbb{R}^2}\left(\left|\nabla E_{1k}\right|^2+\left|\nabla E_{2k}\right|^2\right)dx-\mathcal{E}\left(E_{1k},E_{2k}\right)\right)
  \\[0.4cm]
  &\displaystyle\qquad\geqslant\frac{2}{1+\eta}.
  \end{array}
\right.\eqno(3.133)
$$
On the other hand, since
$$
\left.
\begin{array}{ll}
&\displaystyle\int_{\mathbb{R}^2}\left(n_k+\left(|E_{1k}|^2+|E_{2k}|^2\right)\right)^2dx
-(1+\eta)\int_{\mathbb{R}^2}\left(|E_{1k}|^2+|E_{2k}|^2\right)^2dx
 \\[0.4cm]
&\displaystyle\quad
 \leqslant\int_{\mathbb{R}^2}\left(n_k+\left(|E_{1k}|^2+|E_{2k}|^2\right)\right)^2dx
-\int_{\mathbb{R}^2}\left(|E_{1k}|^2+|E_{2k}|^2\right)^2dx
\\[0.4cm]
&\displaystyle\qquad-\eta\int_{\mathbb{R}^2}\left|\overline{E_{1k}}E_{2k}
-E_{1k}\overline{E_{2k}}\right|^2dx
\\[0.4cm]
 &\displaystyle\quad=2\left(\mathcal{H}_1\left(E_{1k},E_{2k},n_k\right)-\|\nabla E_{1k}\|_{L^2(\mathbb{R}^2)}^2-\|\nabla E_{2k}\|_{L^2(\mathbb{R}^2)}^2\right),
\end{array}
\right.
$$
one has
$$
\left.
\begin{array}{ll}
&\displaystyle\limsup_{k\to+\infty}\int_{\mathbb{R}^2}\left(n_k+\left(|E_{1k}|^2
+|E_{2k}|^2\right)\right)^2dx
  \\[0.3cm]
 &\displaystyle\qquad -(1+\eta)\int_{\mathbb{R}^2}\left(|E_{1k}|^2+|E_{2k}|^2\right)^2dx\leqslant-2.
\end{array}
\right.\eqno(3.134)
$$
Noting Lemma 2.2, (3.4), (3.127) and
$$
\left.
\begin{array}{ll}
\displaystyle
\eta\int_{\mathbb{R}^2}\left(|E_{1k}|^2+|E_{2k}|^2\right)^2dx&\displaystyle\leqslant
 \frac{2\eta\left(\|E_{1k}\|_{L^2(\mathbb{R}^2)}^2
+\|E_{2k}\|_{L^2(\mathbb{R}^2)}^2\right)}{\|Q\|_{L^2(\mathbb{R}^2)}^2}
 \\[0.4cm]
&\displaystyle\quad\cdot\left(\|\nabla E_{1k}\|_{L^2(\mathbb{R}^2)}^2+\|\nabla E_{2k}\|_{L^2(\mathbb{R}^2)}^2\right)
 \\[0.3cm]
&\displaystyle
\leqslant  2,
 \end{array}
\right.
$$
we get
$$
\left.
\begin{array}{ll}
 &\displaystyle\limsup_{k\to+\infty}\int_{\mathbb{R}^2}\left(n_k+\left(|E_{1k}|^2+|E_{2k}|^2\right)
\right)^2dx -\int_{\mathbb{R}^2}\left(|E_{1k}|^2+|E_{2k}|^2\right)^2dx\leqslant 0.
 \end{array}
\right.
$$
In view of (3.38), one obtains
$$
\left.
\begin{array}{ll}
&\displaystyle\frac{1}{2}\int_{\mathbb{R}^2}\left[\left(|E_1|^2+|E_2|^2\right)^2
+\frac{\eta}{2}\left|\overline{E_1}E_2-E_1\overline{E_2}\right|^2\right]dx
  \\[0.4cm]
 &\displaystyle\quad=\|\nabla E_1\|_{L^2(\mathbb{R}^2)}^2+\|\nabla E_2\|_{L^2(\mathbb{R}^2)}^2-\mathcal{E}(E_1,E_2).
\end{array}
\right.
$$
Hence there holds
$$
\left.
\begin{array}{ll}
&\displaystyle  \liminf_{k\to+\infty}\left(\int_{\mathbb{R}^2}
\left(|E_{1k}|^2+|E_{2k}|^2\right)^2dx
+\eta\int_{\mathbb{R}^2}\left|\overline{E_{1k}}E_{2k}
-E_{1k}\overline{E_{2k}}\right|^2dx\right)
  \\[0.4cm]
 &\displaystyle\qquad=2\liminf_{k\to+\infty}\left(\|\nabla E_{1k}\|_{L^2(\mathbb{R}^2)}^2+\|\nabla E_{2k}\|_{L^2(\mathbb{R}^2)}^2-\mathcal{E}\left(E_{1k},E_{2k}\right)\right)
  \\[0.3cm]
   &\displaystyle\qquad\geqslant 2.
\end{array}
\right.
$$
In addition, (3.39) yields
$$
\left.
\begin{array}{ll}
&\displaystyle\frac{1}{2}\int_{\mathbb{R}^2}n_k^2dx
 \\[0.4cm]
&\displaystyle\quad= \mathcal{H}_1\left(E_{1k},E_{2k},n_k\right)-\|\nabla E_{1k}\|_{L^2(\mathbb{R}^2)}^2-\|\nabla E_{2k}\|_{L^2(\mathbb{R}^2)}^2
  \\[0.4cm]
&\displaystyle\qquad-\int_{\mathbb{R}^2}n_k\left(|E_{1k}|^2+|E_{2k}|^2\right)dx
+\frac{\eta}{2}\int_{\mathbb{R}^2}\left|\overline{E_{1k}}E_{2k}-E_{1k}
\overline{E_{2k}}\right|^2dx.
   \end{array}
\right.
$$
Now for any $\epsilon\in(0,1)$, by Young's inequality, H\"{o}lder inequality and Cauchy -Schwartz inequality: $ab\leq\dfrac{(a+b)^{2}}{4},~~\forall~~a,b>0$, we have
$$
\left.
\begin{array}{ll}
\displaystyle\int_{\mathbb{R}^2}n_k^2dx&\displaystyle\leqslant  2\int_{\mathbb{R}^2}-n_k\left(|E_{1k}|^2+|E_{2k}|^2\right)dx
\\[0.4cm]
&\displaystyle\quad+\eta\int_{\mathbb{R}^2}\left|\overline{E_{1k}}E_{2k}-E_{1k}
\overline{E_{2k}}\right|^2dx
\\[0.4cm]
&\displaystyle\leqslant \epsilon\|n_k\|_{L^2(\mathbb{R}^2)}^2
+\frac{1}{\epsilon}\int_{\mathbb{R}^2} \left(|E_{1k}|^2+|E_{2k}|^2\right)^2dx
  \\[0.4cm]
 &\displaystyle\quad+\eta\int_{\mathbb{R}^2}\left(|E_{1k}|^2+|E_{2k}|^2\right)^2dx.
\end{array}
\right.
$$
This yields that
$$
\left.
\begin{array}{ll}
 \displaystyle\limsup_{k\to+\infty}\int_{\mathbb{R}^2}n_k^2dx&\displaystyle\leqslant
\frac{1+\eta\epsilon}{(1-\epsilon)\epsilon}\int_{\mathbb{R}^2}
\left(|E_{1k}|^2+|E_{2k}|^2\right)^2dx
  \\[0.4cm]
 &\displaystyle\leqslant\frac{2+2\eta\epsilon}{(1-\epsilon)\epsilon}
 \cdot\frac{\|E_{1k}\|_{L^2(\mathbb{R}^2)}^2+\|E_{2k}\|_{L^2(\mathbb{R}^2)}^2}
 {\|Q\|_{L^2(\mathbb{R}^2)}^2}.
 \end{array}
\right.\eqno(3.135)
$$
Due to (3.124), (3.125) and $\lambda_k\stackrel{k\to+\infty}{\longrightarrow}+\infty$, one obtains $\forall R>0$,
$$\limsup_{k\to+\infty}\int_{|x|<R}\left(|E_{1k}|^2+|E_{2k}|^2\right)dx
\leqslant\frac{1}{1+\eta}\|Q\|_{L^2}^2-\delta_0,\eqno(3.136)$$
or
$$\forall R>0,~~\limsup_{k\to+\infty}\int_{|x|<R}|n_k|dx=0.\eqno(3.137)$$
We continue to our discussion by a compactness argument. \\

By (3.127) and (3.135), there exists $\left(E'_1,E'_2,N'\right)\in H^1(\mathbb{R}^2)\times H^1(\mathbb{R}^2)\times L^2(\mathbb{R}^2)$ such that
$$\left(E_{1k},E_{2k}\right)\rightharpoonup\left(E'_1,E'_2\right)~\mbox{in}~H^1(\mathbb{R}^2)\times H^1(\mathbb{R}^2),\eqno(3.138)$$
$$n_k\rightharpoonup N'~\mbox{in}~L^2(\mathbb{R}^2).\eqno(3.139)$$
Since $H_r^1(\mathbb{R}^2)\hookrightarrow L_r^p(\mathbb{R}^2)~(2<p<+\infty)$ is compact, we obtain
$$\left(E_{1k},E_{2k}\right)\rightharpoonup\left(E'_1,E'_2\right)~\mbox{in}~L^4(\mathbb{R}^2)\times L^4(\mathbb{R}^2).\eqno(3.140)$$
On one hand,
$$\left(E^{2}_{1k},E^{2}_{2k}\right)\rightharpoonup\left(E'^{2}_1,E'^{2}_2\right)~\mbox{in}~L^2(\mathbb{R}^2)\times L^2(\mathbb{R}^2).\eqno(3.141)$$
On the other hand,
$$\int_{\mathbb{R}^2}\left(|E'_1|^2+|E'_2|^2\right)^2dx\geqslant\frac{1}{1+\eta},  \quad \left(E'_1,E'_2\right)\not\equiv(0,0).\eqno(3.142)$$
Let $R\to+\infty$, it follows from (3.136) that
$$\int_{\mathbb{R}^2}\left(|E'_1|^2+|E'_2|^2\right)dx
<\frac{1}{ 1+\eta}\|Q\|_{L^2(\mathbb{R}^2)}^2,\eqno(3.143)$$
or
$$N'=0.\eqno(3.144)$$
Then from the boundedness of weakly convergent sequence, we have
$$
\left.
\begin{array}{ll}
 \displaystyle\int_{|x|<R}\left(|E'_1|^2+|E'_2|^2\right)dx&\displaystyle\leqslant\liminf_{k\to+\infty}
\int_{|x|<R}\left(|E_{1k}|^2+|E_{2k}|^2\right)dx
  \\[0.4cm]
 &\displaystyle\leqslant\frac{1}{1+\eta}\|Q\|_{L^2(\mathbb{R}^2)}^2-\delta_0,
\end{array}
\right.\eqno(3.145)
$$
or
$$\int_{|x|<R}|N'|dx\leqslant\liminf_{k\to+\infty}\int_{|x|<R}|n_k|dx=0.\eqno(3.146)$$
In addition, there hold:
$$\lim _{k\rightarrow+\infty}\int_{\mathbb{R}^{2}}n_{k}\left(|E_{k}|^{2}+|E_{2k}|^{2}\right)dx
=\lim _{k\rightarrow+\infty}\int_{\mathbb{R}^{2}}N'\left(|E'_{1}|^{2}+|E'_{2}|^{2}\right)dx,
\eqno(3.147)$$
\\
$$\lim _{k\rightarrow+\infty}\int_{\mathbb{R}^{2}}|E_{1k}|^{2}|E_{2k}|^{2}dx=\lim _{k\rightarrow+\infty}\int_{\mathbb{R}^{2}}\left|E'_{1}|^{2}|E'_{2}\right|^{2}dx,
\eqno(3.148)$$
\\
$$\lim_{k\rightarrow+\infty}Re\int_{\mathbb{R}^{2}}\left(E_{1k}\right)^2
\left(\overline{E_{2k}}\right)^2dx
=\lim_{k\rightarrow+\infty}Re\int_{\mathbb{R}^2}\left(E'_1\right)^2\left(\overline{E'_2}\right)^2dx,
\eqno(3.149)$$
\\
$$\lim _{k\rightarrow+\infty}\int_{\mathbb{R}^{2}}\left|\overline{E_{1k}}E_{2k}-E_{1k}
\overline{E_{2k}}\right|^2dx=\lim _{k\rightarrow+\infty}\int_{\mathbb{R}^{2}}\left|\overline{E'_{1}}E'_{2}
-E'_1\overline{E'_2}\right|^2dx.\eqno(3.150)$$
\indent Indeed, from Lemma 2.1 and (3.141) it follows that (3.147) holds.\\

\indent Next, direct calculation yields
$$
\left.
\begin{array}{ll}
&\displaystyle\left|\int_{\mathbb{R}^{2}}|E_{1k}|^{2}|E_{2k}|^{2}dx
-\int_{\mathbb{R}^2}\left|E'_{1}\right|^{2}\left|E'_{2}\right|^{2}dx\right|
  \\[0.4cm]
 &\displaystyle\qquad= \int_{\mathbb{R}^2}\left(|E_{1k}|^2-|E'_1|^2\right)|E_{2k}|^2dx
+ \int_{\mathbb{R}^2}|E'_1|^2\left(|E_{2k}|^2-|E'_2|^2\right)dx
\\[0.4cm]
&\displaystyle\qquad\leqslant\left\| \left(|E_{1k}|^2-|E'_1|^2\right)\right\|_{L^2(\mathbb{R}^2)}
\left\| E_{2k}^2\right\|_{L^2(\mathbb{R}^2)}
\\[0.4cm]
&\displaystyle\qquad\quad +\left\|{E'_{1}}^2\right\|_{L^2(\mathbb{R}^2)}\left\| |E_{2k}|^2-|E'_2|^2 \right\|_{L^2(\mathbb{R}^2)}.
\end{array}
\right.
$$
Let $k\to+\infty$, one gets (3.148). \\

We next note that
$$
\left.
\begin{array}{ll}
&\displaystyle\left|Re\int_{\mathbb{R}^2}\left(E_{1k}\right)^2\left(\overline{E_{2k}}\right)^2dx
-Re\int_{\mathbb{R}^2}\left(E'_1\right)^2\left(\overline{E'_2}\right)^2dx\right|
  \\[0.4cm]
 &\displaystyle\qquad\leqslant\int_{\mathbb{R}^2}
 \left|\left(E_{1k}\right)^2\left(\overline{E_{2k}}\right)^2
-\left(E'_1\right)^2\left(\overline{E'_2}\right)^2\right|dx
  \\[0.4cm]
 &\displaystyle\qquad=\int_{\mathbb{R}^2}\left|\left[\left(E_{1k}\right)^2
-\left(E'_1\right)^2\right]\left(\overline{E_{2k}}\right)^2+
\left(E'_1\right)^2\left[\left(\overline{E_{2k}}\right)^2
-\left(\overline{E'_2}\right)^2\right]\right|dx
  \\[0.4cm]
 &\displaystyle\qquad\leqslant \left\|\left(E_{1k}\right)^2-\left(E'_1\right)^2\right\|_{L^2(\mathbb{R}^2)}
\left\| |\overline{E_{2k}}|^2\right\|_{L^2(\mathbb{R}^2)}
\\[0.4cm]
&\displaystyle\qquad\quad+\left\| ~\left|E'_1\right|^2\right\|_{L^2(\mathbb{R}^2)}\left\|\left(\overline{E_{2k}}\right)^2
-\left(\overline{E'_2}\right)^2\right\|_{L^2(\mathbb{R}^2)}.
\end{array}
\right.
$$
\\
As $k\to+\infty$, we get (3.149).  Moreover, according to (3.148) and (3.149), (3.150) holds. We now use estimats (3.130), (3.147) and (3.150) to get
$$\mathcal{H}_1\left(E'_1,E'_2,N'\right)\leqslant\liminf_{k\to+\infty}
\mathcal{H}_1\left(E_{1k},E_{2k},n_k\right)\leqslant 0,\eqno(3.151)$$
that is,
$$\mathcal{E}\left(E'_1,E'_2\right)+\frac{1}{2}
\int_{\mathbb{R}^2}\left[N'+\left(\left|E'_1\right|^2+\left|E'_2\right|^2\right)\right]^2dx\leqslant 0.\eqno(3.152)$$
According to  $\displaystyle\int_{\mathbb{R}^2}\left(\left|E'_1\right|^2+\left|E'_2\right|^2\right)dx<
\frac{\|Q\|_{L^2(\mathbb{R}^2)}^2}{1+\eta}$, (3.152)  then yields
$$
\left.
\begin{array}{ll}
&\displaystyle\int_{\mathbb{R}^2}\left(\left |\nabla E'_1\right|^2 + \left|\nabla E'_2\right|^2 \right)dx
  \\[0.4cm]
 & \displaystyle\quad\leqslant \frac{1}{2} \int_{\mathbb{R}^2} \left(\left|E'_1\right|^2 + \left|E'_2\right|^2\right)^2dx + \frac{\eta}{2} \int_{\mathbb{R}^2}\left | \overline{E' } _1 E'_2 - E'_1 \overline{E' }_2\right|^2dx
  \\[0.4cm]
  &\displaystyle\quad \leqslant \frac{\eta +1}{2} \int_{\mathbb{R}^2}
 \left( \left|E'_1\right|^2 + \left|E'_2\right|^2\right)^2dx
  \\[0.4cm]
  &\displaystyle\quad\leqslant (1 + \eta) \frac{\left\| E'_1\right\|_{L^2(\mathbb{R}^2)}^2 + \left\|E'_2\right\|_{L^2(\mathbb{R}^2)}^2}{\|Q\|_{L^2(\mathbb{R}^2)}^2} \left(\left\| \nabla E'_1\right\|_{L^2(\mathbb{R}^2)}^2 + \left\| \nabla E'_2\right\|_{L^2(\mathbb{R}^2)}^2\right)
  \\[0.4cm]
 &\displaystyle\quad <  \left\| \nabla E'_1\right\|_{L^2(\mathbb{R}^2)}^2 + \left\| \nabla E'_2\right\|_{L^2(\mathbb{R}^2)}^2.
\end{array}
\right.
$$
\\
This is a contradiction.\\
\indent On the other hand, if $N'=0$, (3.4) yields
$$
\left.
\begin{array}{ll}
&\displaystyle  \mathcal{H}_{1} \left(E'_{1},E'_{2},0\right)
  \\[0.4cm]
  &\displaystyle\qquad\geqslant\int_{\mathbb{R}^{2}}\left(\left|\nabla E'_{1}\right|^{2}+\left|\nabla E'_{2}\right|^{2}\right)dx-2\eta\int_{\mathbb{R}^{2}}\left|E'_{1}\right|^{2}\left|E'_{2}\right|^{2}dx    \\[0.4cm]
  &\displaystyle\qquad\geqslant \int_{\mathbb{R}^{2}}\left(\left|\nabla E'_{1}\right|^{2}+\left|\nabla E'_{2}\right|^{2}\right)dx-\frac{\eta}{2}\int_{\mathbb{R}^{2}}
\left(\left|E'_{1}\right|^{2}+\left|E'_{2}\right|^{2}\right)^2dx
    \\[0.4cm]
  &\displaystyle\qquad\geqslant\int_{\mathbb{R}^{2}}\left(\left|\nabla E'_{1}\right|^{2}+\left|\nabla E'_{2}\right|^{2}\right)dx
      \\[0.4cm]
  &\displaystyle\qquad\quad-\eta\frac{\left(\left\|E'_{1}\right\|_{L^2(\mathbb{R}^2)}^{2}
+\left\|E'_{2}\right\|_{L^2(\mathbb{R}^2)}^{2}\right)}{\|Q\|_{L^2(\mathbb{R}^2)}^2}
\cdot\left(\left\|\nabla E'_{1}\right\|_{L^2(\mathbb{R}^2)}^{2}+\left\|\nabla E'_{2}\right\|_{L^2(\mathbb{R}^2)}^{2}\right)
  \\[0.4cm]
  &\displaystyle\qquad\geqslant \int_{\mathbb{R}^{2}}\left(\left|\nabla E'_{1}\right|^{2}+\left|\nabla E'_{2}\right|^{2}\right)dx\left(1-\frac{\eta\left(\|E_{10}\|_{L^2(\mathbb{R}^2)}^{2}
+\|E_{20}\|_{L^2(\mathbb{R}^2)}^{2}\right)}{\|Q\|_{L^2(\mathbb{R}^2)}^{2}}\right)    \\[0.4cm]
&\displaystyle\qquad>0,
\end{array}
\right.
$$
which is contradictory to (3.151). Hence there exists a constant $m>0$ depending only on initial data such that for any $R>0$, (3.119) and (3.120) hold.\\
\indent Now we turn to consider the non-radial case (2). Assume that there exist constants $R_0>0,~~\delta_0>0$ and a sequence $t_k$ such that as $t_k\to T~(k\to+\infty)$,
$$\liminf_{k\to+\infty}\left(\sup_{y}\int_{|x-y|<R_0}
\left(|E_1(t_k,x)|^2+|E_2(t_k,x)|^2\right)dx\right)\leqslant
\frac{\|Q\|_{L^2(\mathbb{R}^2)}^2}{1+\eta}-\delta_0,$$
or
$$\liminf_{k\to+\infty}\left(\sup_{y}\int_{|x-y|<R_0}|n(t_k,x)|dx\right)\leqslant m_n-\delta_0.$$
Then it follows from Lemma 3.7 that as $t_k\to T$,
$$\int_{\mathbb{R}^2}\left(\left|\nabla E_1(t_k)\right|^2+\left|\nabla E_2(t_k)\right|^2+\left|n(t_k)\right|^2+\left|\textbf{v}(t_k)\right|^2\right)dx\leqslant C.$$
This is contradictory to the assumption that $(E_1,E_2,n,\textbf{v})$ blows up at a finite time $T$. So (3.121) and (3.122) hold.\\
\indent This finishes the proof of Proposition 3.8.\hfill$\Box$\\

\indent We are now in the position to prove Proposition 3.5 by utilizing Proposition 3.6, Lemma 3.7 and Proposition 3.8.\\
\\
{\bf Proof of Proposition 3.5.}\\
\\
\indent Due to (2.1), Proposition 3.8 implies the conclusion (1) in Proposition 3.5.\\
\indent In fact, let $R_{1}>0$ be a fixed constant, then
$$
\left.
\begin{array}{ll}
&\displaystyle\left\|\tilde E_1(0,x)\right\|_{L^2(|x-x(t)| \leqslant R_1)}^2 + \left\| \tilde E_2(0,x)\right\|_{L^2(|x-x(t)| \leqslant R_1)}^2
\\[0.4cm]
 &\displaystyle \quad= \left\|E_1(t,x)\right\|_{L^2\left(|x-x(t)| \leqslant \frac{R_1}{\lambda(t)}\right)}^2 +  \left\|E_2(t,x)\right\|_{L^2\left(|x-x(t)| \leqslant \frac{R_1}{\lambda(t)}\right)}^2.
\end{array}
\right.
$$
Noting that $\lambda(t) \to +\infty$ as $t \to T$ and (3.119), we have
$$
\left.
\begin{array}{ll}
&\displaystyle\liminf_{t \to T} \left(\| \tilde E_1(0,x)\|_{L^2(|x-x(t)|\leqslant R_1)}^2 + \| \tilde E_2(0,x)\|_{L^2(|x-x(t)|\leqslant R_1)}^2 \right)
\\[0.4cm]
 &\displaystyle\qquad\geqslant \liminf_{t \to T} \left(\|E_1(t,x)\|_{L^2\left(|x-x(t)| \leqslant \frac{R_1}{\lambda(t)}\right)}^2 +  \|E_2(t,x)\|_{L^2\left(|x-x(t)| \leqslant \frac{R_1}{\lambda(t)}\right)}^2 \right)
 \\[0.4cm]
 &\displaystyle\qquad\geqslant \frac{\|Q\|_{L^{2}(\mathbb{R}^2)}^2}{1 + \eta}.
\end{array}
\right.
$$
On the other hand, in view of Proposition 3.8, for any fixed $R_{1}>0$, H\"older's inequality yields
$$
\left.
\begin{array}{ll}
&\displaystyle  R_1^{\frac{1}{2}}  \liminf_{t \to T} \left\|\tilde n(0,x)\right\|_{L^2(|x-x(t)| \leqslant R_1)}
\\[0.4cm]
 &\displaystyle\qquad \geqslant  \liminf_{t \to T} \left\|\tilde n(0,x) \right\|_{L^1(|x-x(t)| \leqslant R_1)}
  \\[0.4cm]
 &\displaystyle\qquad\geqslant  \liminf_{t \to T} \left\|n(t,x)\right\|_{L^1\left(|x-x(t)| \leqslant \frac{R_1}{\lambda(t)}\right)}
\\[0.4cm]
 &\displaystyle\qquad\geqslant m_n.
 \end{array}
\right.
$$
Hence (3.26) and (3.27) hold.\\
\indent We now show conclusion (2) in Proposition 3.5 by using the same scaling argument as that adopted in Proposition 3.8. \\
\indent For a sequence $t_k\to T$~($k\to +\infty$), let
$$
\left\{
\begin{array}{ll}
&\displaystyle \hat{E}_{1n}=\frac{1}{\tilde\lambda_n}\tilde{E_1}\left(\frac{x}{\tilde\lambda_n}\right),~~\hat E_{2n}=\frac{1}{\tilde\lambda_n}\tilde{E_2}\left(\frac{x}{\tilde\lambda_n}\right),
\\[0.5cm]
 &\displaystyle \hat{n}_{n} = \frac{1}{\tilde\lambda_n^2} \tilde n \left(t_n,\frac{x}{\tilde\lambda_n}\right),\hat{\textbf{v}}_n = \frac{1}{\tilde\lambda_n^2} \tilde {\textbf{v}} \left(t_n,\frac{x}{\tilde\lambda_n}\right)
 \end{array}
\right.
$$
satisfy
$$\int_{\mathbb{R}^2}\left(\left|\nabla\hat E_{1n}\right|^2+\left|\nabla\hat E_{2n}\right|^2+\frac{1}{2}\left|\hat n_n\right|^2+\frac{1}{2}\left|\hat{\textbf{v}}_n\right|^2\right)dx=1.$$
Here,
$$
\left.
\begin{array}{ll}
&\displaystyle\tilde\lambda_n^2(t)=\int_{\mathbb{R}^2}\left(\left|\nabla\tilde E_1\left(0,x\right)\right|^2+\left|\nabla\tilde E_2\left(0,x\right)\right|^2\right.
 \\[0.4cm]
 &\displaystyle\qquad\qquad\qquad\left.+\frac{1}{2}\left|\tilde n\left(0,x\right)\right|^2+\frac{1}{2}\left|\tilde v\left(0,x\right)\right|^2\right) dx,
\end{array}
\right.
$$
and $\tilde\lambda_n\to 1$ as $n\to+\infty$.
Direct calculation gives $\left(\hat E_{1n},\hat E_{2n},\hat n_n\right)$ satisfy (3.28)-(3.31). Hence for a sequence $x_n(t)\in\mathbb{R}^2$ and $x_n(t)\to x(t)$ as $n\to+\infty$, there holds
$$
\left.
\begin{array}{ll}
&\displaystyle\|\hat E_{1n}\|_{L^2(|x-x_n(t)|\leqslant R_1)}^2+\|\hat E_{2n}\|_{L^2(|x-x_n(t)|\leqslant R_1)}^2
  \\[0.4cm]
 &\displaystyle\qquad=\|\tilde E_1(0,x)\|_{L^2\left(|x-x_n(t)|\leqslant \frac{R_1}{\tilde{\lambda}_n}\right)}^2+\|\tilde E_2(0,x)\|_{L^2\left(|x-x_n(t)|\leqslant \frac{R_1}{\tilde{\lambda}_n}\right)}^2,
\end{array}
\right.
$$
 and
 $$\|\hat n_n\|_{L^2(|x-x_n(t)|\leqslant R_1)}=\|\tilde n(0,x)\|_{L^2\left(|x-x_n(t)|\leqslant \frac{R_1}{\tilde{\lambda}_n}\right)}.$$
Letting $n\to+\infty$ yields (3.32) and (3.33) due to (3.121) and (3.122), which ends the proof of Proposition 3.5.\hfill$\Box$\\

\subsection{Compactness of the Solution to the Rescaled Zakharov System (2.3)}
\qquad\\
\indent Here, we discuss the compactness of $\left(\tilde E_1(0,x),\tilde E_2(0,x),\tilde n(0,x)\right)$.\\
\begin{remark}\label{3.9}
From Proposition 3.5, it follows that $\left(\tilde E_1(0,x),\tilde E_2(0,x)\right)$ is bounded and weakly compact in $H^1(\mathbb{R}^2)\times H^1(\mathbb{R}^2)$. Then we can choose a sequence $t_n \to T$, and extract a subsequence (still denoted by $t_n$). Let
$$\left(\tilde E_1(0,x+x(t_n)),\tilde E_2(0,x+x(t_n))\right) \rightharpoonup (E'_{1},E'_{2})~~\mbox{in}~~H^1(\mathbb{R}^2)\times H^1(\mathbb{R}^2),$$
and
$$\tilde  n(0,x+x(t_n)) \rightharpoonup N'~~\mbox{in}~~L^2(\mathbb{R}^2).$$
Note that the embedding $H^1(\mathbb{R}^2) \hookrightarrow L_{loc}^2(\mathbb{R}^2)$ is compact, from Proposition 3.5, for a bounded domain $\Omega\subset \mathbb{R}^2$, there holds
$$ \left\|E'_{1}\right\|_{L^{2}(\Omega)}^2+\left\|E'_{2}\right\|_{L^{2}(\Omega)}^2 \geqslant c_{1}>0.$$
Under this case, one can not exclude the case of $N'\equiv 0$, that is, Proposition 3.5 can not  guarantee the non-vanishing property of $N'$.\hfill$\Box$
\end{remark}
\indent To overcome the difficulty mentioned in Remark 3.9, we will investigate the relation of $\left(E'_{1},E'_{2}\right)$ and $N'$ to obtain the compactness property of $N'$ following the related information for $\left(E'_{1},E'_{2}\right)$.\\
\\
\indent We now claim:
\begin{proposition}\label{3.10}
Let $t_n\to T$. Then there exists a subsequence (still denote $t_n$) such that for a sequence $x_n:=x(t_n)\in\mathbb{R}^2$ and $\left(E_1',E_2',N'\right)\in H^1(\mathbb{R}^2)\times H^1(\mathbb{R}^2)\times L^2(\mathbb{R}^2)$, as $n\to+\infty$, the conclusions hold as below:
$$\left(\tilde E_1\left(0,x+x_n\right),\tilde E_2\left(0,x+x_n\right)\right)\rightharpoonup\left(E_1',E_2'\right)~~\mbox{in}~~H^1(\mathbb{R}^2)\times H^1(\mathbb{R}^2),\eqno(3.153)$$
and
$$\tilde{n}(0,x+x_n)\rightharpoonup N'~~\mbox{in}~~L^2(\mathbb{R}^2).\eqno(3.154)$$
Furthermore, there exist constants $\beta_1>0$ and $R_1>0$ depending only on\\ $ \|E_{10}\|_{L^2(\mathbb{R}^2)},~\|E_{20}\|_{L^2(\mathbb{R}^2)})$ such that
$$\left(\left\|E_1'\right\|_{L^2(|x|\leqslant R_1)}^2+\left\|E_2'\right\|_{L^2(|x|\leqslant R_1)}^2\right)^{\frac{1}{2}}\geqslant\beta_1,\eqno(3.155)$$
$$\mathcal{H}\left(E_1',E_2',N',0\right)\leqslant 0.\eqno(3.156)$$
\end{proposition}
{\bf Proof.} In view of Proposition 3.5, we will show this proposition by implementing the classical iteration technique, Concentration-compactness principle and mathematical induction.\\
\indent Let $\beta_{1}$ be defined as in (ii) of Proposition 3.5. Assume that $\left(\hat{E}_{1n},\hat{E}_{2n},\hat{n}_n,\hat{\textbf{v}}_n\right)\in H^1(\mathbb{R}^2)\times H^1(\mathbb{R}^2)\times L^2(\mathbb{R}^2)\times L^2(\mathbb{R}^2)$ satisfy
$$
\left\{
\begin{array}{ll}
&\displaystyle \hat{E}_{1n}=\frac{1}{\lambda_n(t_n)}\tilde{E}_1\left(t_n+\frac{s}{\lambda_n(t_n)},
\frac{x}{\lambda_n(t_n)}\right),
\\[0.4cm]
&\displaystyle \hat{E}_{2n}=\frac{1}{\lambda_n(t_n)}\tilde{E}_2\left(t_n+\frac{s}{\lambda_n(t_n)},
\frac{x}{\lambda_n(t_n)}\right),
\\[0.4cm]
&\displaystyle \hat{n}_n=\frac{1}{\lambda_n^2(t_n)}\tilde{n}\left(t_n+\frac{s}{\lambda_n(t_n)},
\frac{x}{\lambda_n(t_n)}\right),
\\[0.4cm]
&\displaystyle \hat{\textbf{v}}_n=\frac{1}{\lambda_n^2(t_n)}\tilde{\textbf{v}}\left(t_n+\frac{s}{\lambda_n(t_n)},
\frac{x}{\lambda_n(t_n)}\right),
 \end{array}
\right.\eqno(3.157)
$$
\\
$$\lambda_n^2(t_n)=\int_{\mathbb{R}^2}^{}\left|\nabla \tilde{E}_{1} \right|^2\,dx+\int_{\mathbb{R}^2}^{}\left|\nabla \tilde{E}_{2}\right|^2\,dx+\frac{1}{2}\int_{\mathbb{R}^2}^{}\left|\tilde{n}\right|^2\,
dx+\frac{1}{2}\int_{\mathbb{R}^2}^{}\left|\tilde{\textbf{v}}\right|^2\,dx.
\eqno(3.158)$$
\indent Direct calculation yields
$$\int_{\mathbb{R}^2}\left(\left|\hat{E}_{1n}\right|^2+\left|\hat{E}_{2n}\right|^2\right)dx\leqslant
\int_{\mathbb{R}^2}\left(|E_{10}|^2+|E_{20}|^2\right)dx,\eqno(3.159)$$
$$\lim_{n\to+\infty}\mathcal{H}\left(\hat{E}_{1n},\hat{E}_{2n},\hat{n}_n,
\hat{\textbf{v}}_n\right)=0,\eqno(3.160)$$
$$\int_{\mathbb{R}^2}\left(\left|\nabla \hat{E}_{1n}\right|^2+\left|\nabla \hat{E}_{2n}\right|^2+\frac{\left|\hat{n}_n\right|^2}{2}+\frac{\left|\hat{\textbf{v}}_n\right|^2}{2}\right)dx=1.\eqno(3.161)$$
We will show that there exist $\left(E_1',E_2',N'\right) \in H^1(\mathbb{R}^2)\times H^1(\mathbb{R}^2)\times L^2(\mathbb{R}^2)$ and a sequence $x_n\in\mathbb{R}^2$ such that as $n\to+\infty$,

$$\left(\hat{E}_{1n}(x_n+x),\hat{E}_{2n}(x_n+x)\right)
\rightharpoonup\left(E_1'(x),E_2'(x)\right)~~\mbox{in}~~H^1(\mathbb{R}^2)\times H^1(\mathbb{R}^2),\eqno(3.162)$$
\\
$$\hat{n}_n(x_n+x)\rightharpoonup N'(x) ~~\mbox{in}~~L^2(\mathbb{R}^2),\eqno(3.163)$$
and
$$\left(\int_{|x|\leqslant R_1}\left(\left|E_1'\right|^2+\left|E_2'\right|^2\right) dx\right)^{\frac{1}{2}}\geqslant\beta_1,~~\mathcal{H}\left(E_1',E_2',N'\right)\leqslant 0.\eqno(3.164)$$
By Proposition 3.3 and Corollary 3.4, there exist constants $c_1>0,~~c_2>0$ such that
$$c_1\leqslant\int_{\mathbb{R}^2}\left(\left|\nabla \hat{E}_{1n}\right|^2+\left|\nabla \hat{E}_{2n}\right|^2\right)dx\leqslant c_2,~~c_1\leqslant\int_{\mathbb{R}^2}\left|\hat{n}_n\right|^2dx\leqslant c_2.\eqno(3.165)$$
Let the integer $k_{0}$ be defined by

$$\left(\int_{\mathbb{R}^2}\left( \left|\hat{E}_{1n}\right|^2+\left|\hat{E}_{2n}\right|^2\right)dx\right)
^{\frac{1}{2}}<(k_0+1)\beta_1,
\eqno(3.166)$$
where
$$k_0=1,2,...,E\left[\frac{\|E_{10}\|_{L^2(\mathbb{R}^2)}^2
+\|E_{20}\|_{L^2(\mathbb{R}^2)}^2}{\beta_1}-1\right].\eqno(3.167)$$
We then show Proposition 3.10 by induction on the integer $k_0$.\\
\\
(1)\quad For $k_0=1$, there holds $\dfrac{\|E_{10}\|_{L^2(\mathbb{R}^2)}^2
+\|E_{20}\|_{L^2(\mathbb{R}^2)}^2}{\beta_1}=2$, that is,
$$\|E_{10}\|_{L^2(\mathbb{R}^2)}^2
+\|E_{20}\|_{L^2(\mathbb{R}^2)}^2=2\beta_1.\eqno(3.168)$$
From (3.159) it follows that
$$\left\|\hat{E}_{1n}\right\|_{L^2(\mathbb{R}^2)}^2
+\left\|\hat{E}_{2n}\right\|_{L^2(\mathbb{R}^2)}^2\leqslant\|E_{10}\|_{L^2(\mathbb{R}^2)}^2
+\|E_{20}\|_{L^2(\mathbb{R}^2)}^2=2\beta_1.\eqno(3.169)$$
Hence (3.166) holds for $k_0=1$. From compactness and the boundedness of weakly convergent sequence, similar argument to the proof of Lemma 3.7 yields
$$\mathcal{H}\left(E_1',E_2',N'\right)\leqslant\lim_{n\to+\infty}
\mathcal{H}\left(\hat{E}_{1n},\hat{E}_{2n},\hat{n}_n\right)\leqslant 0.$$
So (3.162)-(3.165) are true. Note that Proposition 3.5,  Proposition 3.10 holds for $k_0=1$.\\
\\
(2)\quad Assume that (3.162)-(3.165) are true for $k_0>1$, we then show that they also hold for $k_{0}+1$.\\
\\
\indent Let $\left(\hat{E}_{1n},\hat{E}_{2n},\hat{n}_n\right)$ be the sequence satisfying (3.164), (3.165), (3.166) and
$$\lim_{n\to+\infty}\mathcal{H}\left(\hat{E}_{1n},\hat{E}_{2n},\hat{n}_n,0\right)\leqslant 0.$$
In view of Proposition 3.5, we may assume that there exist a sequence $x_n\in\mathbb{R}^2$ and a constant $R=R(c_1,c_2)$ satisfying
$$\left(\int_{|x-x_n|\leqslant R}\left(\left|\hat{E}_{1n}\right|^2+\left|\hat{E}_{2n}\right|
^2\right)dx\right)^{\frac{1}{2}}\geqslant\beta_1,
\eqno(3.170) $$
and $\left(\hat E_1,\hat E_2,\hat N\right)\in H^1(\mathbb{R}^2)\times H^1(\mathbb{R}^2)\times L^2(\mathbb{R}^2)$ such that
$$\left(\hat{E}_{1n}\left(x+x_n\right),\hat{E}_{2n}\left(x+x_n\right)\right)\rightharpoonup
\left(\hat E_1,\hat E_2\right)~~\mbox{in}~~H^1(\mathbb{R}^2)\times H^1(\mathbb{R}^2), $$
$$\hat{n}_n(x+x_n)\rightharpoonup\hat N ~~\mbox{in}~~L^2(\mathbb{R}^2). $$
We extract a subsequence still denoted by $\left(\hat{E}_{1n},\hat{E}_{2n},\hat{n}_{n}\right)$ for simplicity. We make the following decomposition for the extracted subsequence:
$$
\left\{
\begin{array}{ll}
&\displaystyle \hat{E}_{1n}(x+x_n)=\hat{E}_{1n,1}(x+x_n)+\hat{E}_{1n,2}(x+x_n),
\\[0.3cm]
&\displaystyle \hat{E}_{2n}(x+x_n)=\hat{E}_{2n,1}(x+x_n)+\hat{E}_{2n,2}(x+x_n),
\\[0.3cm]
&\displaystyle \hat{n}_n(x+x_n)=\hat{n}_{n,1}(x+x_n)+\hat{n}_{n,2}(x+x_n).
 \end{array}
\right.
$$
Let $R_n\to+\infty$ as $n\to +\infty$. This decomposition admits the following properties:\\
\\
(I)\quad
$$
\left.
\begin{array}{ll}
&\displaystyle \hat{E}_{1n,1}(x)=\hat{E}_{2n,1}(x)=\hat{n}_{n,1}(x)=0,
 \quad|x|\leqslant\frac{R_n}{2},
  \\[0.3cm]
 &\displaystyle \hat{E}_{1n,2}(x)=\hat{E}_{2n,2}(x)=\hat{n}_{n,2}(x)=0, \quad|x|\geqslant R_n.
\end{array}
\right.
$$
\\
(II)\quad As $n\to+\infty$,
$$
\left.
\begin{array}{ll}
&\displaystyle\int_{\mathbb{R}^2}\left(\left|\hat{E}_{1n,1}\right|^2
+\left|\hat{E}_{1n,2}\right|^2
+\left|\hat{E}_{2n,1}\right|^2+
\left|\hat{E}_{2n,2}\right|^2\right)dx
  \\[0.4cm]
  &\displaystyle \qquad\qquad\qquad\qquad\qquad\qquad-\int_{\mathbb{R}^2}\left(\left|\hat{E}_{1n}\right|^2
  +\left|\hat{E}_{2n}\right|^2\right)dx\to 0,
      \\[0.5cm]
      &\displaystyle\int_{\mathbb{R}^2}\left(\left|\nabla \hat{E}_{1n,1}\right|^2+\left|\nabla \hat{E}_{1n,2}\right|^2+\left|\nabla \hat{E}_{2n,1}\right|^2+\left|\nabla \hat{E}_{2n,2}\right|^2\right)dx
  \\[0.4cm]
  &\displaystyle \qquad\qquad\qquad\qquad\qquad\qquad-\int_{\mathbb{R}^2}\left(\left|\nabla \hat{E}_{1n}\right|^2+\left|\nabla \hat{E}_{2n}\right|^2\right)dx\to 0,
   \\[0.5cm]
   &\displaystyle\int_{\mathbb{R}^2}\left(\left|\hat{n}_{n,1}\right|^2+\left|\hat{n}_{n,2}\right|^2\right)dx
   -\int_{\mathbb{R}^2}\left|\hat{n}_{n}\right|^2dx\to 0,
   \end{array}
\right.
$$
\\
(III)
$$\quad\lim_{n\to+\infty}\mathcal{H}\left(\hat{E}_{1n,1},\hat{E}_{2n,1},\hat{n}_{n,1},
0\right)+\lim_{n\to+\infty}\mathcal{H}\left(\hat{E}_{1n,2},\hat{E}_{2n,2},\hat{n}_{n,2},0\right)\leqslant 0.$$
\\
Note that as $n\to+\infty$,
$$\int_{\mathbb{R}^2}\left(\left |\hat{E}_{1n,1}\right|^2+ \left|\hat{E}_{2n,1}\right|^2\right)dx\to
\int_{\mathbb{R}^2}\left(\left |\hat E_1\right|^2+\left |\hat E_2\right|^2\right)dx.$$
Especially, $\displaystyle \int_{\mathbb{R}^2}\left(\left|\hat{E_1}\right|^2+\left|\hat{E_2}\right|^2\right)
dx \geqslant\beta_1$. So when $n$ large enough, there holds
$$ \int_{\mathbb{R}^2}\left( \left|\hat{E}_{1n,2}\right|^2+\left|\hat{E}_{2n,2}\right|^2\right)dx <k_0\beta_1.$$
We proceed our argument by dividing two cases:\\
\\
\indent {\bf Case 1:}
$$\mathcal{H}\left(\hat{E_1},\hat{E_2},\hat N,0\right)\leqslant\lim_{n\to+\infty}
\mathcal{H}\left(\hat{E}_{1n,1},\hat{E}_{2n,1},\hat{n}_{n,1},0\right)\leqslant 0.$$
In this case, letting $E_1'=\hat E_1,~~E_2'=\hat E_2,~~N'=\hat N$, one can get the conclusion of Proposition 3.10.\\
\\
\indent {\bf Case 2:}
$$\mathcal{H}\left(\hat{E_1},\hat{E_2},\hat N,0\right)>0.$$
In this case, let $P_1=\lim\limits_{n\to+\infty}\mathcal{H}\left(\hat{E}_{1n,1},\hat{E}_{2n,1},
\hat{n}_{n,1},0 \right)>0$, and for $n$ large enough,
$$\mathcal{H}\left(\hat{E}_{1n,2},\hat{E}_{2n,2},\hat{n}_{n,2},0\right)\leqslant-\frac{P_1}{2}<0.$$
By induction assumption, there exists a sequence $y_n\in\mathbb{R}^2$ such that
$$\left(\hat{E}_{1n,2}\left(+y_n\right),\hat{E}_{2n,2}\left(+y_n\right)\right)\rightharpoonup \left(E_1',E_2'\right)~~\mbox{in}~~H^1(\mathbb{R}^2)\times H^1(\mathbb{R}^2), $$
\\
$$\hat{n}_{n,2}(+y_n)\rightharpoonup N'~~\mbox{in}~~L^2(\mathbb{R}^2),$$
with
$$\left(\int_{|x|\leqslant R_1}\left(\left|E_1'\right|^2+\left|E_2'\right|^2\right)dx\right)^{\frac{1}{2}}\geqslant\beta_1,
\quad\mathcal{H}\left(E_1',E_2',N',0\right)\leqslant 0.$$
Using translation transform finishes the proof of Proposition 3.10.\hfill$\Box$\\
\\
\begin{corollary}\label{3.11}
There exists a constant $c_*>0$ such that
$$
c_*\leqslant\|N'\|_{L^2(\mathbb{R}^2)}\leqslant 1.\eqno(3.171)
$$
\end{corollary}
{\bf Proof.} From Proposition 3.10 it follows that
$$
\left.
\begin{array}{ll}
&\displaystyle\mathcal{H}\left(E_1',E_2',N'\right)
  \\[0.3cm]
 &\displaystyle\qquad= \int_{\mathbb{R}^2}\left(\left|\nabla E_1'\right|^2+\left|\nabla E_2'\right|^2+N'\left(\left|E_1'\right|^2+\left|E_2'\right|^2\right)+\frac{1}{2}N'^2\right)dx
  \\[0.4cm]
  &\displaystyle\qquad\quad+\frac{\eta}{2}\int_{\mathbb{R}^2}\left[\left( (E_1' )^2
\left(\overline{E_2'}\right)^2+\left(E_2'\right)^2\left(\overline{E_1'}
\right)^2\right)-2\left|E_1'\right|^2\left|E_2'\right|^2\right]dx
   \\[0.4cm]
 &\displaystyle\qquad\leqslant  0,
\end{array}
\right.\eqno(3.172)
$$
and
$$\left\|E_1'\right\|_{L^2(|x|\leqslant R_1)}^2+\left\|E_2'\right\|_{L^2(|x|\leqslant R_1)}^2\geqslant\beta_1.\eqno(3.173)$$
Sobolev embedding theorem $\left(  \forall q\in[n,+\infty),~~W^{1,n}(\mathbb{R}^n)\subset L^{q}(\mathbb{R}^n)  \right)$ yields $H^1(\mathbb{R}^2)\subset L^2(\mathbb{R}^2)$. In view of (3.161), there exists $c_{1}>0$ such that
$$0<(c_{1})^2\leqslant\int_{\mathbb{R}^2}\left(\left|\nabla E_1'\right|^2+\left|\nabla E_2'\right|^2\right)dx\leqslant 1.\eqno(3.174)$$
On one hand, by (3.172) one has
$$
\left.
\begin{array}{ll}
&\displaystyle-\int_{\mathbb{R}^2}N'\left(\left|E_1'\right|^2+\left|E_2'\right|^2\right)dx\qquad\qquad
  \\[0.4cm]
 &\displaystyle\quad\geqslant\int_{\mathbb{R}^2}\left(\left|\nabla E_1'\right|^2+\left|\nabla E_2'\right|^2\right)dx-\eta\int_{\mathbb{R}^2}\left|E_1'\right|^2\left|E_2'\right|^2dx
 \\[0.4cm]
 &\displaystyle\quad\quad+\frac{\eta}{2}\int_{\mathbb{R}^2}\left(\left(E_1'\right)^2
\left(\overline{E_2'}\right)^2+\left(E_2'\right)^2\left(\overline{E_1'}\right)^2
\right)dx
 \\[0.4cm]
 &\displaystyle\quad\geqslant\int_{\mathbb{R}^2}\left(\left|\nabla E_1'\right|^2+\left|\nabla E_2'\right|^2\right)dx-\frac{\eta}{2}\int_{\mathbb{R}^2}\left(\left|E_1'\right|^2
+\left|E_2'\right|^2\right)^2dx
 \\[0.4cm]
  &\displaystyle\quad\geqslant \int_{\mathbb{R}^2}\left(\left|\nabla E_1'\right|^2+\left|\nabla E_2'\right|^2\right)dx
 \\[0.4cm]
 &\displaystyle\quad\quad-\frac{\eta}{\|Q\|_{L^2(\mathbb{R}^2)}^2}
\int_{\mathbb{R}^2}\left(\left|E_1'\right|^2+\left|E_2'\right|^2\right)dx
\int_{\mathbb{R}^2}\left(\left|\nabla E_1'\right|^2+\left|\nabla E_2'\right|^2\right)dx
 \\[0.4cm]
  &\displaystyle\quad= \left[1-\frac{\eta}{\|Q\|_{L^2(\mathbb{R}^2)}^2}\int_{\mathbb{R}^2}
  \left(\left|E_{10}\right|^2
+\left|E_{20}\right|^2\right)dx\right]
 \\[0.5cm]
&\displaystyle\qquad\qquad\qquad\qquad\cdot\int_{\mathbb{R}^2}\left(\left|\nabla E_1'\right|^2+\left|\nabla E_2'\right|^2\right)dx.
\end{array}
\right.\eqno(3.175)
$$
Let
$$c_0=1-\frac{\eta}{\|Q\|_{L^2(\mathbb{R}^2)}^2}\int_{\mathbb{R}^2}
\left(|E_{10}|^2+|E_{20}|^2\right)dx>0.\eqno(3.176)$$
According to (3.4),
one has $0<c_0<\dfrac{1}{1+\eta}$. On the other hand,
$$
\left.
\begin{array}{ll}
&\displaystyle-\int_{\mathbb{R}^2}N'\left(\left|E_1'\right|^2+\left|E_2'\right|^2\right)dx
 \\[0.4cm]
 &\displaystyle\quad\leqslant \left(\int_{\mathbb{R}^2}\left|N'\right|^2dx\right)^{\frac{1}{2}}
\left(\int_{\mathbb{R}^2}\left(\left|E_1'\right|^2+\left|E_2'\right|^2\right)^2dx\right)^{\frac{1}{2}}
 \\[0.4cm]
 &\displaystyle\quad\leqslant \sqrt{2}\left\|N'\right\|_{L^2(\mathbb{R}^2)}\frac{1}{\|Q\|_{L^2(\mathbb{R}^2)}}
\left[\int_{\mathbb{R}^2}\left(\left|E_1'\right|^2+\left|E_2'\right|^2\right)dx\right.
 \\[0.4cm]
 &\displaystyle\qquad\qquad\left.\cdot\int_{\mathbb{R}^2}\left(\left|\nabla E_1'\right|^2+\left|\nabla E_2'\right|^2\right)dx\right]^{\frac{1}{2}}
 \\[0.4cm]
  &\displaystyle\quad
   \leqslant \sqrt{2}\left\|N'\right\|_{L^2(\mathbb{R}^2)}\cdot\frac{1}{\sqrt{\eta}}
\left(\int_{\mathbb{R}^2}\left(\left|\nabla E_1'\right|^2+\left|\nabla E_2'\right|^2\right)dx\right)^{\frac{1}{2}}.
\end{array}
\right.
$$
Combining (3.174) with (3.175) and (3.176) yields
$$\|N'\|_{L^2(\mathbb{R}^2)}\geqslant\sqrt{\dfrac{\eta}{2}}c_0 c_1=c^*>0.$$
From (3.174) and (3.176) it follows that
$0<c_0<\dfrac{1}{1+\eta},~~0<c_1\leqslant 1 $. We then conclude
$$c^*=\sqrt{\frac{\eta}{2}}c_0 c'=\sqrt{\frac{\eta}{2}}c_0 c_1<\sqrt{\frac{\eta}{2}}\frac{1}{1+\eta}c_1<1.$$
So far, the proof of Corollary 3.11 is completed.\hfill$\Box$\\
\\
\subsection{Proof of Theorem 3.1.}
\qquad\\
\\
\indent In this subsection, we shall establish some estimates for $\left(\tilde E_1,\tilde E_2 ,\tilde n,\tilde{\textbf{v}}\right)(0)$ based on these estimates obtained in subsection 3.1, subsection 3.2 and subsection 3.3. Next, by considering the rescaled Zakharov system (2.3c)-(2.3d):
$$\tilde n_s=-\nabla\cdot\tilde{\textbf{v}},\eqno(3.177)$$
\\
$$\tilde {\textbf{v}}_s=-\nabla\left(\tilde n+\left|\tilde E_{1}\right|^{2}+\left|\tilde E_{2}\right|^{2}\right),\eqno(3.178)$$
We then finish the proof of Theorem 3.1.\\
\\
{\bf Proof of Theorem 3.1.}\\
\\
\indent $\forall~t>0$, we consider $\left(\tilde E_{1}(s),\tilde E_{2}(s),\tilde n(s),\tilde {\textbf{v}}(s)\right)$ for $ \left[0,\lambda(t)(T-t)\right)$. From (2.5) it follows that
$$\left\| \left(\tilde E_{1}(0),\tilde E_{2}(0),\tilde n(0),\tilde {\textbf{v}}(0) \right)\right\|_{H^{1}(\mathbb{R}^2)\times H^{1}(\mathbb{R}^2)\times L^{2}(\mathbb{R}^2)\times L^{2}(\mathbb{R}^2)}^2 = 1,$$
\\
$$\lim_{s \to \lambda(t)(T-t)} \left\| \left(\tilde E_{1}(s),\tilde E_{2}(s),\tilde n(s),\tilde {\textbf{v}}(s)\right)\right\|_{H^{1}(\mathbb{R}^2)\times H^{1}(\mathbb{R}^2)\times L^{2}(\mathbb{R}^2)\times L^{2}(\mathbb{R}^2)}^2 = + \infty.$$
\\
Let $A>1$ be a fixed constant. By the continuity of $s$, there exists a $\theta(t) >0$ such that $\forall  s\in \left[0,\theta(t)\right]$,
$$\left\| \left(\tilde E_{1}(s),\tilde E_{2}(s),\tilde n(s),\tilde {\textbf{v}}(s)\right)\right\|_{H^{1}(\mathbb{R}^2)\times H^{1}(\mathbb{R}^2)\times L^{2}(\mathbb{R}^2)\times L^{2}(\mathbb{R}^2)}^2 \leqslant A,\eqno(3.179)$$
\\
$$\left\| \left(\tilde E_{1}(\theta(t)),\tilde E_{2}(\theta(t)),\tilde n(\theta(t)),\tilde {\textbf{v}}(\theta(t)\right)\right\|_{H^{1}(\mathbb{R}^2)\times H^{1}(\mathbb{R}^2)\times L^{2}(\mathbb{R}^2)\times L^{2}(\mathbb{R}^2)}^2 = A.\eqno(3.180)$$
We now claim that as $t\rightarrow T$, there exists a uniform lower bound $\theta_{0}>0$ for $\theta(t)$, that is,
$$\theta(t)\geqslant \theta_{0}.\eqno(3.181)$$
We proceed our discussion through two sides. On one hand, some properties for $\left(\tilde E_{1}(\theta),\tilde E_{2}(\theta),\tilde n(\theta),\tilde {\textbf{v}}(\theta)\right)$ will be established. On the other hand, in view of (3.177) and the compactness for $\tilde n$, we complete the proof of (3.181) by contradiction.\\
\\
\indent Firstly we claim that:
\begin{proposition}\label{3.12}
\indent\quad There exist constants $c_1>0$ and $c_2>0$ independent of $t$ and $A>1$ such that\\
$(1)$~~~~$\forall s\in[0,\theta(t)]$,\\
$$\left(\left\|\nabla\tilde E_1(s)\right\|_{L^2(\mathbb{R}^2)}^2+\left\|\nabla\tilde E_2(s)\right\|_{L^2(\mathbb{R}^2)}^2\right)^{\frac{1}{2}}\leqslant Ac_2,\eqno(3.182)$$
$$\left\|\tilde n(s)\right\|_{L^2(\mathbb{R}^2)}\leqslant Ac_2,~~~~\left\|\tilde{\textbf{v}}(s)\right\|_{L^2(\mathbb{R}^2)}\leqslant Ac_2.\eqno(3.183)$$
\\
$(2)$ Let $t_n\to T$. Extracting a subsequence, still denoted by $t_n$, such that for sequences $x_n:=x(t_n)\in\mathbb{R}^2$,~~$ \left(E_1',E_2',N'\right)\in H^1(\mathbb{R}^2)\times H^1(\mathbb{R}^2)\times L^2(\mathbb{R}^2)$, there hold
$$
\left.
\begin{array}{ll}
&\displaystyle\left(\tilde E_1\left(t_n,\theta\left(t_n\right),x-x_n\right),\tilde E_2\left(t_n,\theta\left(t_n\right),x-x_n\right)\right)
  \\[0.4cm]
  &\displaystyle\qquad=\left(\tilde E_1\left(\theta\left(t_n\right),x-x_n\right),\tilde{E_2}
  \left(\theta\left(t_n\right),x-x_n\right)\right)
  \\[0.4cm]
 &\displaystyle\qquad\qquad\rightharpoonup(E_1',E_2')~~\mbox{in}~~H^1(\mathbb{R}^2)\times H^1(\mathbb{R}^2),
\end{array}
\right.\eqno(3.184)
$$

$$\tilde n\left(t_n,\theta(t_n),x-x_n\right)=\tilde n\left(\theta(t_n),x-x_n\right)\rightharpoonup N'
~~\mbox{in}~~L^2(\mathbb{R}^2),\eqno(3.185)$$

$$\left(\left\|\nabla E_1'\right\|_{L^2(\mathbb{R}^2)}^2+\left\|\nabla E_2'\right\|_{L^2(\mathbb{R}^2)}^2\right)^{\frac{1}{2}}\geqslant A c_1, \quad\left\|N'\right\|_{L^2(\mathbb{R}^2)}\geqslant A c_1.\eqno(3.186)$$
\end{proposition}
\qquad\\
{\bf Proof.} It follows from (3.179) and (3.180) that
$$\int_{\mathbb{R}^2}\left(\left|\nabla\tilde E_1(s) \right|^2+\left|\nabla\tilde E_2(s)\right|^2+\frac{1}{2}\left|\tilde n(s)\right|^2+\frac{1}{2}\left|\tilde{\textbf{v}}(s)\right|^2\right)dx\leqslant A^2.\eqno(3.187)$$
Note that (2.1), there holds
$$
\left.
\begin{array}{ll}
&\displaystyle\left\|\left(\tilde E_1,\tilde E_2,\tilde n,\tilde{\textbf{v}}\right)(s)\right\|_{H^{1}(\mathbb{R}^2)\times H^{1}(\mathbb{R}^2)\times L^{2}(\mathbb{R}^2)\times L^{2}(\mathbb{R}^2)}^2
  \\[0.4cm]
  &\displaystyle\quad=\int_{\mathbb{R}^2}\left(\left|\nabla\tilde E_1(t,s)\right|^2+\left|\nabla\tilde E_2(t,s)\right|^2+\frac{1}{2}\left|\tilde
n(t,s)\right|^2+\frac{1}{2}\left|\tilde{\textbf{v}}(t,s)\right|^2\right)dx
  \\[0.4cm]
  &\displaystyle\quad=\frac{1}{\lambda^2(t)}\left[\int_{\mathbb{R}^2}\left(\left|\nabla E_1 \left(t+\frac{s}{\lambda(t)}\right)\right|^2+\left|\nabla
E_2\left(t+\frac{s}{\lambda(t)}\right)\right|^2\right.\right.
\\[0.4cm]
  &\displaystyle\qquad\qquad+\left.\left.\frac{1}{2}\left|n\left(t+\frac{s}{\lambda(t)}\right)
\right|^2+\frac{1}{2}
\left|\textbf{v}\left(t+\frac{s}{\lambda(t)}\right)\right|^2\right)dx\right]
 \\[0.4cm]
 &\displaystyle\quad=\left(\frac{\lambda\left(t+\frac{s}{\lambda(t)}\right)}{\lambda(t)}\right)^2.
\end{array}
\right.\eqno(3.188)
$$
So $\forall s\in[0,\theta(t)]$,
$$\frac{\lambda\left(t+\frac{s}{\lambda(t)}\right)}{\lambda(t)}\leqslant A,~~\frac{\lambda\left(t+\frac{\theta(t)}{\lambda(t)}\right)}{\lambda(t)}=A.\eqno(3.189)$$
This yields that
$$
\left.
\begin{array}{ll}
\tilde E_1(t,\theta(t),x)&\displaystyle=\frac{1}{\lambda(t)}E_1\left(t+\frac{\theta(t)}{\lambda(t)},
\frac{x}{\lambda(t)}\right)
  \\[0.4cm]
  &\displaystyle=\frac{A}{A\lambda(t)}E_1\left(t+\frac{\theta(t)}{\lambda(t)},
\frac{Ax}{A\lambda(t)}\right)
  \\[0.4cm]
  &\displaystyle =\frac{A}{\lambda\left(t+\frac{\theta(t)}{\lambda(t)}\right)}
E_1\left(t+\frac{\theta(t)}{\lambda(t)},\frac{Ax}
{\lambda\left(t+\frac{\theta(t)}{\lambda(t)}\right)}\right)
\\[0.4cm]
  &\displaystyle =A\tilde E_1\left(t+\frac{\theta(t)}{\lambda(t)},0,Ax\right).
  \end{array}
\right.\eqno(3.190)
$$
Similar argument to (3.190) yields
$$\tilde E_2 \left(t,\theta(t),x \right)=A\tilde E_2\left(t+\frac{\theta(t)}{\lambda(t)},0,Ax\right),\eqno(3.191)$$
\\
$$\tilde n\left(t,\theta(t),x\right)=A^2\tilde n\left(t+\frac{\theta(t)}{\lambda(t)},0,Ax\right).\eqno(3.192)$$
As $t_n\to T$, there holds $t_n+\frac{\theta(t_n)}{\lambda(t_n)}\to T$. Hence from Proposition 3.10 and Corollary 3.11, it follows that there exists a subsequence (still denoted by $t_n$) such that for sequences $x_n\in\mathbb{R}^2,~~\left(E_1',E_2',N'\right)\in H^1(\mathbb{R}^2)\times H^1(\mathbb{R}^2)\times L^2(\mathbb{R}^2)$, as $n\to+\infty$, the following conclusions hold:
$$
\left.
\begin{array}{ll}
&\displaystyle\left(\tilde E_1\left(t_n+\frac{\theta(t_n)}{\lambda(t_n)},0,x+x_n\right),\tilde E_2\left(t_n+\frac{\theta(t_n)}{\lambda(t_n)},0,x+x_n \right)\right)
  \\[0.4cm]
 &\qquad\qquad\displaystyle \rightharpoonup \left(E_1',E_2'\right) ~~\mbox{in}~~H^1(\mathbb{R}^2)\times H^1(\mathbb{R}^2),
\end{array}
\right.\eqno(3.193)
$$
$$\tilde n\left(t_n+\frac{\theta(t_n)}{\lambda(t_n)},0,x+x_n\right)\rightharpoonup N'~~\mbox{in}~~L^2(\mathbb{R}^2).\qquad\qquad\eqno(3.194)$$
In addition, in view of Proposition 3.10 and Corollary 3.11, there exists a constant $c_1>0$ such that
$$\left(\left\|E_1'\right\|_{L^2(\mathbb{R}^2)}^2+\left\|E_2'\right\|_{L^2(\mathbb{R}^2)}^2\right)^{\frac{1}{2}}\geqslant c_1,\quad\left\|N'\right\|_{L^2(\mathbb{R}^2)}\geqslant c_1.\eqno(3.195)$$
Hence, combining (3.190)-(3.195) yields
$$
\left.
\begin{array}{ll}
&\displaystyle\left(\tilde E_1\left(t_n,\theta\left(t_n\right),x+\frac{x_n}{A}\right),\quad\tilde E_2\left(t_n,\theta\left(t_n\right),x+\frac{x_n}{A}\right)\right)
  \\[0.4cm]
 &\displaystyle  = \left(A\tilde E_1\left(t_n+\frac{\theta\left(t_n\right)}{\lambda\left(t_n\right)},0,Ax+x_n\right),A\tilde E_2\left(t_n+\frac{\theta\left(t_n\right)}{\lambda\left(t_n\right)},
 0,Ax+x_n\right)\right)
   \\[0.5cm]
   &\displaystyle\quad\rightharpoonup
\left(AE_1'\left(Ax\right),AE_2'\left(Ax\right)\right)~~\mbox{in}~~H^1(\mathbb{R}^2)\times H^1(\mathbb{R}^2),
\end{array}
\right.\eqno(3.196)
$$
$$
\left.
\begin{array}{ll}
\tilde n\left(t_n,\theta(t_n),x+\frac{x_n}{A}\right)&\displaystyle=A^2\tilde
n\left(t_n+\frac{\theta(t_n)}{\lambda(t_n)},0,Ax+x_n\right)
  \\[0.4cm]
 &\displaystyle\rightharpoonup A^2 N'(Ax)~~\mbox{in}~~L^2(\mathbb{R}^2),
\end{array}
\right.\eqno(3.197)
$$
and
$$
\left.
\begin{array}{ll}
&\displaystyle\left(\left\|\nabla\left[AE_1'\left(Ax\right)\right]\right\|
_{L^2(\mathbb{R}^2)}^2+\left\|\nabla\left[AE_2'\left(Ax\right)\right]\right\|
_{L^2(\mathbb{R}^2)}^2\right)^{\frac{1}{2}}
  \\[0.4cm]
 &\displaystyle \qquad=A\left(\left\|\nabla E_1'\right\|_{L^2(\mathbb{R}^2)}^2+\left\|\nabla E_2'\right\|_{L^2(\mathbb{R}^2)}^2\right)^{\frac{1}{2}}
  \\[0.4cm]
&\displaystyle \qquad\geqslant Ac_1,
\end{array}
\right.\eqno(3.198)
$$
\\
$$\left\|A^2 N'(Ax)\right\|_{L^2(\mathbb{R}^2)}=A\left\|N'\right\|_{L^2(\mathbb{R}^2)}\geqslant Ac_1.\eqno(3.199)$$
\\
\indent This finishes the proof of Proposition 3.12.\hfill$\Box$ \\
\\
\begin{remark}\label{3.13}
By the property of weak convergence, from (3.196)-(3.199) it follows that
$$
\left.
\begin{array}{ll}
c_1&\displaystyle\leqslant\left(\left\|\nabla E_1'\right\|_{L^2(\mathbb{R}^2)}^2
+\left\|\nabla E_2'\right\|_{L^2(\mathbb{R}^2)}^2\right)^{\frac{1}{2}}
\\[0.3cm]
&\displaystyle\leqslant\liminf_{n\to+\infty}
\left(\left\|\nabla\tilde E_1\right\|_{L^2(\mathbb{R}^2)}^2+\left\|\nabla\tilde E_2\right\|_{L^2(\mathbb{R}^2)}^2\right)^{\frac{1}{2}},
\end{array}
\right.
$$

$$c_1\leqslant\left\|N'\right\|_{L^2(\mathbb{R}^2)}\leqslant\liminf_{n\to+\infty}\left\|\tilde n\right\|_{L^2(\mathbb{R}^2)}.\qquad\qquad\qquad\qquad$$
\qquad\hfill$\Box$
\end{remark}
\indent Next, we fixed $A$ such that $A c_1\geqslant 4$.\\
\begin{remark}\label{3.14}
 ~~Note that $\left\|\tilde n(t,0)\right\|_{L^2(\mathbb{R}^2)}\leqslant \sqrt{2}$, we confine the value of $A$ to distinguish $\left\|\tilde n(t,\theta(t))\right\|_{L^2(\mathbb{R}^2)}$ and ~$\left\|\tilde n(t,0)\right\|_{L^2(\mathbb{R}^2)}$.
\qquad\hfill$\Box$\\
\end{remark}
\begin{remark}\label{3.15}
~~ By Proposition 3.12, one can obtain more delicate estimates on $\tilde{\textbf{v}}(s)$. Compared to the classical case, these estimates obtained are uniform.
\qquad\hfill$\Box$
\end{remark}
\begin{corollary}\label{3.16}
~~ For any $ s\in\left[0,\theta(t)\right]$, there holds\\
$$\left\|\tilde{\textbf{v}}(s)\right\|_{L^2(\mathbb{R}^2)}\leqslant A\sqrt{\frac{2}{\eta}}.$$
\end{corollary}
{\bf Proof.} From (2.7) it follows that
$$
\left.
\begin{array}{ll}
&\displaystyle \int_{\mathbb{R}^2}\left(\left|\nabla\tilde E_1\right|^2+\left|\nabla\tilde E_2\right|^2\right)dx-\frac{1}{2}\int_{\mathbb{R}^2}\left(\left|\tilde E_1\right|^2+\left|\tilde E_2\right|^2\right)^2dx
  \\[0.4cm]
 &\displaystyle\qquad-\frac{\eta}{2}\int_{\mathbb{R}^2}\left|\overline{\tilde E_1}\tilde E_2-\tilde E_1\overline{\tilde E_2}\right|^2dx
  \\[0.4cm]
  &\displaystyle\quad\geqslant\int_{\mathbb{R}^2}\left(\left|\nabla\tilde E_1\right|^2+\left|\nabla\tilde E_2\right|^2\right)dx-\frac{ 1+\eta }{2}\int_{\mathbb{R}^2}\left(\left|\tilde E_1\right|^2+\left|\tilde E_2\right|^2\right)^2dx
  \\[0.4cm]
  &\displaystyle\quad\geqslant\int_{\mathbb{R}^2}\left(\left|\nabla\tilde E_1\right|^2+\left|\nabla\tilde E_2\right|^2\right)dx-\frac{ 1+\eta }{\|Q\|_{L^2(\mathbb{R}^2)}^2}
  \int_{\mathbb{R}^2}\left(|E_{10}|^2+|E_{20}|^2\right)dx
  \\[0.4cm]
  &\displaystyle\qquad\qquad\cdot\int_{\mathbb{R}^2}\left(\left|\nabla\tilde E_1\right|^2+\left|\nabla\tilde E_2\right|^2\right)dx
  \\[0.4cm]
   &\displaystyle\quad=\int_{\mathbb{R}^2}\left(\left|\nabla\tilde E_1\right|^2+\left|\nabla\tilde E_2\right|^2\right)dx\left(1-\frac{ 1+\eta }{\|Q\|_{L^2(\mathbb{R}^2)}^2}
   \int_{\mathbb{R}^2}\left(|E_{10}|^2+|E_{20}|^2\right)dx\right).
\end{array}
\right.
$$
\indent By (3.4) and (3.179), there holds
$$
\left.
\begin{array}{ll}
&\displaystyle \int_{\mathbb{R}^2}\left[\tilde n+\left(|\tilde E_1|^2+|\tilde E_2|^2\right)\right]^2dx+\|\tilde{ \textbf{v}}\|_{L^2(\mathbb{R}^2)}^2
  \\[0.4cm]
 &\displaystyle\quad =\frac{2H_0}{\lambda^2}-2\int_{\mathbb{R}^2}\left(\left|\nabla\tilde E_1\right|^2+\left|\nabla\tilde E_2\right|^2\right)dx+\int_{\mathbb{R}^2}\left(\left|\tilde E_1\right|^2+\left|\tilde E_2\right|^2\right)^2dx
  \\[0.4cm]
  &\displaystyle\qquad +\eta\int_{\mathbb{R}^2}\left|\overline{\tilde E_1}\tilde E_2-\tilde E_1\overline{\tilde E_2}\right|^2dx
  \\[0.4cm]
  &\displaystyle\quad \leqslant\frac{2H_0}{\lambda^2}-\int_{\mathbb{R}^2}\left(\left|\nabla\tilde E_1\right|^2+\left|\nabla\tilde E_2\right|^2\right)dx
  \\[0.4cm]
  &\displaystyle\qquad \qquad \qquad
  \cdot\left(2-\frac{2(1+\eta)}{\|Q\|_{L^2(\mathbb{R}^2)}^2}
  \int_{\mathbb{R}^2}\left(|E_{10}|^2+|E_{20}|^2\right)dx\right)
  \\[0.5cm]
  &\displaystyle\quad \leqslant\frac{2H_0}{\lambda^2}+\left(\frac{2(1+\eta)}{\|Q\|_{L^2(\mathbb{R}^2)}^2}
  \int_{\mathbb{R}^2}\left(|E_{10}|^2+|E_{20}|^2\right)dx-2\right)A^2
  \\[0.4cm]
  &\displaystyle\quad\leqslant\frac{2H_0}{\lambda^2}+\dfrac{2}{\eta}A^{2}.
 \end{array}
\right.
$$
\\
Note that $\lambda\to+\infty$ as $t\to T$, one then obtains\\
$$\left\|\tilde{\textbf{v}}(s)\right\|_{L^2(\mathbb{R}^2)}
\leqslant\sqrt{\frac{2}{\eta}}A.
 $$
 \\
This finishes the proof of Corollary 3.16.\qquad\hfill$\Box$\\
\begin{proposition}\label{3.17}
\quad There exists a constant $c>0$ such that
$$\liminf_{t\to T}\int_{0}^{\theta(t)}\left\|\tilde{\textbf{v}}(s)\right\|_{L^2(\mathbb{R}^2)}\,ds\geqslant c.
\eqno(3.200)$$
\end{proposition}
{\bf Proof.}\quad We argue it by contradiction. Assume that as $n\to+\infty$, there exists a sequence $t_n\to T$ such that
$$\int_{0}^{\theta(t_n)}\|\tilde{\textbf{v}}(s)\|_{L^2(\mathbb{R}^2)}\,ds\to 0.\eqno(3.201)$$
From (3.177) it follows that $\forall\psi(x)\in C_{0}^{\infty}(\mathbb{R}^2)$,
$$
\left.
\begin{array}{ll}
&\displaystyle\int_{\mathbb{R}^2}\tilde n(t_n,\theta(t_n))\psi dx-\int_{\mathbb{R}^2}\tilde n(t_n,0)\psi dx
\\[0.4cm]
 &\displaystyle\quad=\int_{0}^{\theta(t_n)}\int_{\mathbb{R}^2}\left(-
 \nabla\cdot\tilde{\textbf{v}}(s)\psi\right)\,dxds
 \\[0.4cm]
 &\displaystyle\quad=\int_{0}^{\theta(t_n)}\int_{\mathbb{R}^2}
 \tilde{\textbf{v}}(s)\cdot\nabla\psi\,dxds,
\end{array}
\right.\eqno(3.202)
$$
which yields
$$
\left.
\begin{array}{ll}
&\displaystyle\left|\int_{\mathbb{R}^2}\tilde n(t_n,\theta(t_n))\psi\,dx-\int_{\mathbb{R}^2}\tilde n(t_n,0)\psi\,dx\right| \\[0.4cm]
 &\displaystyle\quad\leqslant\left(\int_{0}^{\theta(t_n)}
 \|\tilde{\textbf{v}}(s)\|_{L^2(\mathbb{R}^2)}\,ds\right)\|\nabla\psi\|
 _{L^2(\mathbb{R}^2)}.
\end{array}
\right.\eqno(3.203)
$$
By (3.185) and (3.186), taking a subsequence still denoted by $t_n$ yields that there exist a sequence $x_n\in\mathbb{R}^2$, and $N'\in L^2(\mathbb{R}^2)$ such that
$$\tilde n\left(t_n,\theta(t_n),x-x_n\right)\rightharpoonup N'~~\mbox{in}~~L^2(\mathbb{R}^2),$$
and
$$\|N'\|_{L^2(\mathbb{R}^2)}\geqslant Ac_1.$$
Let $\psi_0(x)\in C_{0}^{\infty}(\mathbb{R}^2)$ and satisfy $\left(\displaystyle\int_{\mathbb{R}^2}\psi_{0}^2\,dx\right)^{\frac{1}{2}}=1$ and $\displaystyle\int_{\mathbb{R}^2}N'\psi_0\,dx\geqslant\frac{1}{2}
\left(\int_{\mathbb{R}^2}N'^2\,dx\right)
^{\frac{1}{2}}$. In view of the assumptions (3.201) and (3.202), we have as $n\to+\infty$,
$$
\left.
\begin{array}{ll}
&\displaystyle\left|\int_{\mathbb{R}^2}\tilde n\left(t_n,\theta(t_n),x\right)\psi_0\left(x+x_n\right)dx-\int_{\mathbb{R}^2}\tilde n\left(t_n,0,x\right)\psi_0\left(x+x_n\right)dx\right|
  \\[0.4cm]
 &\displaystyle\quad\leqslant \left(\int_{0}^{\theta(t_n)}\left\|\tilde {\mathbf{v}}(s)\right\|_{L^2(\mathbb{R}^2)}\,ds\right)\left\|\nabla\psi_0
 \left(x+x_n\right)\right\|_{L^2(\mathbb{R}^2)}\to 0.
   \end{array}
\right.\eqno(3.204)
$$
On the other hand, one has
$$
\left.
\begin{array}{ll}
\displaystyle\int_{\mathbb{R}^2}\tilde n\left(t_n,\theta(t_n),x\right)\psi_0(x+x_n)\,dx=&\displaystyle\int_{\mathbb{R}^2}
\tilde n\left(t_n,\theta(t_n),x-x_n\right)\psi_0(x)\,dx
   \\[0.3cm]
 &\displaystyle \to\int_{\mathbb{R}^2}N'\psi_0dx\quad(n\to+\infty)
  \\[0.3cm]
  &\geqslant\displaystyle\frac{1}{2}
  \left(\int_{\mathbb{R}^2}N'^2dx\right)^{\frac{1}{2}}
  \geqslant\frac{Ac_1}{2}\geqslant2.
 \end{array}
\right.\eqno(3.205)
$$
However, by (2.5) one has
$$\left|\int_{\mathbb{R}^2}\tilde n(t_n,0)\psi_0\,dx\right|\leqslant\left(\int_{\mathbb{R}^2}\left|\tilde n(t_n,0)\right|^2dx\right)^{\frac{1}{2}}\left(\int_{\mathbb{R}^2}\psi_0^2\,
dx\right)^{\frac{1}{2}}\leqslant \sqrt{2}. \eqno(3.206)$$
It is obviously contradictory to (3.205). \\
\indent This finishes the proof of Proposition 3.17. \qquad\hfill$\Box$\\
\begin{remark}\label{3.18}
\quad (3.200) gives the estimate for $\theta(t)$. \\

In fact, in view of (3.183), if
$$\int_{0}^{\theta(t)}\|\tilde{\textbf{v}}(s)\|_{L^2(\mathbb{R}^2)}\,ds
\leqslant\int_{0}^{\theta(t)}Ac_2\,ds=Ac_2\theta(t),
\eqno(3.207)$$
then (3.200) implies that there exists a constant $c>0$ such that
$$\liminf_{t\to T}Ac_2\theta(t)\geqslant c,$$
that is,
$$ \liminf_{t\to T}\theta(t)\geq c.\eqno(3.208)$$
On the other hand, by Corollary 3.16 one has
$$ c\leqslant\int_0^{\theta(t)}\|\tilde {\mathbf{v}} \|_{L^2(\mathbb{R}^2)}ds\leqslant\theta(t)\sqrt{\frac{2}{\eta}}A,\eqno(3.209)$$
namely, there exists $\tilde c=\dfrac{c}{\sqrt{2}A}$ such that
$$ \theta(t)\geqslant\tilde c\sqrt{\eta}.\eqno(3.210)$$
Therefore (3.208) and (3.210) imply that there exists a constant
$\theta_{0}>0$ such that as $t\rightarrow T$, $\theta(t)\geq \theta_{0}$, and
$$\forall s\in[0,\theta),\quad\left\|\left(\tilde E_1,\tilde E_2,\tilde n,\tilde{\textbf{v}}\right)(s)\right\|_{H^1(\mathbb{R}^2)\times H^1(\mathbb{R}^2)\times L^2(\mathbb{R}^2) \times L^2(\mathbb{R}^2)}\leqslant A. \eqno(3.211)$$
This finishes the proof of Theorem 3.1.\qquad\hfill$\Box$
\end{remark}

 \section{Proof of the main results(Theorem 1.3)}

 In this section, based on these estimates obtained in Section 2 and Section 3, we prove the main result (Theorem 1.3) of the present paper.\\
 \indent
 We first show Conclusion (1) of Theorem 1.3.\\
 \indent By Theorem 3.1, as $t\rightarrow T$, there exist $\theta_{0}=\theta_{0}\left(\left\|  E_{10}\right\|_{L^2(\mathbb{R}^2)},\left\|  E_{20}\right\|_{L^2(\mathbb{R}^2)}
 \right)$ and $A>0$ such that
 $$\forall s\in[0,\theta_{0}),\quad\left\|\left(\tilde E_1,\tilde E_2,\tilde n,\tilde{\textbf{v}}\right)(s)\right\|_{H^1(\mathbb{R}^2) \times H^1(\mathbb{R}^2)\times L^2(\mathbb{R}^2)\times L^2(\mathbb{R}^2)}\leqslant A.\eqno(4.1)$$
 In view of (2.2) and (2.4), one gets
 $$\lambda(t)(T-t)\geq \theta_{0}. \eqno(4.2)$$
 This yields the estimate (1.6).\\
\indent In addition, it follows from (2.1) that
$$\left(\left\|\nabla\tilde E_1(0)\right\|^{2}_{L^2(\mathbb{R}^2)}+\left\|\nabla\tilde E_2(0)\right\|^{2}_{L^2(\mathbb{R}^2)}\right)^{\frac{1}{2}}\qquad\qquad $$
$$\qquad\qquad=\frac{1}{\lambda(t)}
\left(\left\|\nabla  E_1(t) \right\|^{2}_{L^2(\mathbb{R}^2)}+\left\|\nabla  E_2(t) \right\|^{2}_{L^2(\mathbb{R}^2)}\right)^{\frac{1}{2}},\eqno(4.3)$$
\\
$$\left\| \tilde n(0)\right\|_{L^2(\mathbb{R}^2)}=\frac{1}{\lambda(t)}\left\| n(t)\right\|_{L^2(\mathbb{R}^2)}.\qquad\qquad\qquad\eqno(4.4)$$
Going back to Proposition 3.3, (4.3) and (4.4) yields
$$\left(\left\|\nabla  E_1(t) \right\|^{2}_{L^2(\mathbb{R}^2)}+\left\|\nabla  E_2(t) \right\|^{2}_{L^2(\mathbb{R}^2)}\right)^{\frac{1}{2}}\geq c_{1}\lambda(t)\geq \dfrac{c_{1}\theta}{T-t}=\dfrac{\tilde{c}}{T-t},\eqno(4.5)$$
$$ \left\|n(t) \right\|_{L^2(\mathbb{R}^2)} \geq c_{1}\lambda(t)\geq \frac{c_{1}\theta}{T-t}=\frac{\tilde{c}}{T-t}.\eqno(4.6)$$
This completes the proof of (1) in Theorem 1.3.\\
\\
\indent Next we are going to prove conclusion (2). Firstly we claim the following:\\
\begin{proposition} \label{4.1}
For $\theta_{0}$ in Theorem 3.1, there exists a contant $c>0$ such that
$$\theta_{0} \geq \frac{c}{\left(\|E_{10}\|_{L^2(\mathbb{R}^2)}^2+\|E_{20}\|_{L^2(\mathbb{R}^2)}^2 - \frac{1}{\eta +1}\|Q\|_{L^2(\mathbb{R}^2)}^2\right)^{\frac{1}{2}}}.\eqno(4.7)$$
\end{proposition}
\begin{proof}
Due to the Hamiltonian given by (2.7), one gains
$$
\left.
\begin{array}{ll}
&\displaystyle\int_{\mathbb{R}^2}\left(\left|\nabla  \tilde{E}_1\right|^2 + \left|\nabla   \tilde{E}_2\right|^2\right)dx - \frac{1}{2} \int_{\mathbb{R}^2}\left(\left|  \tilde{E}_1\right|^2 + \left| \tilde{E}_2\right|^2\right)^2 dx
\\[0.4cm]
 &\displaystyle\qquad\quad+ \frac{1}{2}\int_{\mathbb{R}^2} \left(n + \left| \tilde{E}_1\right|^2 +\left |  \tilde{E}_2\right|^2\right)^2 + \frac{1}{2} \int_{\mathbb{R}^2} \left| \tilde{\mathbf{v}}\right|^2dx
 \\[0.4cm]
 &\qquad\displaystyle\leq  \frac{\mathcal{H}_0}{\lambda^2(t)} + 2 \eta \int_{\mathbb{R}^2}\left|  \tilde{E}_1\right|^2\left | \tilde{E}_2\right|^2 dx,
\end{array}
\right.\eqno(4.8)
$$
which yields

$$
\left.
\begin{array}{ll}
&\displaystyle\int_{\mathbb{R}^2}\left(\left|\nabla  \tilde{E}_1\right|^2 +\left |\nabla   \tilde{E}_2\right|^2\right)dx - \frac{1+\eta}{2} \int_{\mathbb{R}^2}\left(\left|  \tilde{E}_1\right|^2 +\left | \tilde{E}_2\right|^2\right)^2 dx
\\[0.4cm]
 &\displaystyle\qquad\quad+ \frac{1}{2}\int_{\mathbb{R}^2} \left(n +\left | \tilde{E}_1\right|^2 + \left|  \tilde{E}_2\right|^2\right)^2 + \frac{1}{2} \int_{\mathbb{R}^2} \left| \tilde{\mathbf{v}}\right|^2dx
 \\[0.4cm]
 &\qquad\displaystyle\leq  \frac{\mathcal{H}_0}{\lambda^2(t)}.
\end{array}
\right.\eqno(4.9)
$$

Direct calculation then gives
$$\int_{\mathbb{R}^2}\left(\left|\nabla  \tilde{E}_1\right|^2 + \left|\nabla   \tilde{E}_2\right|^2\right)dx - \frac{1+\eta}{2} \int_{\mathbb{R}^2}\left(\left|  \tilde{E}_1\right|^2 + \left| \tilde{E}_2\right|^2\right)^2 dx\leq  \frac{\mathcal{H}_0}{\lambda^2(t)},\eqno(4.10)$$
$$
\left.
\begin{array}{ll}
&\displaystyle\frac{1}{2}\int_{\mathbb{R}^2} \left(n +\left | \tilde{E}_1\right|^2 + \left|  \tilde{E}_2\right|^2\right)^2 + \frac{1}{2} \int_{\mathbb{R}^2}\left | \tilde{\mathbf{v}}\right|^2dx
\\[0.4cm]
&\displaystyle\qquad\leq \frac{\mathcal{H}_0}{\lambda^2(t)}- \int_{\mathbb{R}^2}\left(\left|\nabla  \tilde{E}_1\right|^2 + \left|\nabla   \tilde{E}_2\right|^2\right)dx
 \\[0.4cm]
&\displaystyle\qquad\quad
 + \frac{1+\eta}{2} \int_{\mathbb{R}^2}\left(\left|  \tilde{E}_1\right|^2 + \left| \tilde{E}_2\right|^2\right)^2 dx
\end{array}
\right.\eqno(4.11)
$$
Note that
$$
\left.
\begin{array}{ll}
&\displaystyle\left( 1 - \frac{(\eta+1) \left(\| E_{10}\|_{L^{2}(\mathbb{R}^2)}^2+\|E_{20}\|_{L^{2}(\mathbb{R}^2)}^2\right)}
{\|Q\|_{L^{2}(\mathbb{R}^2)}^2}\right)
\\[0.5cm]
&\displaystyle\qquad\qquad\qquad\qquad
\cdot\int_{\mathbb{R}^2}\left(\left|\nabla  \tilde{E}_1\right|^2 + \left|\nabla   \tilde{E}_2\right|^2\right)dx
\\[0.4cm]
&\displaystyle\qquad\leq\int_{\mathbb{R}^2}\left(\left|\nabla  \tilde{E}_1\right|^2 + \left|\nabla   \tilde{E}_2\right|^2\right)dx- \frac{1+\eta}{2} \int_{\mathbb{R}^2}\left(\left|  \tilde{E}_1\right|^2 + \left| \tilde{E}_2\right|^2\right)^2 dx,
  \end{array}
\right.\eqno(4.12)
$$
then one obtains
$$
\left.
\begin{array}{ll}
&\displaystyle\frac{1}{2}\int_{\mathbb{R}^2} \left(n + | \tilde{E}_1|^2 + |  \tilde{E}_2|^2\right)^2 + \frac{1}{2} \int_{\mathbb{R}^2} | \tilde{\mathbf{v}}|^2dx
\\[0.3cm]
&\displaystyle\qquad\leq \frac{\mathcal{H}_0}{\lambda^2(t)}
+\left(  \frac{(\eta+1) \left(\|E_{10}\|_{L^{2}(\mathbb{R}^2)}^2+\|E_{20}\|_{L^{2}(\mathbb{R}^2)}^2\right)}
{\|Q\|_{L^{2}(\mathbb{R}^2)}^2}-1\right)
\\[0.8cm]
&\qquad\qquad\qquad\qquad\displaystyle\cdot\int_{\mathbb{R}^2}\left(|\nabla  \tilde{E}_1|^2 + |\nabla   \tilde{E}_2|^2\right)dx.
\end{array}
\right.\eqno(4.13)
$$
Since $\lambda(t)\rightarrow \infty$ as $t\rightarrow T$, taking $t\rightarrow T$ one obtains
$$
\left.
\begin{array}{ll}
&\displaystyle\frac{1}{2} \int_{\mathbb{R}^2} | \tilde{\mathbf{v}}|^2dx
\\[0.5cm]
&\displaystyle\qquad\leq\frac{\eta+1}{\|Q\|_{L^{2}(\mathbb{R}^2)}^2}
\left(\|E_{10}\|_{L^{2}(\mathbb{R}^2)}^2+\|E_{20}\|_{L^{2}(\mathbb{R}^2)}^2
-\frac{1}{1+\eta}\|Q\|_{L^{2}(\mathbb{R}^2)}^2\right)
\\[0.5cm]
&\displaystyle\qquad\qquad\qquad\qquad\qquad\cdot\int_{\mathbb{R}^2}\left(|\nabla  \tilde{E}_1|^2 + |\nabla   \tilde{E}_2|^2\right)dx
\\[0.5cm]
&\displaystyle\qquad\leq\frac{A^{2}(\eta+1)}{\|Q\|_{L^{2}(\mathbb{R}^2)}^2}
\left(\|E_{10}\|_{L^{2}(\mathbb{R}^2)}^2+\|E_{20}\|_{L^{2}(\mathbb{R}^2)}^2
-\frac{1}{1+\eta}\|Q\|_{L^{2}(\mathbb{R}^2)}^2\right).
  \end{array}
\right.\eqno(4.14)
$$
Therefore, by Proposition 3.17, we obtain
$$
\left.
\begin{array}{ll}
&\displaystyle c  \leq \liminf_{t \rightarrow T} \int_0^{\theta(t)} \left\|  \tilde{\mathbf{v}}(s) \right\|_{L^{2}(\mathbb{R}^2)} ds
\\[0.4cm]
&\displaystyle\qquad\leq\frac{A \sqrt{2(\eta+1)}}{\|Q\|_{L^{2}(\mathbb{R}^2)}}
\left(\|E_{10}\|_{L^{2}(\mathbb{R}^2)}^2+\|E_{20}\|_{L^{2}(\mathbb{R}^2)}^2
-\frac{1}{1+\eta}\|Q\|_{L^{2}(\mathbb{R}^2)}^2\right)^{\frac{1}{2}}\theta_{0},
    \end{array}
\right.\eqno(4.15)
$$
and
$$\theta_{0}\geq c'\left(\|E_{10}\|_{L^{2}(\mathbb{R}^2)}^2+\|E_{20}\|_{L^{2}(\mathbb{R}^2)}^2
-\frac{1}{1+\eta}\|Q\|_{L^{2}(\mathbb{R}^2)}^2\right)^{-\frac{1}{2}}.\eqno(4.16)$$
This finishes the proof of Proposition 4.1.
\end{proof}
\indent Using Proposition 4.1 and taking $\theta_0=c'\left(\|E_{10}\|_{L^{2}(\mathbb{R}^2)}^2+\|E_{20}\|_{L^{2}(\mathbb{R}^2)}^2
-\frac{1}{1+\eta}\|Q\|_{L^{2}(\mathbb{R}^2)}^2\right)^{-\frac{1}{2}}$ in the proof of (1.7) and (1.8), we  achieve (1.9) and (1.10).\\
\indent This finishes the proof of Theorem 1.3.\hfill$\Box$

\addcontentsline{toc}{section}{Reference}

\end{document}